\documentclass{article}

\usepackage[utf8]{inputenc}   % LaTeX, comprends les accents !
\usepackage[T1]{fontenc}      % Police contenant les caractÃ¨res franÃ§ais
\usepackage{makecell}

\usepackage{comment}
\usepackage{natbib}

\usepackage[tbtags]{amsmath}
\usepackage{amsthm}
\allowdisplaybreaks
\usepackage{amssymb,mathrsfs}
\usepackage{amsfonts}
\usepackage{upgreek}
\usepackage{xspace}
\usepackage{geometry}
\usepackage{authblk}
\usepackage{graphicx}
\usepackage{color}
\usepackage[colorlinks = true, citecolor = blue]{hyperref}
\usepackage[ruled,vlined]{algorithm2e}
\usepackage{stmaryrd}

\usepackage[inline]{enumitem}
\usepackage{url}

\usepackage{lmodern} %% improve pdf rendering
\usepackage{mathtools}
\usepackage{tikz}
\usepackage{pgfplots}
\usepackage{xcolor}
\usepackage{bbm}
\usepackage{ifthen}
\usepackage{xargs}
\usepackage[textwidth=1.8cm]{todonotes}

\usepackage{aliascnt}
\usepackage{cleveref}
\usepackage{autonum}
\makeatletter
\newtheorem{theorem}{Theorem}
\crefname{theorem}{theorem}{Theorems}
\Crefname{Theorem}{Theorem}{Theorems}

\newtheorem*{lemma_nonumber*}{Lemma}

\newaliascnt{lemma}{theorem}
\newtheorem{lemma}[lemma]{Lemma}
\aliascntresetthe{lemma}
\crefname{lemma}{lemma}{lemmas}
\Crefname{Lemma}{Lemma}{Lemmas}

\newaliascnt{corollary}{theorem}
\newtheorem{corollary}[corollary]{Corollary}
\aliascntresetthe{corollary}
\crefname{corollary}{corollary}{corollaries}
\Crefname{Corollary}{Corollary}{Corollaries}

\newaliascnt{proposition}{theorem}
\newtheorem{proposition}[proposition]{Proposition}
\aliascntresetthe{proposition}
\crefname{proposition}{proposition}{propositions}
\Crefname{Proposition}{Proposition}{Propositions}

\newaliascnt{definition}{theorem}
\newtheorem{definition}[definition]{Definition}
\aliascntresetthe{definition}
\crefname{definition}{definition}{definitions}
\Crefname{Definition}{Definition}{Definitions}

\newaliascnt{remark}{theorem}

\aliascntresetthe{remark}
\crefname{remark}{remark}{remarks}
\Crefname{Remark}{Remark}{Remarks}

\crefname{example}{example}{examples}
\Crefname{Example}{Example}{Examples}

\crefname{figure}{figure}{figures}
\Crefname{Figure}{Figure}{Figures}

\usepackage{chngcntr}

\crefformat{assumption}{{\textbf{A}}#2#1#3}

\newtheorem{assumptionF}{\textbf{F}\hspace{-3pt}}
\crefformat{assumptionF}{{\textbf{F}}#2#1#3}

\Crefname{assumptionB}{\textbf{B}\hspace{-3pt}}{\textbf{B}\hspace{-3pt}}
\crefname{assumptionB}{\textbf{B}}{\textbf{B}}

\Crefname{assumptionC}{\textbf{C}\hspace{-3pt}}{\textbf{C}\hspace{-3pt}}
\crefname{assumptionC}{\textbf{C}}{\textbf{C}}

\newtheorem{assumptionH}{\textbf{H}\hspace{-3pt}}
\Crefname{assumptionH}{\textbf{H}\hspace{-3pt}}{\textbf{H}\hspace{-3pt}}
\crefname{assumptionH}{\textbf{H}}{\textbf{H}}

\Crefname{assumptionT}{\textbf{T}\hspace{-3pt}}{\textbf{T}\hspace{-3pt}}
\crefname{assumptionT}{\textbf{T}}{\textbf{T}}

\Crefname{assumptionT}{\textbf{T}\hspace{-3pt}}{\textbf{T}\hspace{-3pt}}
\crefname{assumptionT}{\textbf{T}}{\textbf{T}}

\Crefname{assumptionL}{\textbf{L}\hspace{-3pt}}{\textbf{L}\hspace{-3pt}}
\crefname{assumptionL}{\textbf{L}}{\textbf{L}}

\Crefname{assumptionQ}{\textbf{Q}\hspace{-3pt}}{\textbf{Q}\hspace{-3pt}}
\crefname{assumptionQ}{\textbf{Q}}{\textbf{Q}}

\Crefname{assumptionAR}{\textbf{AR}\hspace{-3pt}}{\textbf{AR}\hspace{-3pt}}
\crefname{assumptionAR}{\textbf{AR}}{\textbf{AR}}
\def\Ltt{\mathtt{L}}
\def\lsi{\mathtt{c_{LS}}}
\newcommand{\poly}[2]{\mathtt{Poly}(#1, #2)}
\def\sigmoid{\mathrm{sigm}}
\def\macroeps{\rho_\vareps}
\def\energygap{\emph{thermodynamic barrier}}
\def\Pens{\mathscr{P}}
\def\Adj{\mathrm{Adj}}
\def\rmL{\mathrm{L}}
\def\L{\mathcal{L}}
\def\C{\mathcal{C}}
\def\I{\mathcal{I}}
\def\J{\mathcal{J}}

\newcommand{\hessinvx}[2]{\det(\nabla^2_x u(#1,#2))^{-1/2}}
\newcommand{\jacinv}[1]{\mathrm{J}F(#1)^{-1}}
\newcommand{\jacinvx}[1]{\mathrm{J}F_x(#1)^{-1}}
\newcommand{\jacinvpower}[2]{\mathrm{J}F(#1)^{#2}}

\newcommand{\jac}[1]{\mathrm{J}F(#1)}

\newcommand{\tta}{\mathtt{A}}

\newcommand{\Capprox}{\tta}

\def\ctt{\mathtt{c}}

\newcommandx\ctun[1][1=T]{\Capprox_{#1,1}}

\newcommand{\rref}[1]{\tup{\Cref{#1}}}

\newcommandx{\expec}[2]{{\mathbb E}\left[#1 \middle \vert #2  \right]} %%%% esperance conditionnelle

\def\calH{\mathcal{H}}

\def\eps{\vareps}

\newcommand{\bvareps}{\bar{\vareps}}
\newcommand{\transference}{\mathbf{T}}

\newcommand{\rme}{\mathrm{e}}

\newcommand{\Lip}{\mathtt{L}}
\newcommand{\Lipset}{\mathrm{Lip}}
\newcommand{\Mtt}{\mathtt{M}}

\newcommand{\measfun}{\mathbb{F}}

\newcommandx{\norm}[2][1=]{\ifthenelse{\equal{#1}{}}{\left\Vert #2 \right\Vert}{\left\Vert #2 \right\Vert^{#1}}}
\newcommandx{\normLigne}[2][1=]{\ifthenelse{\equal{#1}{}}{\Vert #2 \Vert}{\Vert #2\Vert^{#1}}}

\def\bfc{\mathbf{c}}

\def\bfX{\mathbf{X}}

\def\bfM{\mathbf{M}}
\def\bfB{\mathbf{B}}

\def\msa{\mathsf{A}}

\def\msk{\mathsf{K}}

\def\msc{\mathsf{C}}
\def\mse{\mathsf{E}}

\def\msu{\mathsf{U}}
\def\msw{\mathsf{W}}
\def\msv{\mathsf{V}}

\def\msx{\mathsf{X}}
\def\msz{\mathsf{Z}}
\def\msy{\mathsf{Y}}

\newcommand{\mcb}[1]{\mathcal{B}(#1)}

\def\mcz{\mathcal{Z}}
\def\mcy{\mathcal{Y}}
\def\mcx{\mathcal{X}}

\def\rset{\mathbb{R}}

\def\nset{\mathbb{N}}
\def\nsets{\mathbb{N}^{\star}}

\def\rmd{\mathrm{d}}

\def\rme{\mathrm{e}}

\def\rmc{\mathrm{C}}
\def\rmC{\mathrm{C}}

\newcommand{\R}{\mathbb R}

\def\trace{\operatorname{Tr}}
\newcommandx{\functionspace}[2][1=+]{\mathbb{F}_{#1}(#2)}
\newcommand{\argmax}{\operatorname*{arg\,max}}
\newcommand{\argmin}{\operatorname*{arg\,min}}

\newcommandx{\VarDeux}[3][3=]{\operatorname{Var}^{#3}_{#1}\left\{#2 \right\}}

\newcommand{\1}{\mathbbm{1}}

\newcommand{\LeftEqNo}{\let\veqno\@@leqno}

\newcommand{\ceil}[1]{\left\lceil #1 \right\rceil}
\newcommand{\ceilLigne}[1]{\lceil #1 \rceil}

\newcommand{\N}{\ensuremath{\mathbb{N}}}

\newcommand{\PE}{\mathbb{E}}

\newcommand{\abs}[1]{\left\vert #1 \right\vert}
\newcommand{\absLigne}[1]{\vert #1 \vert}

\newcommandx{\Vnorm}[2][1=V]{\| #2 \|_{#1}}
\newcommandx{\VnormEq}[2][1=V]{\left\| #2 \right\|_{#1}}
\newcommand{\parenthese}[1]{\left(#1 \right)}

\newcommandx\probaMarkovTilde[2][2=]
{\ifthenelse{\equal{#2}{}}{{\widetilde{\mathbb{P}}_{#1}}}{\widetilde{\mathbb{P}}_{#1}\left[ #2\right]}}

\newcommand{\expeLigne}[1]{\PE [ #1 ]}

\newcommand{\bigO}{\ensuremath{\mathcal O}}

\def\ie{\textit{i.e.}}

\def\eqsp{\;}
\newcommand{\coint}[1]{\left[#1\right)}
\newcommand{\ocint}[1]{\left(#1\right]}
\newcommand{\ooint}[1]{\left(#1\right)}
\newcommand{\ccint}[1]{\left[#1\right]}

\newcommand{\ocintLigne}[1]{(#1]}
\newcommand{\oointLigne}[1]{(#1)}
\newcommand{\ccintLigne}[1]{[#1]}

\newcommandx{\weight}[2][2=n]{\omega_{#1,#2}^N}

\newcommand{\ballinfty}[2]{\operatorname{B}_\infty(#1,#2)}
\newcommand{\ball}[2]{\operatorname{B}(#1,#2)}

\newcommand{\cball}[2]{\bar{\operatorname{B}}(#1,#2)}
\newcommand{\cballinfty}[2]{\bar{\operatorname{B}}_\infty(#1,#2)}
\newcommand{\diameter}{\operatorname{diam}}

\newcommandx\sequence[3][2=,3=]
{\ifthenelse{\equal{#3}{}}{\ensuremath{\{ #1_{#2}\}}}{\ensuremath{\{ #1_{#2}, \eqsp #2 \in #3 \}}}}

\newcommandx\sequenceD[3][2=,3=]
{\ifthenelse{\equal{#3}{}}{\ensuremath{\{ #1_{#2}\}}}{\ensuremath{( #1)_{ #2 \in #3} }}}

\newcommandx{\sequencen}[2][2=n\in\N]{\ensuremath{\{ #1_n, \eqsp #2 \}}}
\newcommandx\sequenceDouble[4][3=,4=]
{\ifthenelse{\equal{#3}{}}{\ensuremath{\{ (#1_{#3},#2_{#3}) \}}}{\ensuremath{\{  (#1_{#3},#2_{#3}), \eqsp #3 \in #4 \}}}}
\newcommandx{\sequencenDouble}[3][3=n\in\N]{\ensuremath{\{ (#1_{n},#2_{n}), \eqsp #3 \}}}

\newcommand{\wrt}{w.r.t.}

\newcommand{\opnorm}[1]{{\left\vert\kern-0.25ex\left\vert\kern-0.25ex\left\vert #1
    \right\vert\kern-0.25ex\right\vert\kern-0.25ex\right\vert}}

\def\Id{\operatorname{Id}}

\newcommandx{\CPE}[3][1=]{{\mathbb E}_{#1}\left[#2 \middle \vert #3  \right]} %%%% esperance conditionnelle
\newcommandx{\CPELigne}[3][1=]{{\mathbb E}_{#1}[#2  \vert #3  ]} %%%% esperance conditionnelle
\newcommandx{\CPEsq}[3][1=]{{\mathbb{E}^{1/2}}_{#1}\left[#2 \middle \vert #3  \right]} %%%% esperance conditionnelle
\newcommandx{\CPVar}[3][1=]{\mathrm{Var}^{#3}_{#1}\left\{ #2 \right\}}
\newcommand{\CPP}[3][]
{\ifthenelse{\equal{#1}{}}{{\mathbb P}\left(\left. #2 \, \right| #3 \right)}{{\mathbb P}_{#1}\left(\left. #2 \, \right | #3 \right)}}

\def\measSet{\mathbb{M}}

\newcommandx{\osc}[2][1=]{\mathrm{osc}_{#1}(#2)}

\def\Id{\operatorname{Id}}

\def\v{v}

\def\y{y}

\def\bgamma{\bar{\gamma}}
\def\bdelta{\bar{\delta}}

\newcommand{\ensembleLigne}[2]{\{#1\,:\eqsp #2\}}

\def\rmD{\mathrm{D}}%%rmd déjà pris

\newcommand\coupling[2]{\Gamma(\mu,\nu)}

\newcommand{\complementary}{\mathrm{c}}

\def\diam{\mathrm{diam}}

\def\Lebesgue{\lambda}
\def\Leb{\lambda}

\def\interior{\mathrm{int}}

\def\vareps{\varepsilon}
\def\bvareps{\bar{\varepsilon}}

\newcommandx{\KL}[2]{\operatorname{KL}\left( #1 | #2 \right)}
\newcommandx{\KLLigne}[2]{\operatorname{KL}( #1 | #2 )}

\def\gaStep
\def\QKer{Q}

\def\distance{\mathbf{d}}
\newcommandx{\wasserstein}[3][1=\distance,3=]{\mathbf{W}_{#1}^{#3}\left(#2\right)}
\newcommandx{\wassersteinLigne}[3][1=\distance,3=]{\mathbf{W}_{#1}^{#3}(#2)}
\newcommandx{\wassersteinD}[1][1=\distance]{\mathbf{W}_{#1}}
\newcommandx{\wassersteinDLigne}[1][1=\distance]{\mathbf{W}_{#1}}

\def\Rcoupling{\mathrm{R}}

\def\Sker{\mathrm{S}}

\def\Kcoupling{\mathrm{K}}

\def\sigmaD{\sigma^2}

\newcommandx{\phibfs}[1][1=]{\pmb{\varphi}_{\sigmaD_{#1}}}

\newcommandx\sequenceg[3][2=,3=]
{\ifthenelse{\equal{#3}{}}{\ensuremath{( #1_{#2})}}{\ensuremath{( #1_{#2})_{ #2 \geq #3}}}}

\def\Kker{\Kcoupling}

\def\Rker{\Rcoupling}

\def\Qker{\mathrm{Q}}

\newcommandx{\distV}[1][1=\bfc]{\mathbf{W}_{#1}}
\newcommandx{\distVdeux}[1][1=W_2]{\mathbf{d}_{#1}}

\def\mtt{\mathtt{m}}

\newcommand{\tup}[1]{\textup{#1}}

\usepackage{color}
\definecolor{linkcol}{rgb}{0,0,0.4} 
\definecolor{citecol}{rgb}{0.5,0,0}
\definecolor{urlcol}{rgb}{0,0,0.75}

\begin{document} 

\author[1]{Valentin De Bortoli \footnote{corresponding author: valentin.debortoli@gmail.com}}
\author[2]{Agn\`{e}s Desolneux \footnote{agnes.desolneux@ens-paris-saclay.fr}}

\affil[1]{Department of Statistics, University of Oxford, UK}
\affil[2]{Centre Borelli, CNRS and ENS Paris-Saclay, France}
%\date{}

\title{On quantitative Laplace-type convergence results for some exponential
  probability measures, with two applications}

\maketitle

\begin{abstract}
  Laplace-type results characterize the limit of sequence of measures
  $(\pi_\vareps)_{\vareps >0}$ with density w.r.t the Lebesgue measure
  $(\rmd \pi_\vareps / \rmd \Leb)(x) \propto \exp[-U(x)/\vareps]$ when the
  temperature $\vareps>0$ converges to $0$. If a limiting distribution $\pi_0$
  exists, it concentrates on the minimizers of the potential $U$. Classical
  results require the invertibility of the Hessian of $U$ in order to establish
  such asymptotics. In this work, we study the particular case of norm-like
  potentials $U$ and establish quantitative bounds between $\pi_\vareps$ and
  $\pi_0$ w.r.t. the Wasserstein distance of order $1$ under an invertibility
  condition of a generalized Jacobian. One key element of our proof is the use
  of geometric measure theory tools such as the coarea formula. We apply our
  results to the study of maximum entropy models (microcanonical/macrocanonical
  distributions) and to the convergence of the iterates of the Stochastic Gradient
  Langevin Dynamics (SGLD) algorithm at low temperatures for non-convex minimization.
\end{abstract}
%\tableofcontents

\section{Introduction}

Asymptotic expansions of integrals are ubiquitous in probability theory with
applications in simulated annealing
\citep{gelfand1993metropolis,pelletier1998weak}, Bayesian inference
\citep{haughton1988choice,tierney1986accurate,rue2009approximate},
statistical 
physics \citep{ellis1982laplace} or chaos expansion
\citep{korshunov2015asymptotic}. In this paper, we are interested in families of
Gibbs probability measures $\pi_\vareps$ given by the density w.r.t. the Lebesgue
measure
$(\rmd \pi_\vareps / \rmd \Leb)(x) \propto \exp[-\normLigne{F(x)}^k/\vareps]$,
for any $x \in \rset^d$, where $k \in \nset$ and $F: \ \rset^d \to \rset^p$ is
some function with $F(0)=0$. In particular, we derive non asymptotic bounds
between $(\pi_\vareps)_{\vareps > 0}$ and $\pi_0$ when $\vareps \to
0$. 

Taking a step back, one way to deal with  integrals of the form
$\int_{\rset^d} \exp[-U(x)/\vareps] \rmd x$ where $U : \ \rset^d \to \rset_+$ is
to rely on Laplace-type techniques, see
\cite{bleistein1975asymptotic,olver1997asymptotics,de1981asymptotic,evgrafov2020asymptotic,erdelyi1956asymptotic,fedoryuk1989asymptotic,bender2013advanced,wong2001asymptotic}
for instance. Assuming that $U$ admits a unique minimizer $U(0)=0$ and under
additional regularity conditions, it can be shown that if the Hessian of $U$
evaluated at $0$ is invertible then
$\int_{\rset^d} \exp[-U(x)/\vareps] \rmd x \sim C_U \vareps^{d/2}$, with
explicit constant $C_U > 0$ (series expansions are also available under similar
conditions). This result on asymptotic integrals can be immediately used to
prove the convergence of probability measures with density w.r.t. the Lebesgue
measure $(\rmd \pi_\vareps / \rmd \Leb)(x) \propto \exp[-U(x)/\vareps]$, see
\cite{hwang1980} for instance. However the invertibility of the Hessian is a
restricting condition which is not satisfied in the setting of this paper where
$U(x) = \normLigne{F(x)}^k$ in the case $p \leq d$.

When the Hessian $\nabla^2 U(0)$ is not invertible, the asymptotic integral expansion
becomes \emph{degenerate} and is more difficult to analyze. However, several
approaches have been proposed in order to tackle this issue. For instance in
\cite{rytova2016multidimensional} a multidimensional version of the Watson lemma
is established under restrictive conditions on the potential $U$ near its
singularities. Another approach consists in integrating over the manifold of
minimizers and apply local change of variables \cite{korshunov2015asymptotic},
see also \cite{barbe2003approximation,breitung2006asymptotic}. In particular,
\cite{korshunov2015asymptotic} derives an asymptotic integral expansion similar
to the one obtained in the non-degenerate case, under some invertibility
condition of the minors of the Hessian. Similarly, \cite{hwang1980} uses the
tubular neighborhood theorem to obtain that $\pi_\vareps$ converges to a
limiting measure under the assumption that the minimizers of $U$ can be
partitioned into a collection of manifolds. Closer to the method introduced in
the present paper, \cite{combet2006integrales,arnold2012singularities} propose
to use the so-called \emph{Gelfand-Leray forms} to tackle the non-degeneracy
problem. By integrating over the level sets it can be shown that
\begin{equation}
  \label{eq:leray_formula}
  \textstyle{ \int_{\msv} \varphi(x) \exp[-U(x)/\vareps] \rmd x = \int_0^{+\infty} \exp[-t/\vareps](\int_{U^{-1}(t) \cap \msv} \varphi(x) \omega_U(x)) \rmd t \eqsp , }
\end{equation}
where $\msv$ is an open set, $\varphi \in \rmc^\infty(\rset^d, \rset)$ a test
function and $\omega_U$ the so-called Gelfand-Leray form associated with
$U$. Then by establishing regularity properties for the functional
$t \mapsto \int_{U^{-1}(t)} \varphi(x) \omega_U(x)$ one can recover asymptotic
integral expansion in the case where $U$ is analytic, see
\cite{combet2006integrales,arnold2012singularities}. This theory does not rely
on any invertibility condition on the Hessian of $U$ but the exponents appearing
in the asymptotic expansion are usually not available in closed form in general
since their derivation relies on resolution of singularities
\citep{hironaka1964resolution}.

In this paper, we consider a different approach which also does not rely on
invertibility conditions on the Hessian of $U$ at singularity points but allows
us to derive quantitative expansions. To do so, we restrict the set of functions
$U: \ \rset^d \to \rset_+$ to the set of norm-like functions, \ie \ $U$ is
norm-like if there exists $F: \ \rset^d \to \rset^p$ and $k \in \nset$ such that
for any $x \in \rset^d$, $U(x) = \normLigne{F(x)}^k$ (however our results can be
extended to general potentials $U$ under classical invertibility conditions on
the Hessian of $U$). Instead of invertibility conditions of the Hessian of $U$
we require invertibility conditions on the generalized Jacobian of $F$ to be
defined below. {The setting of norm-like potentials is of
  particular interest in machine learning as it arises in many applications such
  as Variational AutoEncoders (VAEs) or macrocanonical/microcanonical
  distributions, as we will discuss}. Our approach relies on tools from the
geometric measure theory and in particular the coarea formula, see
\cite{ambrosio2000functions} for instance. Indeed, using this formula we are
able to establish a similar result as \eqref{eq:leray_formula} where
the 
Gelfand-Leray form is replaced by a twisted Hausdorff measure $\nu$. By
establishing the Lipschitz regularity of the mapping
$t \mapsto \int_{U^{-1}(t) \cap \msv} \varphi(x) \rmd \nu(x)$ we are able to
provide quantitative bounds for the integral expansion. This expansion is the
key to our main result which establishes quantitative bounds w.r.t. the
Wasserstein distance of order $1$ between $\pi_\vareps$ and a limit probability
measure $\pi_0$ for $\vareps > 0$ small enough. Concurrently to this work,
\cite{bras2021convergence} establish similar quantitative results in the case
where the Hessian of $U$ is invertible.

One of our main motivation for this study is the application of Laplace-type
results to maximum entropy distributions. In particular, using our quantitative bounds,
we are able to provide a link between two possible maximum entropy distributions
commonly used in statistical physics and image processing
\citep{BrunaMallat,de2021maximum}. Given a reference measure $\mu$ and
constraints $F: \ \rset^d \to \rset^p$ it is possible to define the
\emph{microcanonical distribution}, which corresponds to the probability measure
$\pi_{\mathrm{mic}}$ with the minimum Kullback-Leibler divergence w.r.t. $\mu$
and such that $\pi_{\mathrm{mic}}$ is supported on the set
$\ensembleLigne{x \in \rset^d}{F(x)=0}$. Similarly, one can define the
\emph{macrocanonical distribution} with level $\vareps > 0$ denoted
$\pi_{\mathrm{mac}}^\vareps$ such that $\pi_{\mathrm{mac}}^\vareps$ minimizes
the Kullback-Leibler divergence w.r.t. $\mu$ and satisfies the integrability
condition
$\int_{\rset^d} \normLigne{F(x)}^2 \rmd \pi_{\mathrm{mac}}^\vareps(x) =
\vareps$. Using results from information geometry \citep{csiszar1975divergence}
and under mild regularity conditions, $\pi_{\mathrm{mac}}^\vareps$ can be written as a Gibbs
measure and its limiting behavior when $\vareps \to 0$ can be investigated in
the context of asymptotic expansions of integrals. In this work, we show that
the limit of the macrocanonical distributions when $\vareps \to 0$ is a twisted
microcanonical distribution and propose an algorithm to asymptotically recover the
original microcanonical distribution. {Another motivation for the
  study of the limiting behavior of distributions with norm-like potentials comes
  from the study of the posterior of the latent variables in VAE. In particular,
  we show that under assumptions on the decoding neural network, the posterior
  distribution of a VAE concentrates with explicit rates.}

Finally, we also consider the non-asymptotic study of
Stochastic Gradient Langevin Dynamics (SGLD) \citep{welling2011bayesian}, a
popular algorithm used to approximate the minimizers of non-convex functions in a
machine learning setting. Given a potential (or population risk)
$U: \ \rset^d \to \rset_+$ which can be written for any $x \in \rset^d$ as
$U(x) = \int_{\msz} u(x, z) \rmd \mu(z)$, we assume that we have access to an
empirical version of this risk for any $n \in \nset$ and
$z^{1:n} = \{z_i\}_{i=1}^n \in \msz^n$ given by
$U_n(x,z^{1:n}) = (1/n) \sum_{i=1}^n u(x,z_i)$ as well as an unbiased estimator
of its gradient denoted $g$. Conditionally to the samples $z^{1:n}$, SGLD
corresponds to the recursion associated with an unadjusted Langevin dynamics
\citep{roberts1996exponential,durmus2017nonasymptotic,dalalyan2017theoretical}
with target $U_n(\cdot, z^{1:n})/\vareps$ for a small value of $\vareps > 0$ and
where the gradient of $U_n(\cdot, z^{1:n})$ is replaced by its unbiased
estimator $g$. Since the invariant measure of the underlying Langevin diffusion
is given by $\pi_\vareps$ with density w.r.t. the Lebesgue measure
$(\rmd \pi_\vareps / \rmd \Leb)(x) \propto \exp[-U_n(x, z^{1:n})/\vareps]$ we
obtain that for small values of the parameter $\vareps > 0$, the samples
$(X_k)_{k \in \nset}$ are approximately concentrated around the minimizers of
$U_n(\cdot, z^{1:n})$.  In existing works on SGLD
\citep{raginsky2017non,chau2019stochastic,gao2018global,erdogdu2018global,nguyen2019non,xu2017global}
all quantitative theoretical guarantees are obtained w.r.t.  the sequence of
$(\expeLigne{U(X_k)})_{k \in \nset}$. On the contrary, we characterize the
limiting distribution of SGLD and establish non-asymptotic quantitative bounds
between SGLD and this limit. Finally, we show that the support of the limiting
distribution is included into (but not necessarily equal to) the set of
minimizers of $U$ when $n$ is large. This last result relies on the notion of
\energygap \ which we introduce in the context of integral
expansion. To our
knowledge this is the first time that the importance of such a barrier to derive
parametric Laplace-type results is highlighted (as opposed to the \emph{kinetic
  barrier} which is a well-known quantity in simulated annealing
\citep{hajek1988cooling}).

To summarize, our main contributions are three-fold:
\begin{enumerate}[wide, labelwidth=!, labelindent=0pt, label=(\roman*)]
\item We establish quantitative bounds between $\pi_\vareps$ and $\pi_0$ under
  the assumption that the potential $U$ is norm-like and additional regularity
  conditions. In particular we provide an upper-bound between $\pi_\vareps$ and
  $\pi_0$ w.r.t. the Wasserstein distance of order $1$. We emphasize that our
  results also hold under the more classical invertibility condition on the
  Hessian of $U$. {We show a first application of our results with a study
    of the concentration of posteriors in VAEs.}
\item We apply our results to show that under mild conditions the limit of a
  natural class of macrocanonical distributions is \emph{not} the original
  microcanonical distribution. However, we prove that a twisted sequence of
  macrocanonical distributions converges to the microcanonical distribution. This
  observation allows us to construct a Langevin-based algorithm in order to
  sample from this distribution. We illustrate our method in low-dimensional
  settings.
\item Finally, we apply our theory to study the behavior of the iterates of SGLD
  in the context of non-convex optimization. In particular, we characterize the
  limiting distribution of SGLD at low temperature and show that it concentrates
  on the minimizers of the population loss for a large number of samples using
  the notion of \energygap . Note that in this study, we no longer assume that
  $U$ is norm-like but instead rely on invertibility conditions on the Hessian
  of $U$.
\end{enumerate}

The rest of the paper is organized as follows. In \Cref{sec:main-results} we
present our main results, \ie \ quantitative bounds between $\pi_\vareps$ and
the limiting distribution $\pi_0$. In \Cref{sec:applications}, we present our
two main applications: the links between the macrocanonical and microcanonical
maximum entropy distributions in \Cref{sec:from-macr-micr}, and a study of the
convergence of SGLD for non-convex minimization in
\Cref{sec:non_cvx_mon}. The proofs of our results are gathered in
\Cref{sec:proof}.

%%% Local Variables:
%%% mode: latex
%%% TeX-master: "main"
%%% End:

\section*{Notation}
\label{sec:notation}

Let $d \in \nsets$. We denote $\{e_i\}_{i=1}^d$ the canonical basis of
$\rset^d$.  Let $\langle \cdot, \cdot \rangle$ be the Euclidean scalar product
over $\rset^d$, and $\normLigne{\cdot}$ be the corresponding norm. We denote
$\ball{x}{r}$ the ball with center $x \in \rset^d$ and radius $r > 0$ w.r.t. the
norm $\normLigne{\cdot}$. Similarly, we denote $\ballinfty{x}{r}$ the ball with
center $x \in \rset^d$ and radius $r > 0$ w.r.t. the norm
$\normLigne{\cdot}_\infty$ given for any $x = (x_1, \dots, x_d) \in \rset^d$ by
$\normLigne{x}_\infty = \sup_{i \in \{1, \dots, d\}} \abs{x_i}$.

Let $\msa \subset \rset^d$, we denote
$\diameter(\msa) = \sup_{x,y \in \msa} \normLigne{x -y}$ its diameter and
$\msa^\complementary$ its complementary set.  Let $\mcb{\rset^d}$ denote the
Borel $\sigma$-field of $\rset^d$. Let $\msu$ be an open set of $\rset^d$,
$n \in \nsets$ and let $\rmC^n(\msu, \rset^p)$ be the set of the
$n$-differentiable $\rset^p$-valued functions defined over $\msu$. If $p=1$ we
simply denote $\rmC^n(\msu)$.  Let $f \in \rmC^1(\msu)$, we denote by $\nabla f$ its
gradient. Furthermore, if $f \in \rmC^2(\msu)$ we denote $\nabla^2f$ its Hessian
and $\Delta f$ its Laplacian. By convention we denote $\nabla^0 f = f$.  We also
denote $\rmC(\msu,\rset^p)$ the set of continuous functions defined over $\msu$.
Let $f: \msa \to \rset^p$ with $p \in \nsets$ and $\msa \subset \rset^d$. The
function $f$ is said to be $\Ltt$-Lipschitz with $\Ltt \geq 0$ if for any
$x,y \in \msa$, $\norm{f(x) - f(y)} \leq \Ltt \norm{x-y}$. Let
$F: \ \rset^d \to \rset^p$ such that $F$ is differentiable at $x \in \rset^d$
and denote
$\rmD F(x) = (\partial_j F_i(x))_{i \in \{1, \dots, p\}, j \in \{1, \dots, d\}}$
the Jacobian matrix of $F$ which is a $p \times d$ matrix. We define the
generalized Jacobian $\mathrm{J}F: \ \rset^d \to \rset_+$ given for any
$x \in \rset^d$ by
\begin{equation}
  \mathrm{J}F(x) = \left\lbrace \begin{matrix}
      \textstyle{\det(\rmD F(x) \rmD F(x)^\top)^{1/2} \eqsp , \qquad  \text{if $d \geq p$,}} \\
      \textstyle{\det(\rmD F(x)^\top \rmD F(x))^{1/2} \eqsp , \qquad \text{if $d \leq p$.}}
    \end{matrix} \right.
\end{equation}
Finally, if $F \in \rmc^\ell(\rset^d, \rset^p)$ with $\ell \in \nset$, we define
$\rmD^j F$ for any $j \in \{1, \dots, \ell\}$, recursively by
$\rmD^{j+1} F = \rmD(\rmD^j F)$. Note that for any $j \in \{0, \dots, \ell\}$,
$\rmD^j F$ can be represented as a $p \times d \times \dots \times d$ tensor
(where $d$ appears $j$ times) with symmetric last $j$ coordinates. Hence, we can
define $\rmD^{j+1} F^\top$ the tensor where the first and last dimension have
been exchanged, which is a $d \times \dots d \times p$ tensor.  If $p=1$ then,
we write $\nabla^j F = \rmD^j F$ for any $j \in \{1, \dots, \ell\}$.

Let $(\msx,\mcx)$ be a measurable space.  We denote by $\measfun(\msx)$ the set
of the $\mcx/\mcb{\rset}$-measurable real functions over $\msx$.  Let
$\measSet(\mcx)$ be the set of finite signed measures over $\mcx$ and let
$\mu \in \measSet(\mcx)$. For $f \in \measfun(\msx)$ a $\mu$-integrable function
we denote 
\begin{equation}
 \textstyle{ \mu[f] = \int_{\msx} f(x) \rmd \mu(x) \eqsp , }
\end{equation}
 the integral of $f$ \wrt \ to $\mu$ when it is
well-defined. We also define $\Pens(\msx, \mcx)$ the set
of probability measures over $\mcx$ and when there is no ambiguity on the
sigma-field we simply denote it by $\Pens(\msx)$. We denote by $\Leb$ the Lebesgue
measure on $\rset^d$. Let $(\msx, \mcx)$ and $(\msy, \mcy)$ be two measurable
spaces. A Markov kernel $\Kker$ is a mapping
$\Kker: \ \msx \times \mcy \to \ccint{0,1}$ such that for any $x \in \msx$,
$\Kker(x, \cdot) \in \Pens(\msy, \mcy)$ and for any $\msa \in \mcy$,
$\Kker(\cdot, \msa)$ is measurable.

Let $\mu, \nu \in \Pens(\msx, \mcx)$.  A probability measure
$\zeta \in \Pens(\msx \times \msx, \mcx \otimes \mcx)$ is said to be a
transference plan between $\mu$ and $\nu$ if for any $\msa \in \mcx$,
$\zeta(\msa \times \mcx) = \mu(\msa)$ and $\zeta(\mcx \times \msa) =
\nu(\msa)$. We denote by $\transference(\mu, \nu)$ the set of all transference
plans between $\mu$ and $\nu$.  We define the Wasserstein metric/distance of order $\ell \geq 1$
$\wassersteinD[\ell](\mu, \nu)$ between $\mu$ and $\nu$ by
\begin{equation}
  \label{eq:def_distance_wasser}
  \textstyle{
    \wassersteinD[\ell](\mu, \nu)^\ell = \inf_{\zeta \in \transference(\mu,\nu)} \int_{\msx^2}  \normLigne{x-y}^\ell \rmd \zeta (x,y) \eqsp.
    }
\end{equation}
We denote $\calH^r$ the Hausdorff measure of order $r > 0$, given for any
$\msa \subset \rset^d$ by
$\calH^r(\msa) = \lim_{\delta \to 0} \calH^{r,\delta}(\msa)$, where for any
$\delta > 0$ we have
\begin{equation}
  \textstyle{
    \calH^{r,\delta}(\msa) = \inf \ensembleLigne{(\alpha_r/2^r) \sum_{i \in \nset} \diam(\msu_i)^r}{\msa \subset \cup_{i \in \nset} \msu_i \eqsp , \ \diam(\msu_i) \leq \delta}\eqsp ,
    }
\end{equation}
where $\alpha_r=\pi^{r/2}/\Gamma(\frac{r}{2}+1)$ with $\Gamma$ the usual Gamma function.
Basic facts on the Hausdorff measure are gathered in
\Cref{sec:basics-geom-meas}. Finally, we denote by $\poly{k}{\msa}$ the set of
polynomials of $k \in \nset$ variables with coefficients in
$\msa \subset \rset$.

%%% Local Variables:
%%% mode: latex
%%% TeX-master: "main"
%%% End:

\section{Limit theorems}
\label{sec:main-results}

In this section, we state our main theorem and draw links with previous
approximation results for probability integrals. Let $k \in \nset$,
$\bvareps > 0$, $F:\ \rset^d \to \rset^p$ with $d, p \in \nsets$ such that for
any $\vareps \in \oointLigne{0, \bvareps}$,
$\int_{\rset^d} \exp[-\norm{F(x)}^k/\vareps] \rmd x < +\infty$.  For any
$\vareps \in \oointLigne{0, \bvareps}$, we define
$\pi_\vareps \in \Pens(\rset^d)$ such that for any $\msa \in \mcb{\rset^d}$
\begin{equation}
  \label{eq:pi_eps_def}
  \textstyle{
    \pi_\vareps(\msa) = \left. \int_{\msa} \exp[-\norm{F(x)}^k/\vareps] \rmd x \middle/ \int_{\rset^d} \exp[-\norm{F(x)}^k/\vareps] \rmd x \right. \eqsp .
    }
\end{equation}
Our goal is to study the behavior of $\pi_\vareps$ when $\vareps \to 0$. In
particular, we identify a limit $\pi_0 \in \Pens(\rset^d)$ and derive
non-asymptotic convergence bounds.  For any $\msa \in \mcb{\rset^d}$ we let
(when it is well-defined)
\begin{equation}
  \label{eq:pi_0_def_both}
  \textstyle{
    \pi_0(\msa) = \left. \int_{F^{-1}(0) \cap \msa} \jacinv{x} \rmd \calH^{d-\min(d,p)}(x) \middle/ \int_{F^{-1}(0)} \jacinv{x} \rmd \calH^{d-\min(d,p)}(x) \right .\eqsp .
    }
\end{equation}
In what follows we consider the following assumption on $F$ which implies that
$\pi_\vareps$ is well-defined for any $\vareps \geq 0$.
\begin{assumptionH}
  \label{assum:F}
  $F \in \rmc^\infty(\rset^d, \rset^p)$ and $F(0)=0$. For any $x \in F^{-1}(0)$,
  $\jac{x} \neq 0$. There exist $\mtt, \upalpha > 0$ and $R \geq 0$ such that for any
  $x \in \rset^d$ with $\norm{x} \geq R$, $\norm{F(x)} \geq \mtt \norm{x}^\upalpha$.
\end{assumptionH}

A few remarks are in order. First, the assumption $F(0)=0$ is only technical and
can be replaced by $F^{-1}(0) \neq \emptyset$. Similarly, the smoothness
assumption $F \in \rmc^\infty(\rset^d, \rset^p)$ allows us to avoid some
technicalities in the proofs but can be relaxed. Second, we assume that
$x \mapsto \normLigne{F(x)}$ grows at least polynomially when
$\norm{x} \to +\infty$. This condition can also be relaxed to handle
sub-polynomial growth at infinity. However, changing the rate of growth might
affect the quantitative convergence properties of $(\pi_\vareps)_{\vareps > 0}$
towards $\pi_0$. A study of sub-polynomial growth is left for future
work. Finally the assumption that for any $x \in F^{-1}(0)$, $\jac{x} \neq 0$ is
necessary in order to define $\pi_0$.

\begin{theorem}
  \label{thm:big_theo}
  Assume \rref{assum:F}. Then for any $\vareps \geq 0$,
  $\pi_\vareps$ is well-defined. Let $\msu \subset \rset^d$ be open, bounded and 
  such that $F^{-1}(0) \subset \msu$. Let $\varphi \in \rmc(\rset^d, \rset)$ and
  $C_\varphi \geq 0$ such that for any $x \in \rset^d$
\begin{equation}
  \label{eq:cond_varphi}
  \textstyle{
    \abs{\varphi(x)} \leq C_\varphi \exp[C_\varphi \norm{x}^{\upalpha k}] \eqsp ,
    }
  \end{equation}
  with $\upalpha > 0$ defined in \rref{assum:F} and $k \in \nsets$ in
  \eqref{eq:pi_eps_def}.  Then
  $\lim_{\vareps \to 0} \abs{\pi_\vareps[\varphi] - \pi_0[\varphi]} = 0$. In
  addition, if $\varphi$ is $M_{1, \varphi}$-Lipschitz on $\msu$ with
  $M_{1, \varphi} \geq0$, then there exist
  $A \in \rmc(\rset_+, \rset_+), \bvareps \in \rmc(\rset_+, \rset_+)$ such that
  for any $\vareps \in \oointLigne{0, \bvareps(C_\varphi)}$ we have
  \begin{equation}
    \abs{\pi_{\vareps}[\varphi] - \pi_0[\varphi]} \leq A(C_\varphi) (1 + M_{0,\varphi} + M_{1,\varphi}) \vareps^{1/k} \eqsp ,
  \end{equation}
  with
  $M_{0, \varphi} = \sup \ensembleLigne{\absLigne{\varphi(x)}}{x \in \msu}$,
  $A, \bvareps$ functions that do not depend on $\varphi$, $A$ is non-decreasing and
  $\bvareps$ is non-increasing.
\end{theorem}

\begin{proof}
  We provide a sketch of the proof. The whole proof is postponed to
  \Cref{sec:proof_thm_macro}. We first start by showing that we can restrict our
  study to versions of $\pi_\vareps$ and $\pi_0$ with truncated support.
  {This is done by studying the decay of
    $\pi_\vareps(\msk^\complementary)$ where $\msk$ is some compact set.} Then
  our study differs depending on if $d \geq p$ or $d \leq p$. If $d \geq p$ then
  we use tools from geometric measure theory, and in particular the coarea
  formula in combination with the Lipschitz property of applications of the form
  $t \mapsto \int_{F^{-1}(t)} \Phi(x) \rmd \calH^{d-p}(x)$ for regular mappings
  $\Phi$. {More precisely, we apply the coarea formula with the
    mapping $F$. Doing so, we have 
    $\int_\msu \exp[-\normLigne{F(x)}^k/\vareps] \varphi(x) \rmd x = \int_\msv
    \exp[-\normLigne{t}^k/\vareps] \int_{F^{-1}(t)} \varphi(x) \jacinv{x} \rmd
    \calH^{d-p}(x) \rmd t$, where $\msu, \msv$ are two open sets. Hence,
    controlling the regularity of
    $t \mapsto \int_{F^{-1}(t)} \varphi(x) \jacinv{x} \rmd \calH^{d-p}(x)$ is the key
    to control the convergence rate.} In the case $d \leq p$ we adapt arguments
  from Morse theory and Laplace theory to derive quantitative bounds on
  $\abs{\pi_{\vareps}[\varphi] - \pi_0[\varphi]}$. These analyses are first conducted
  with smooth test functions $\varphi \in \rmc^1(\rset^d, \rset)$ and then we relax
  this hypothesis using smoothing arguments.
\end{proof}

We highlight a few key points from this theorem and draw links with the existing
literature.

\begin{enumerate}[wide, labelwidth=!, labelindent=0pt, label=(\alph*)]
\item Our result is related to the Laplace-type convergence results of
  \cite{wong2001asymptotic, korshunov2015asymptotic,
    barbe2003approximation, breitung2006asymptotic, 
    bleistein1975asymptotic, olver1997asymptotics, de1981asymptotic,
    evgrafov2020asymptotic, erdelyi1956asymptotic, fedoryuk1989asymptotic,
    bender2013advanced}. In these works, the authors study integrals of the form
  $\int_{\rset^d} \varphi(x) \exp[-U(x)/\vareps] \rmd x$ in the limiting case
  where $\vareps \to 0$ under regularity assumptions on $U$ and $\varphi$, and
  the non-degeneracy condition that for any $x \in \argmin_{\rset^d} U$,
  $\nabla^2 U(x)$ is invertible. We say that $U : \ \rset^d \to \rset_+$ is
  norm-like with exponent $2$ if there exist $p \in \nsets$ and
  $F: \ \rset^d \to \rset^p$ such that $U(x) = \norm{F(x)}^2$. In this case, and
  under additional regularity assumptions, we have for any $x \in \rset^d$
  \begin{equation}
    \nabla^2 U(x) = 2 \rmD^2F(x)^\top F(x) + 2 \rmD F(x)^\top \rmD F(x) \eqsp . 
  \end{equation}
  In particular for any $x \in F^{-1}(0)$, $\nabla^2 U(x)$ is invertible if and
  only if $\jac{x} \neq 0$ which is precisely the non-degeneracy condition
  imposed in \rref{assum:F} (note that this directly implies that $d \leq
  p$). For any $U \in \rmc^{3}(\rset^d, \rset_+)$ with Lipschitz third order
  derivatives, there exist $p \in \nsets$ and $F \in \rmc^1(\rset^d, \rset^p)$
  with Lipschitz derivatives such that for any $x \in \rset^d$,
  $U(x) = \norm{F(x)}^2$, see
  \cite{bony2006nonnegative,fefferman1978positivity}. Hence, our restriction to
  functions $U : \ \rset^d \to \rset_+$ of the form $U = \norm{F}^k$ for some
  $F: \ \rset^d \to \rset^p$ with $p,k \in \nsets$ is not too constraining under
  additional regularity assumptions.
  \item We emphasize that our results can be
  extended to a general potential $U \in \rmc^\infty(\rset^d, \rset_+)$ if the
  Hessian of $U$ is invertible on the set $U^{-1}(0)$, thus recovering the usual
  setting of Laplace-type convergence results.
\item As highlighted previously, we aim at providing quantitative bounds for
  $\abs{\pi_{\vareps}[\varphi] - \pi_0[\varphi]}$. In the case where $d \leq p$
  we adapt classical arguments from Morse and Laplace theory. Our main
  contribution is the use of geometric measure theory to cover the case
  $d \geq p$ which cannot be treated with classical Laplace arguments. The case
  $d \geq p$ is of particular interest in machine learning applications
  requiring to sample from distributions of the form
  $(\rmd \pi / \rmd \Leb)(x) \propto \exp[-\norm{F(x)}^2]$, where
  $F: \ \rset^d \to \rset^p$ is an encoder neural network with $d \geq p$.  In
  \Cref{sec:from-macr-micr} we investigate this situation in details in the case
  of generative modeling with macrocanonical and microcanonical distributions.
\item {We notice that choosing $k >1$ hinders the convergence
    towards $\pi_0[\varphi]$. The intuition behind this result is that for large
    values of $k \in \nset$ it is harder to distinguish global minimizers from
    their neighborhood (for example in the case where $F(x) = x$ and $d=p=1$, we
    have that $\normLigne{F(0.1)}^1 = 0.1$ whereas
    $\normLigne{F(0.1)}^{3} = 10^{-3}$).}
\item {Finally, in \cite{raginsky2017non} we have that
    $\pi_\vareps[U] - \min_{\rset^d} U \leq C \vareps$ under mild
    assumptions. In the case where $U = \normLigne{F}^2$ our results suggest
    that $\pi_\vareps[U] - \min_{\rset^d} U \leq C \vareps^{1/2}$ which appears
    to be suboptimal. However, our result is more general as it is valid for
    every locally Lipschitz function. In order to recover the rate
    $\bigO(\vareps)$ one must modify the proof of \Cref{thm:big_theo} to
    leverage the fact that $\nabla U(x) = 0$ for any $x \in \rset^d$ which is a
    minimizer of $U$. More precisely, under the assumptions of
    \Cref{thm:big_theo} and if $\varphi \in \rmc^\infty(\rset^d, \rset)$ with
    $\nabla^\ell \varphi(x)= 0$ for any $\ell \in \{0, \dots, p\}$ with
    $p \in \nset$ and $x \in \rset^d$ which is a minimizer of $U$, we get that
    $\abs{\pi_{\vareps}[\varphi] - \pi_0[\varphi]} \leq C_\varphi \vareps^{(1 +
      p)/2}$.}
\end{enumerate}

As a by-product of \Cref{thm:big_theo} we obtain the following corollary which
establishes quantitative bounds for the Wasserstein distance of order $1$ 
between $\pi_\vareps$ and $\pi_0$.

\begin{corollary}
  \label{coro:wass_dist}
  Assume \rref{assum:F}. Then for any $\vareps \geq 0$,
  $\pi_\vareps$ is well-defined. In addition, there exist $A_1 \geq 0$ and
  $\bvareps > 0$ such that for any $\vareps \in \coint{0, \bvareps}$ we have
  \begin{equation}
    \wassersteinD[1](\pi_\vareps, \pi_0) \leq A_1 \vareps^{1/k} \eqsp ,
  \end{equation}
  where we recall that $k \in \nsets$ is defined in \eqref{eq:pi_eps_def}.
\end{corollary}

\begin{proof}
  Denote
  $\Lipset = \ensembleLigne{\varphi: \ \rset^d \to \rset}{\abs{\varphi(x) -
      \varphi(y)} \leq \norm{x-y} \eqsp , \ \text{for any } x,y \in
    \rset^d}$. In addition, denote
  $\Lipset_0 = \ensembleLigne{\varphi: \ \rset^d \to \rset}{\abs{\varphi(x) -
      \varphi(y)} \leq \norm{x-y} \eqsp , \ \text{for any } x,y \in \rset^d, \
    \varphi(0) = 0}$. First, using \cite[Theorem 5.10]{villani2009optimal} we
  have for any $\vareps > 0$.
  \begin{align}
    \label{eq:def_wass}
    \wassersteinD[1](\pi_\vareps, \pi_0) = \sup \ensembleLigne{\pi_\vareps[\varphi] - \pi_0[\varphi]}{\varphi \in \Lipset} = \sup \ensembleLigne{\abs{\pi_\vareps[\varphi] - \pi_0[\varphi]}}{\varphi \in \Lipset_0} \eqsp . 
  \end{align}
  Let $r = \upalpha k$ and $\beta_r = \ceil{1/r}$. For any
  $\varphi \in \Lipset_0$ and $x \in \rset^d$ with $\norm{x} \geq 1$ we have
  \begin{equation}
    \beta_r \exp[\beta_r \norm{x}^{\upalpha k}] \geq \beta_r^{\beta_r}/(\beta_r !) \norm{x}^{r \beta_r} \geq \norm{x} \geq \abs{\varphi(x)} \eqsp .
  \end{equation}
  Similarly, if $\norm{x} \leq 1$ we have that
  $\abs{\varphi(x)} \leq 1 \leq \beta_r \exp[\beta_r \norm{x}]$. Hence for any
  $x \in \rset^d$, $\abs{\varphi(x)} \leq \beta_r \exp[\beta_r \norm{x}]$. For
  any $\varphi \in \Lipset_0$ we have $M_{1, \varphi} = 1$. In addition, the set
  $F^{-1}(0)$ is compact since
  $\lim_{\norm{x} \to +\infty} \norm{F(x)} = +\infty$ and
  $F \in \rmc(\rset^d, \rset^p)$. Therefore there exists $R > 0$ such that
  $F^{-1}(0) \subset \ball{0}{R}$. In particular, note that $\msu = \ball{0}{R}$
  is open and bounded. We have that for any $\varphi \in \Lipset_0$ and
  $x \in \msu$
  \begin{equation}
    \abs{\varphi(x)} \leq \norm{x} \leq R \eqsp . 
  \end{equation}
  Therefore we get that $M_{0,\varphi} = R$. Combining these results with
  \Cref{thm:big_theo} we have for any $\varphi \in \Lipset_0$ and
  $\vareps \in \coint{0, \bvareps(\ceil{1/r})}$
  \begin{equation}
    \abs{\pi_{\vareps}[\varphi] - \pi_0[\varphi]} \leq A(C_\varphi) (1 + M_{0,\varphi} + M_{1,\varphi}) \vareps^{1/k} \leq 2 A(\ceil{1/r})(1 + R) \vareps^{1/k} \eqsp . 
  \end{equation}
We conclude the proof upon combining this result and \eqref{eq:def_wass}.
\end{proof}

Concurrently to our work, \cite{bras2021convergence} establish similar
quantitative bounds w.r.t. the Wasserstein distance of order $1$ without the
norm-like assumption on the potential $U$ but assuming that the Hessian of $U$
is invertible on $\argmin \ensembleLigne{U(x)}{x \in \rset^d}$, see \cite[Lemma
4.6]{bras2021convergence}. To do so, the authors derive estimates of the form
$\wassersteinD[1](\pi_{\vareps_1}, \pi_{\vareps_2}) \leq
C\absLigne{\vareps_2^{1/2} - \vareps_1^{1/2}}$ with $C \geq 0$ for
$\vareps_1, \vareps_2 > 0$ using a coupling lemma and results from Morse
theory. The final bound is obtained by letting $\vareps_2 \to 0$ in the previous
inequality and the fact that $\lim_{\vareps \to 0} \pi_\vareps = \pi_0$
weakly. \Cref{coro:wass_dist} extends these bounds to the case where the Hessian
of $U$ is no longer invertible under a norm-like condition on the potential.

Also, note that \cite{hwang1980} establishes that
$\lim_{\vareps \to 0} \pi_\vareps = \pi_0$ weakly without any norm-like
assumption on $U$ under the condition that 
$\msc^\star = \argmin \ensembleLigne{U(x)}{x \in \rset^d}$ can be partitioned
into a collection of manifolds and that for any $x \in \msc^\star$, the Hessian
of $U$ along the normal plan to the manifold associated with $x$ is invertible
(we denote by $\nabla_{\mathrm{N}}^2 U(x)$ this quantity).  In
\Cref{lemma:hessian_normal}, we show that if $U = \normLigne{F}^2$ then we have
that $\det(\nabla_{\mathrm{N}}^2 U(x)) = \jac{x}^2$ for any $x \in
F^{-1}(0)$. Hence, the invertibility condition can be written as: for any
$x \in F^{-1}(0)$, $\jac{x} \neq 0$ similarly to \Cref{assum:F} and the limiting
measure identified by \cite{hwang1980} is exactly $\pi_0$.

The next result is an extension of \Cref{thm:big_theo} where, given
$k \in \nsets$, $\Psi: \ \rset^d \to \rset_+$, $\pi_\vareps$ is replaced by
$\pi_\vareps^\Psi$ defined for any $\msa \in \mcb{\rset^d}$ by
\begin{equation}
  \label{eq:twisted_def}
  \textstyle{
  \pi_\vareps^\Psi(\msa) = \left. \int_{\msa} \Psi(x) \exp[-\norm{F(x)}^k/\vareps] \rmd x \middle/ \int_{\rset^d} \Psi(x) \exp[-\norm{F(x)}^k/\vareps] \rmd x \right. \eqsp . }
\end{equation}
Similarly, for any $\msa \in \mcb{\rset^d}$ we let
(when it is well-defined)
\begin{equation}
  \label{eq:pi_0_def_both_psi}
  \textstyle{
    \pi_0^\Psi(\msa) = \left. \int_{F^{-1}(0) \cap \msa} \Psi(x) \jacinv{x} \rmd \calH^{d-\min(d,p)}(x) \middle/ \int_{F^{-1}(0)} \Psi(x) \jacinv{x} \rmd \calH^{d-\min(d,p)}(x) \right .\eqsp .
    }
\end{equation}
We consider the following assumption on $\Psi$ which ensures that
$\pi_\vareps^\Psi$ is well-defined for $\vareps \geq 0$ when combined with
\rref{assum:F}.

\begin{assumptionH}
  \label{assum:psi}
  $\Psi \in \rmc(\rset^d, \rset_+)$, there exists $C_\Psi$ such that for any
  $x \in \rset^d$, $\Psi(x) \leq C_\Psi \exp[C_\Psi \norm{x}^{\upalpha k}]$ and
  $\Psi(x) > 0$ for any $x \in F^{-1}(0)$.
\end{assumptionH}

Note that in \rref{assum:psi}, $\upalpha > 0$ is given in \rref{assum:F} and
$k \in \nsets$ is given in \eqref{eq:pi_eps_def}. Under this assumption, we
derive the following quantitative bounds.

\begin{theorem}
  \label{thm:big_theo_extension}
  Assume \rref{assum:F} and 
  \textup{\rref{assum:psi}}. Then for any $\vareps \geq 0$, $\pi_\vareps^\Psi$
  is well-defined.  Let $\msu \subset \rset^d$ be open, bounded and such that
  $F^{-1}(0) \subset \msu$. Let $\varphi \in \rmc(\rset^d, \rset)$ and
  $C_\varphi \geq 0$ such that for any $x \in \rset^d$
\begin{equation}
    \label{eq:cond_varphi}
    \abs{\varphi(x)} \leq C_\varphi \exp[C_\varphi \norm{x}^{\upalpha k}] \eqsp ,
  \end{equation}
  with $\upalpha > 0$ defined in \rref{assum:F} and $k \in \nsets$ in
  \eqref{eq:pi_eps_def}.  Then
  $\lim_{\vareps \to 0} \abs{\pi_\vareps^\Psi[\varphi] - \pi_0^\Psi[\varphi]} =
  0$. In addition, if $\varphi$ and $\Psi$ are respectively
  $M_{1, \varphi}$-Lipschitz and $M_{1, \Psi}$-Lipschitz on $\msu$ with
  $M_{1, \varphi}, M_{1, \Psi} \geq 0$, then there exist
  $A \in \rmc(\rset_+, \rset_+)$, $\bvareps \in \rmc(\rset_+, \rset_+^\star)$
  such that for any $\vareps \in \oointLigne{0, \bvareps(C_\varphi)}$ we have
  \begin{equation}
    \textstyle{
      \absLigne{\pi_{\vareps}^\Psi[\varphi] - \pi_0^\Psi[\varphi]} \leq A(C_\varphi) (1 + M_{0,\varphi} + M_{1,\varphi})  \vareps^{1/k} \eqsp ,
      }
  \end{equation}
  with
  $M_{0, \varphi} = \sup \ensembleLigne{\absLigne{\varphi(x)}}{x \in \msu}$,
  $A, \bvareps$ functions that do not depend on $\varphi$, $A$ non-decreasing and
  $\bvareps$ non-increasing.
\end{theorem}

\begin{proof}
  The proof is postponed to
  \Cref{sec:proof_thm_macro}.
\end{proof}

The same remarks formulated after \Cref{thm:big_theo} hold for this
extension. In particular the norm-like assumption on $U$ can be omitted and
replaced by a classical invertibility condition on the Hessian at $U^{-1}(0)$.
Note that \Cref{thm:big_theo} is a direct consequence of
\Cref{thm:big_theo_extension} upon letting $\Psi = 1$.
{Reciprocally, we do not have that \Cref{thm:big_theo_extension}
  is a consequence of \Cref{thm:big_theo} upon replacing $\varphi$ by
  $\varphi \Psi$. Indeed, it must be noted that, contrary to $\varphi$, $\Psi$ also
  appears in the normalizing constant of $\pi_\vareps^\Psi$. } An important
consequence of \Cref{thm:big_theo_extension} is the case where for any
$x \in \rset^d$, $\Psi(x) = \jac{x}$. Indeed, doing so we obtain that
$\pi_0^\Psi$ is the uniform distribution on $F^{-1}(0)$, \ie \ the maximum
entropy distribution with support $F^{-1}(0)$. We discuss this special case in
\Cref{sec:from-macr-micr}. Similarly to \Cref{coro:wass_dist} we can also
establish the following corollary.

\begin{corollary}
  \label{coro:wass_dist_extension}
  Assume \rref{assum:F} and 
  \textup{\rref{assum:psi}}. Then for any $\vareps \geq 0$, $\pi_\vareps^\Psi$
  is well-defined.  In addition, assume that $\Psi$ is Lipschitz
  continuous. Then there exist $A_2 \geq 0$ and $\bvareps > 0$ such that for any
  $\vareps \in \ooint{0, \bvareps}$ we have
  \begin{equation}
    \wassersteinD[1](\pi_\vareps^\Psi, \pi_0^\Psi) \leq A_2 \vareps^{1/k} \eqsp ,
  \end{equation}
  where we recall that $k \in \nsets$ is defined in \eqref{eq:pi_eps_def}.
\end{corollary}

\begin{proof}
  The proof is similar to the one of \Cref{coro:wass_dist}.
\end{proof}

{ We conclude this section with a first application of our
  results to the study of the concentration of the posterior of a Variational
  AutoEncoder (VAE). Let $p, d \in \nset$ with $d > p$. Assume we have access to
  a dataset $x_{1:n} = \{x_i\}_{i=1}^n \in (\rset^d)^n$ with $n \in \nset$ such
  that $\{x_i\}_{i=1}^n$ are i.i.d. samples from $\mu \in \Pens(\rset^d)$ a
  target data distribution. We consider a generative model to approximately sample
  from $\mu$ defined as follows: let $\eta$ be an easy-to-sample distribution in
  $\rset^p$ with density $p_\eta$ and for any $z \in \msz$, let
  $p_\theta(\cdot|z) = \mathcal{N}(m_\theta(z), \vareps \Id/2)$ with
  $\vareps > 0$ and $m_\theta : \ \rset^p \to \rset^d$ given by a neural network
  \footnote{Note that in practice the covariance matrix is also parameterized by
    a neural network but for simplicity we keep it fixed.} with parameters
  $\theta \in \Theta$. The generative model is trained by maximizing an Evidence
  Lower Bound (ELBO) which requires introducing an encoding probability
  distribution. Assuming that the generative model is trained we obtain a set of
  parameters $\theta^\star \in \Theta$ and the generative model is then given
  for any $x \in \rset^d$ by
  \begin{equation}
    \textstyle{
      p_{\theta^\star}(x) = \int_{\rset^p} p_\eta(z) p_{\theta^\star}(x|z) \rmd z \eqsp . 
      }
    \end{equation}
    For encoding purposes, we are also interested in the posterior distribution given for any
    $x \in \rset^d$, $z \in \rset^p$ by
    $p_{\theta^\star}(z|x) = p_\eta(z) p_{\theta^\star}(x|z) /
    p_{\theta^\star}(x)$. Note that for any $x \in \rset^d$ and $\msa \in \mcb{\rset^p}$ we have 
    \begin{equation}
      \textstyle{
        \int_{\msa} p_{\theta^\star}(z|x) \rmd z = \int_{\msa} p_\eta(z) \exp[-\normLigne{x - m_\theta(z)}^2/\vareps] \rmd z / \int_{\rset^p} p_\eta(z) \exp[-\normLigne{x - m_\theta(z)}^2/\vareps] \rmd z \eqsp . 
        }
    \end{equation}
    In particular, we have that for any $x \in \rset^d$ the distribution
    $\pi_{x, \vareps}$ with density $p_{\theta^\star}(\cdot|x)$ is of the form
    \eqref{eq:twisted_def} with $\Psi \leftarrow p_\eta$ and
    $F_x(z) = x - m_\theta(z)$. Under \rref{assum:F} and
    \textup{\rref{assum:psi}}, we define $\pi_{x, 0}$ such that for any
    $\msa \in \mcb{\rset^p}$
\begin{equation}
  \label{eq:pi_ll}
  \textstyle{
    \pi_{x,0}(\msa) = \left. \int_{m_\theta^{-1}(x) \cap \msa} p_\eta(z) \jacinvx{z} \rmd \calH^{0}(z) \middle/ \int_{m_\theta^{-1}(x)} p_\eta(z) \jacinvx{z} \rmd \calH^{0}(z) \right .\eqsp .
    }
  \end{equation}
  Assuming that $p_\eta$ is Lipschitz, we can apply \Cref{coro:wass_dist_extension}
  and there exist $A_{2,x} \geq 0$ and $\bvareps_x > 0$ such that for any
  $\vareps \in \ooint{0, \bvareps_x}$ we have
  \begin{equation}
    \wassersteinD[1](\pi_{x,\vareps}, \pi_{x,0}) \leq A_{2,x} \vareps^{1/2} \eqsp .
  \end{equation}
  This result provides quantitative convergence bounds for the 
  posterior distribution $p_{\theta^\star}(\cdot|x)$ towards a limiting distribution
  $\pi_{x,0}$. Note that $\pi_{x,0}$ will assign more mass on the points where 
(a) the prior distribution $p$ is high ; 
(b) the generalized Jacobian is small, i.e. the mapping $m_{\theta^\star}$
  does not fluctuate too much around $z \in m_{\theta^\star}^{-1}(x)$.
}

%%% Local Variables:
%%% mode: latex
%%% TeX-master: "main"
%%% End:

\section{Applications}
\label{sec:applications}

In this section, we present two applications of our results. First, we draw links
between two different maximum entropy distributions in \Cref{sec:from-macr-micr}. Then,
we study the convergence of SGLD for non-convex minimization in
\Cref{sec:non_cvx_mon}.

\subsection{Maximum entropy distributions}
\label{sec:from-macr-micr}

\subsubsection{Maximum entropy distributions}
\label{sec:maxim-entr-models}

Let $\msx \subset \rset^d$ be a compact space with
$\interior(\msx) \neq \emptyset$ and $G: \msx \to \rset^p$ measurable. The
maximum entropy distribution with constraint $G$ is denoted $\pi_{\mathrm{Ent}}$ and is
given by
\begin{equation}
  \label{eq:maximum_compact}
  \textstyle{\pi_{\mathrm{Ent}} = \argmax \ensembleLigne{H(\pi)}{\pi \in \Pens_{\Lebesgue}(\msx), \ \pi[\normLigne{G}] < +\infty, \ \pi[G] =0}} \eqsp ,
\end{equation}
where $\Pens_\Lebesgue(\msx)$ is the set of probability measures on $\msx$ which
admit a density w.r.t. the Lebesgue measure,
$\pi[G] = \int_\msx G(x) \rmd \pi(x)$ and
$H(\pi) = -\int_{\msx} \log (\rmd \pi / \rmd \Lebesgue)(x) \rmd \pi(x)$ for any
$\pi \in \Pens_\Lebesgue(\msx)$, where $\Lebesgue$ is the Lebesgue measure on
$\rset^d$. Note that $H(\pi)$ is well-defined (but can be infinite) since $\msx$ is 
compact and in
this case $\KLLigne{\pi}{\mu} = -H(\pi) + \log(\Leb(\msx))$, where $\mu$ is the
uniform measure on $\msx$ and we recall that $\Leb$ is the Lebesgue measure.

Such maximum entropy distributions naturally arise in many areas such as statistical
physics \citep{jaynes1957info,BrunaMallat}, econometrics
\citep{golan2008information}, generative modeling
\citep{portilla2000parametric,lu2015learning,de2021maximum}, reinforcement
learning \citep{ziebart2008maximum} or image processing
\citep{geman1984stochastic,leclaire2017stochastic}. {In texture
  synthesis applications such as \cite{lu2015learning,de2021maximum}, the
  feature mapping $G$ is given by a pretrained neural network as in
  \cite{gatys2015texture}.} In what follows, we introduce two extensions of
\eqref{eq:maximum_compact}.  More precisely, we use the analogy between the
maximization of the entropy and the minimization of an appropriate
Kullback-Leibler divergence to extend \eqref{eq:maximum_compact} on non-compact
spaces, following \cite{de2021maximum}.

\paragraph{Macrocanonical distributions} Given $\mu \in \Pens(\rset^d)$ and
$G: \ \rset^d \to \rset^p$ measurable, we say that
$\pi^\star \in \Pens(\rset^d)$ is a macrocanonical distribution
\citep{BrunaMallat} with reference measure $\mu$ and constraint $G$ if it
satisfies
\begin{equation}
  \textstyle{\pi^\star \in \argmin \ensembleLigne{\KLLigne{\pi}{\mu}}{\pi \in \Pens(\rset^d),\ \pi[\normLigne{G}]<+\infty, \ \pi[G]=0}} \eqsp . 
\end{equation}
Using results from information geometry \citep{csiszar1975divergence} any
macrocanonical distribution can be written as an exponential distribution under mild
assumptions.
\begin{proposition}[\citep{de2021maximum}]
  \label{prop:gibbs_measure}
  Assume that $G \in \rmc(\rset^d, \rset^p)$ and that there exist
  $\upalpha, \upbeta, \eta >0$ with $\upbeta > \upalpha$ such that
  \begin{equation}
    \textstyle{\sup \ensembleLigne{\normLigne{G(x)}(1+\normLigne{x}^\upalpha)^{-1}}{x \in \rset^d}<+\infty \eqsp , \qquad \int_{\rset^d} \exp[\eta \normLigne{x}^\upbeta] \rmd \mu(x) < +\infty \eqsp .}    
  \end{equation}
  In addition, assume that for any $\theta \in \rset^p$ with
  $\normLigne{\theta} = 1$, we have
  $\mu (\ensembleLigne{x \in \rset^d}{\langle \theta, G(x) \rangle < 0}) > 0$.
  Then there exists a unique macrocanonical distribution with constraint $G$ and
  reference measure $\mu$ denoted $\pi^\star$. In addition, there exists
  $\theta^\star \in \rset^p$ such that for any $x \in \rset^d$
  \begin{equation}
    \label{eq:gibbs}
    \textstyle{(\rmd \pi^\star / \rmd \mu)(x) = \exp[-\langle \theta^\star, G(x) \rangle - L(\theta^\star)] \eqsp , \qquad L(\theta^\star) = \log \int_{\rset^d} \exp[-\langle \theta^\star, G(x) \rangle ] \rmd \mu(x) \eqsp . }
  \end{equation}
\end{proposition}

We refer to \cite{de2021maximum} for a relaxation of these conditions and a
detailed study of the macrocanonical distribution. \Cref{prop:gibbs_measure} ensures
that macrocanonical distributions can be written as Gibbs measures. We emphasize
that these distributions satisfy a constraint in expectation. An alternative way
to define maximum entropy distributions under constraints is to ensure almost sure
equality. This tighter constraint yields new maximum entropy distributions called
microcanonical distributions.

\paragraph{Microcanonical distributions} A microcanonical
distribution $\pi^\star \in \Pens(\rset^d)$ with reference measure $\mu$ and constraint
$G$ can be defined as 
\begin{equation}
  \label{eq:micro}
  \textstyle{\pi^\star \in \argmin \ensembleLigne{\KLLigne{\pi}{\mu}}{\pi \in \Pens(\rset^d),\ \pi[\normLigne{G}]=0}} \eqsp . 
\end{equation}
We emphasize that imposing \eqref{eq:micro} is equivalent to impose that
$\pi^\star$ has minimum Kullback-Leibler divergence w.r.t. $\mu$ among all
distributions which satisfy $G=0$ almost surely.  Note that if there exists a
microcanonical distribution $\pi^\star$ with $\KLLigne{\pi^\star}{\mu}<+\infty$ then
$\mu(\msa) > 0$ where $\msa = \ensembleLigne{x \in \rset^d}{G(x) = 0}$ and in
this case $\pi^\star = \mu(\cdot \cap \msa)/\mu(\msa)$. In what follows, we will
say that $\pi^\star$ is the uniform microcanonical distribution with constraint
$G$ if $G^{-1}(0)$ is compact with $\calH^{d-\min(d,p)}(G^{-1}(0))<+\infty$ and
$\pi^\star$ is the uniform distribution on $G^{-1}(0)$, \ie \ for any
$\msa \in \mcb{\rset^d}$, we have
\begin{equation}
  \label{eq:micro_model}
  \textstyle{
    \pi^\star(\msa) = \int_{G^{-1}(0) \cap \msa} \rmd \calH^{d-\min(d,p)}(x) / \calH^{d-\min(d,p)}(G^{-1}(0)) \eqsp .
    }
  \end{equation}

  Note that contrary to macrocanonical distributions which admit a
  representation as a Gibbs measure, see \Cref{prop:gibbs_measure}, the
  microcanonical distributions are concentrated on the set $G^{-1}(0)$. In the next
  section, we draw links between these two maximum entropy distributions.

\subsubsection{From macrocanonical to microcanonical}
\label{sec:from-macr-micr-1}

In order to draw links between macrocanonical and microcanonical distributions we
consider a specific sequence of macrocanonical distributions associated with the
constraint of the form $G_\vareps = \normLigne{F}^2 - \vareps$ for $\vareps > 0$, and where $F: \ \rset^d \to \rset^p$. Let
$\mu \in \Pens(\rset^d)$ be a reference probability measure. Under the
conditions of \Cref{prop:gibbs_measure}, we define a family of measures
$\{\macroeps\}_{\vareps > 0}$ such that for any $\vareps > 0$, $\macroeps$
is the macrocanonical distribution associated with $G_\vareps$ and $\mu$. More
precisely, for any $\vareps > 0$, there exists $\theta_\vareps \in \rset$ such
that for any $x \in \rset^d$
\begin{equation}
  \label{eq:macro_eps}
  \textstyle{(\rmd \macroeps / \rmd \mu)(x) = \exp[-\theta_\vareps \normLigne{F(x)}^2 - L_\vareps] \eqsp , \quad L_\vareps = \log \int_{\rset^d} \exp[-\theta_\vareps \normLigne{F(x)}^2] \rmd \mu(x) \eqsp . }
\end{equation}
Our goal is to study the behavior of the family $\{\macroeps\}_{\vareps > 0}$
when $\vareps \to 0$. In particular, we show that there exists a reference
measure $\mu$ such that the limit $\pi_0 = \lim_{\vareps \to 0} \macroeps$
exists where $\pi_0$ is the uniform microcanonical distribution
\eqref{eq:micro_model}. We start with the following proposition which ensures
that $\lim_{\vareps \to 0} \theta_\vareps = +\infty$ with linear rate.

\begin{proposition}
  \label{prop:behavior_theta}
  Let $\mu \in \Pens(\rset^d)$ and $F: \ \rset^d \to \rset^p$. Assume that the
  conditions of \Cref{prop:gibbs_measure} with
  $G_\vareps = \normLigne{F}^2 - \vareps$ for any $\vareps > 0$ are satisfied
  and  let $\macroeps$ be the macrocanonical distribution
  with constraint $G_\vareps$ and reference measure $\mu$. Assume that there
  exists $\Psi: \ \rset^d \to \rset_+$ such that $\mu$ admits a density
  w.r.t. the Lebesgue measure given by $\Psi$. In addition, assume that
  \rref{assum:F} and \textup{\rref{assum:psi}} hold. Then, we have that
  $\theta_\vareps \sim_{\vareps \to 0} p / (2\vareps)$.
\end{proposition}

\begin{proof}
  We provide a sketch of the proof. The whole proof is postponed to
  \Cref{sec:proof-macro}. In the case where $d \geq p$ we use tools from
  geometric measure theory to derive an equivalent of
  $\int_{\rset^d} \Psi(x) \normLigne{F(x)}^2 \exp[-\normLigne{F(x)}^2/\vareps] \rmd x$
  when $\vareps \to 0$ akin to \Cref{prop:upper_bound_i_in}. In the case
  $d \leq p$, using Morse lemma and classical Laplace analysis we show a similar
  result. We conclude upon combining these results and that
  $\macroeps[\normLigne{F}^2] = \vareps$.
\end{proof}

Note that the equivalence in \Cref{prop:behavior_theta} does not depend on the reference measure $\mu$. 
In \Cref{sec:proof-macro}, we prove an extension of \Cref{prop:behavior_theta},
see \Cref{prop:behavior_theta_gene}. In particular we show that
$G_\vareps = \normLigne{F}^2 - \vareps$ can be replaced by
$G_\vareps = \normLigne{F}^k - \vareps$ with $k \in \nsets$ under similar
conditions. In particular, we get that for any $k \in \nsets$,
$\theta_\vareps \sim_{\vareps \to 0} \mathcal{C}_k / \vareps$ with
$\mathcal{C}_k$ explicit if $k=2$.  Finally note that under
the same conditions as \Cref{prop:behavior_theta}, we have that
$\pi_{\vareps}^\Psi[\normLigne{F}^2] \sim_{\vareps \to 0} p /(2\vareps)$, where we
recall that for any $\vareps >0$ and $\msa \in \mcb{\rset^d}$
\begin{equation}
  \textstyle{\pi_\vareps^\Psi(\msa) = \int_{\msa} \Psi(x) \exp[-\normLigne{F(x)}^2/\vareps] \rmd x / \int_{\rset^d} \Psi(x) \exp[-\normLigne{F(x)}^2/\vareps] \rmd x \eqsp . }
\end{equation}
This comes from the facts that $\macroeps=\pi_{\theta_\vareps}^\Psi$ and $\macroeps[\normLigne{F}^2] = \vareps$.

The following proposition establishes quantitative bounds w.r.t. the Wasserstein
of order $1$ between $\{\macroeps\}_{\vareps >0}$ and a limiting measure. We
obtain this result upon combining \Cref{prop:behavior_theta} and
\Cref{coro:wass_dist}.

\begin{proposition}
  \label{prop:convergence_macro}
  Assume the same conditions as \Cref{prop:behavior_theta}. Then there exist
  $\bvareps > 0$ and $A_3 \geq 0$ such that for any
  $\vareps \in \ocint{0, \bvareps}$ we have
  \begin{equation}
    \wassersteinD[1](\macroeps, \pi_0^\Psi) \leq A_3 \vareps^{1/2} \eqsp ,
  \end{equation}
  where $\pi_0^\Psi$ is given in \eqref{eq:pi_0_def_both_psi}.
\end{proposition}

\Cref{prop:convergence_macro} establishes quantitative bounds between the family
of macrocanonical distributions $\{\macroeps\}_{\vareps > 0}$ and a limiting
distribution $\pi_0^\Psi$. Note that this limiting distribution is \emph{not}
the uniform microcanonical distribution $\pi^\star$ defined in
\eqref{eq:micro_model} in general, even though it is supported on the set
$F^{-1}(0)$. However, it is possible to sample from this $\pi^\star$ by
choosing $\mu$ such that for any $x \in F^{-1}(0)$,
$\Psi(x) = (\rmd \mu / \rmd \Leb)(x) = \jac{x}$.

\subsubsection{Some simple experiments}

\paragraph{Methodology}

In this experimental section we consider two simple examples of functions
$F:\R^d\to\R^p$ with $d=1$ or $2$ and $p=1$. We consider the associated distributions
$\{\pi_{\vareps}\}_{\vareps > 0}$ and $\{\pi_{\vareps}^\Psi\}_{\vareps > 0}$
  and their limit as $\eps$ goes to $0$, showing in particular that for
  $\Psi=\mathrm{J}F$, it converges to the uniform microcanonical distribution that is the
  uniform distribution on $F^{-1}(0)$. We also check experimentally the scaling
  relation of Proposition \ref{prop:behavior_theta}.

\paragraph{Zeros of a polynomial}
In this first example, let $P:\R\to\R$ be a polynomial. We are interested in the
zeros of $P$, and, more precisely in sampling in a uniform way on the set of the
zeros of $P$. Using the notations of the previous sections, we set $F=P$.  For
$\eps>0$, let us define the two distributions, with respective densities with
respect to the Lebesgue measure on $\R$ given for any $x \in \rset$ by 
\begin{align}
&\textstyle{(\rmd \pi_{\vareps} / \rmd \Leb)(x) =
  \exp[-F(x)^2/\eps]/\int_{\R}\exp[-F(\tilde{x})^2/\eps] \rmd \tilde{x} } \eqsp , \\
  &\textstyle{(\rmd \pi_{\vareps}^\Psi / \rmd \Leb)(x) =
\jac{x} \exp[-F(x)^2/\eps]/\int_{\R}\jac{\tilde{x}} \exp[-F(\tilde{x})^2/\eps] \rmd \tilde{x} }  \eqsp ,
\label{PQpolynom:eq}
\end{align}
where for any $x \in \rset$, $\Psi(x)=\jac{x}=|P'(x)|$.

In \Cref{polynomial.ex1:fig}, we display the polynomial $P$ where for any
$x \in \rset$, $P(x)=x(x-0.5)(x-1.7)(x-2.5)$. We also check numerically the
scaling relation of Proposition \ref{prop:behavior_theta}. Then, in 
\Cref{polynomial.ex2:fig}, we illustrate the different behaviors of $\pi_{\vareps}$
and $\pi_{\vareps}^\Psi$ when $\eps$ is small ($\eps=10^{-3}$ in our experiments). We
also present the histograms obtained using of $N=10^4$ independent samples of
$\pi_{\vareps}$ and of $\pi_{\vareps}^\Psi$. { These samples were simply obtained by the numerical CDF inversion method.} The stars indicate the target limit
distributions (that is the uniform distribution on the zeros of $P$ for the
limit of $\pi_{\vareps}^\Psi$).

\begin{figure}[h]
\begin{center}
\includegraphics[width=6.5cm]{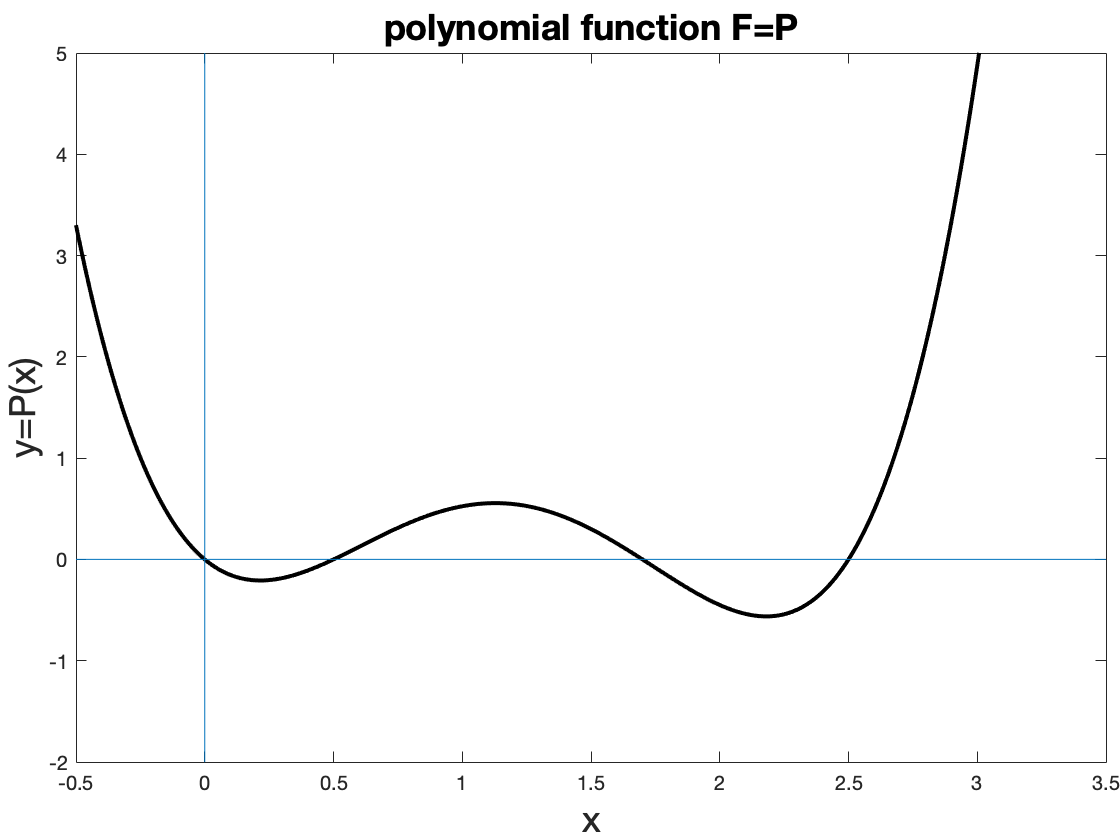} \hspace{0.5cm}
\includegraphics[width=6.5cm]{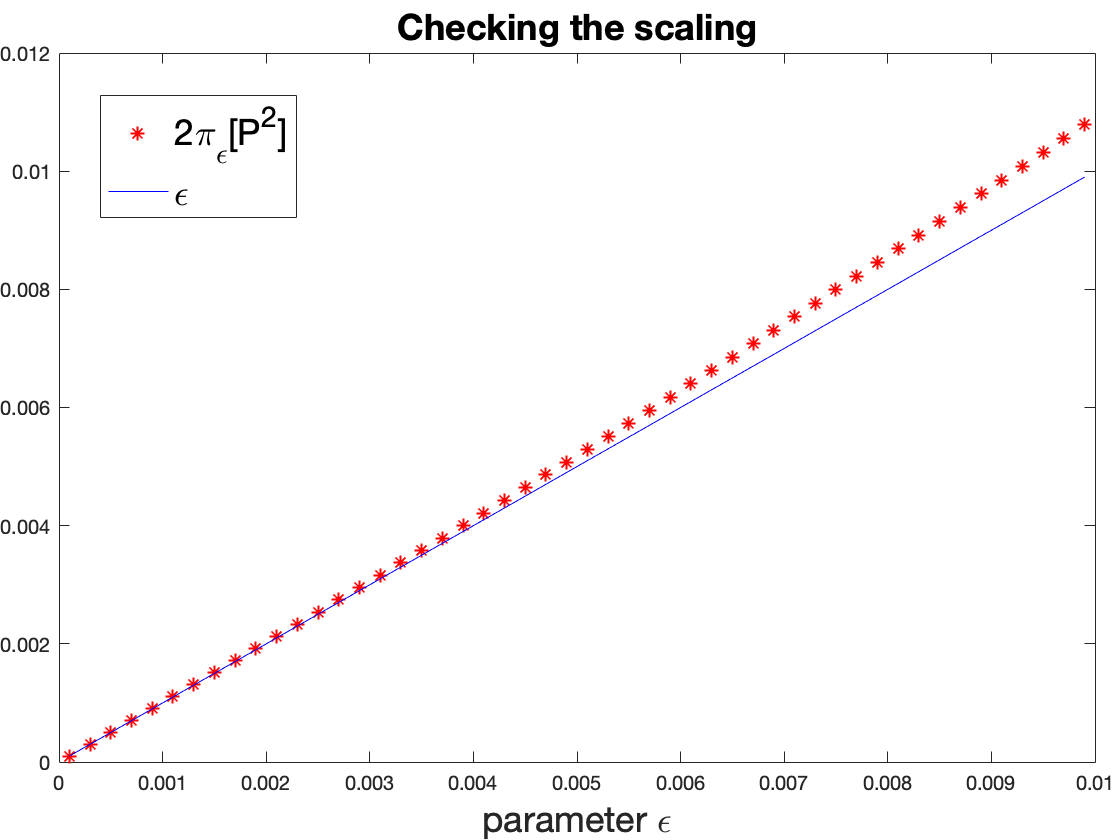} 
\end{center}
\caption{Left: graph of the polynomial $x\mapsto P(x)$, that has $4$
  zeros. Right: verifying the scaling relation of \Cref{prop:behavior_theta} by
  plotting $\eps\mapsto 2 \pi_{\vareps}[P^2]$, which is equivalent to $\eps$ as $\eps$
  goes to $0$.}
\label{polynomial.ex1:fig}
\end{figure}

\begin{figure}[h]
\centerline{\includegraphics[width=13cm]{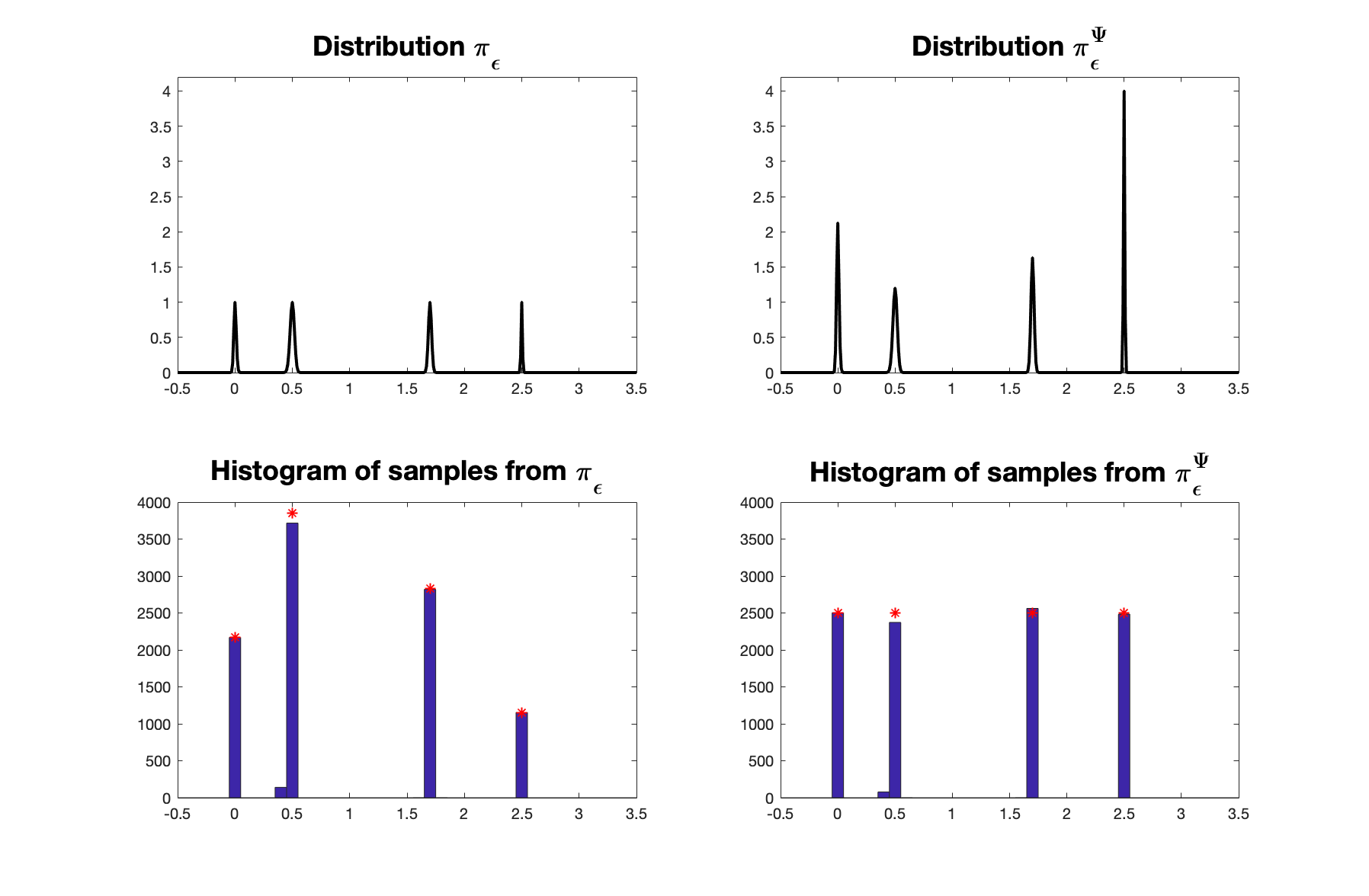}}
\caption{Distributions $\pi_{\vareps}$ and $\pi_{\vareps}^\Psi$ (first line),
  and histogram of their samples from them (second line). This experiment shows
  that the limit distribution of $\pi_{\vareps}^\Psi$ as $\eps$ goes to $0$ is
  the uniform microcanonical distribution given by the uniform distribution on the
  zeros of $P$.}
\label{polynomial.ex2:fig}
\end{figure}

\paragraph{Two-dimensional ellipse}

In this second example, we consider the function $F:\R^2\to\R$ given for any
$x=(x_1,x_2)\in\R^2$ by $F(x)=a_1 x_1^2 + a_2 x_2^2 -1$ with $a_1, a_2 >0$.  For
$\eps>0$, we define $\pi_\vareps$ and $\pi_\vareps^\Psi$ whose densities w.r.t. 
the Lebesgue measure are given for any $x \in \rset^2$ by
\begin{align}
&\textstyle{(\rmd \pi_\vareps / \rmd \Leb)(x) =
  \exp[-\normLigne{F(x)}^2/\eps]/\int_{\R^2}\exp[-\normLigne{F(\tilde{x})}^2/\eps] \rmd \tilde{x} } \eqsp , \\
  &\textstyle{(\rmd \pi_\vareps^\Psi / \rmd \Leb)(x) =
\jac{x} \exp[-\normLigne{F(x)}^2/\eps]/\int_{\R^2}\jac{\tilde{x}} \exp[-\normLigne{F(\tilde{x})}^2/\eps] \rmd \tilde{x} }  \eqsp .
\label{PQellipse:eq}
\end{align}
To sample from $\pi_\vareps$ and $\pi_\vareps^\Psi$, we use two Markov chains given by 
the Unadjusted Langevin Algorithm
\citep{roberts1996exponential,durmus2017nonasymptotic,dalalyan2017theoretical},
given respectively by $X_0, Y_0 \in \rset^2$ and the recursions 
\begin{align}
  &X_{n+1} = X_n - \gamma (\nabla \|F\|^2(X_n) /\eps ) + \sqrt{2\gamma}
    Z_{n+1} \eqsp , \\
  &Y_{n+1} = Y_n  - \gamma (\nabla \|F\|^2(Y_n) /\eps - \nabla \log \jac{Y_n} ) + \sqrt{2\gamma}
Z_{n+1} \eqsp , 
\end{align}
where $\gamma > 0$ is a stepsize and $\{Z_n\}_{n \in \nset}$ is a family of
independent Gaussian random variables with zero mean and identity covariance
matrix. For any $x \in \rset^2$ we have
\begin{align}
  &\textstyle{\jac{x} =2(a_1^2 x_1^2 + a_2^2 x_2^2)^{1/2} \eqsp ,} \quad \textstyle{\nabla \|F\|^2(x)= 2 F(x) \nabla F(x) = 4 (a_1 x_1^2 + a_2 x_2^2
    -1) \left( \begin{array}{l} a_1 x_1 \\  a_2 x_2 \end{array} \right) \eqsp ,} \\
  &\textstyle{\nabla \log \jac{x} = (a_1^2 x_1^2 + a_2^2
  x_2^2)^{-1}  \left( \begin{array}{l} a_1^2 x_1 \\  a_2^2 x_2 \end{array} \right)} \eqsp .
\end{align}
The set $F^{-1}(0)$ is an ellipse, of cartesian equation
$a_1 x_1^2 + a_2 x_2^2 =1$ and has a polar parametrization given for any
$\theta \in [0,2\uppi)$ by
$$\textstyle{r(\theta) = (a_2 + (a_1-a_2)\cos^2\theta)^{1/2} \eqsp .  }$$
This is not an arc-length parametrization, and the infinitesimal length of the
curve element between $\theta$ and $\theta + d\theta$ is given by
$\ell(\theta)d\theta$ where for any $\theta \in \ocint{0,2\uppi}$ we have
$$\ell(\theta) = (r(\theta)^2 + r'(\theta)^2)^{1/2} = (a_2^2 +
(a_1^2-a_2^2)\cos^2\theta)^{1/2}(a_2 + (a_1-a_2)\cos^2\theta)^{-3/2} \eqsp .$$
In \Cref{ellipse.ex1:fig}, we set $a_1=1$, $a_2=4$, $\eps=10^{-3}$,
$\gamma=10^{-5}$, $N=9\times 10^{7}$ iterations in the Markov chains, and
histograms with bins of size $0.05\uppi$.  In \Cref{ellipse.ex1:fig}, we compare
two different histograms: the ones of the angle values ($\theta$) for the points
of the chain $X_n$ (diamonds), of the chain $Y_n$ (stars), and the two
distributions with density
$\theta \mapsto \ell(\theta) / \int_{-\pi}^\pi \ell(\theta) \rmd \theta$ (black
curve) and
$\theta \mapsto \ell(\theta) \jacinv{\theta} / \int_{-\pi}^\pi \ell(\theta)
\jacinv{\theta} \rmd \theta$ (red curve). Note that
$\theta \mapsto \ell(\theta) / \int_{-\pi}^\pi \ell(\theta) \rmd \theta$
corresponds to the Hausdorff measure on the ellipse pushed by the polar
parametrisation $\theta \mapsto r(\theta)$, \ie \ the uniform microcanonical
distribution pushed by $\theta \mapsto r(\theta)$. We observe that the Markov chain
corrected with the generalized Jacobian indeed achieves the uniform 
microcanonical distribution on the set $F^{-1}(0)$. Finally, we also check experimentally in Figure \ref{ellipse.ex2:fig}
the scaling relation of Proposition \ref{prop:behavior_theta}.

\begin{figure}
\begin{center}
\begin{tabular}{cc}
\includegraphics[width=.5\linewidth]{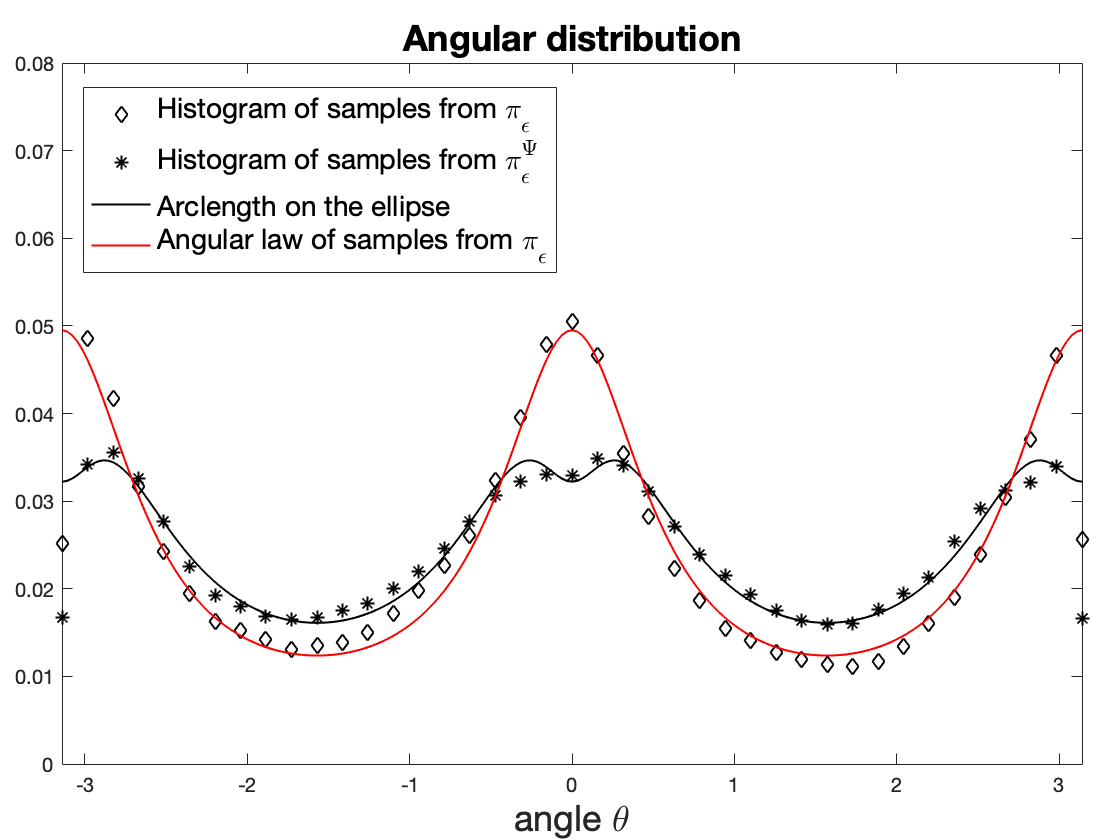} &
\includegraphics[width=.5\linewidth]{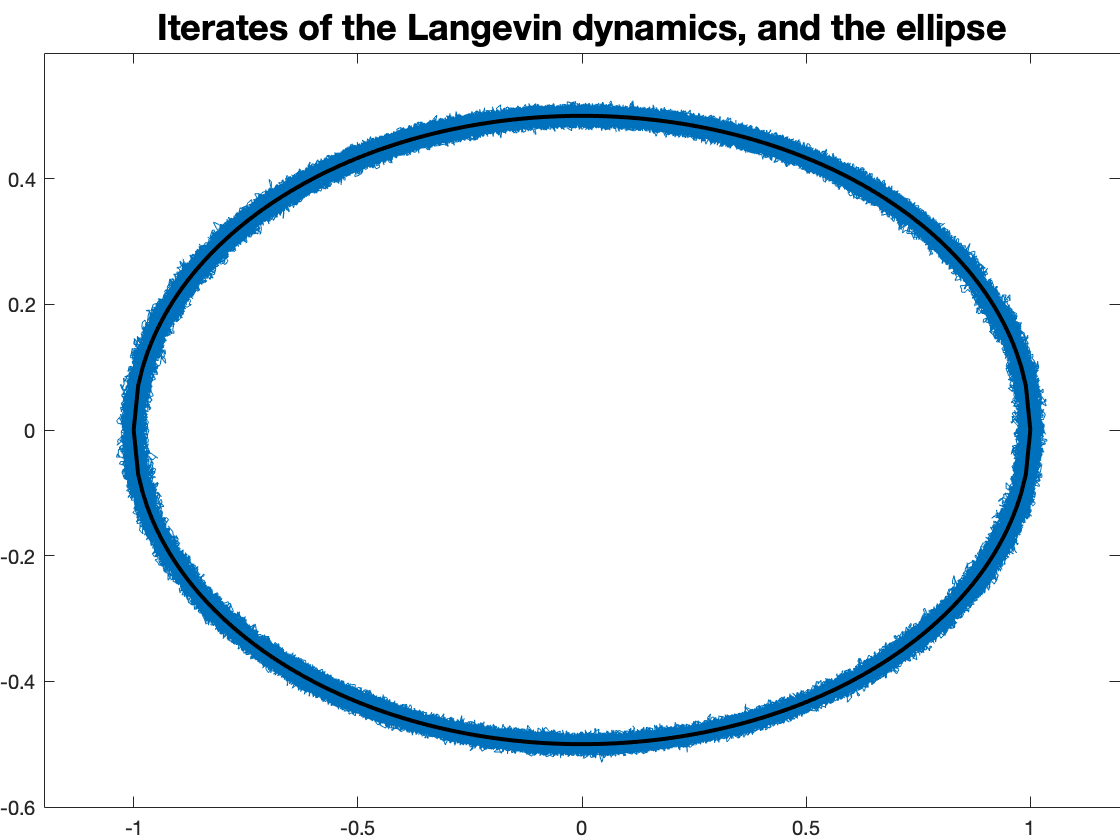} 
\end{tabular}
\end{center}
\caption{Left: histogram of angles of the samples $(X_n)_{n \in \nset}$ and
  $(Y_n)_{n \in \nset}$, showing that $\pi_\vareps^\Psi$ is close to the uniform
  distribution on the ellipse for small values of $\vareps >0$. Right: the
  trajectory of the first $10^6$ samples of $Y_n$ (with a burnin of $10^3$ iterates).}
\label{ellipse.ex1:fig}
\end{figure}

\begin{figure}
\begin{center}
\includegraphics[width=.6\linewidth]{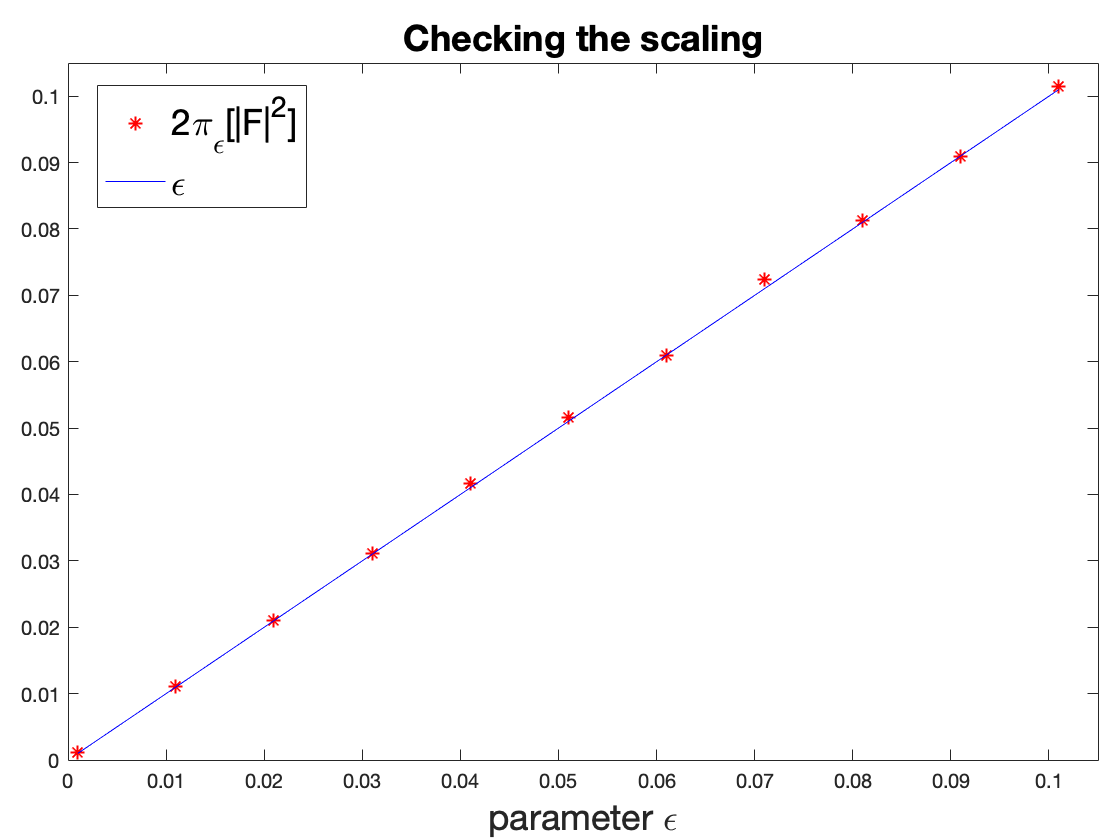} 
\end{center}
\caption{Checking the scaling
  relation of Proposition \ref{prop:behavior_theta}.}
\label{ellipse.ex2:fig}
\end{figure}

%%% Local Variables:
%%% mode: latex
%%% TeX-master: "main"
%%% End:

%%% Local Variables:
%%% mode: latex
%%% TeX-master: "main"
%%% End:

\subsection{Non-convex minimization}
\label{sec:non_cvx_mon}

\subsubsection{Non-convex setting and related work}
\label{sec:related-work}

In this section, we consider the following minimization problem:
\begin{equation}
  \text{find } x^\star \in \argmin \ensembleLigne{U(x)}{x \in \rset^d} \eqsp ,
\end{equation}
with $U \in \rmc^1(\rset^d, \rset)$ and
$\argmin \ensembleLigne{U(x)}{x \in \rset^d}\neq \emptyset$. Here, we do not
assume that $U$ is convex. Hence, classical first-order optimization schemes
such as gradient descent might get trapped in saddle points. Adding
isotropic Gaussian noise to this dynamics circumvents this issue, see
\cite{duflo1996algorithmes,pelletier1998weak} for instance. The algorithm is
then given by the following recursion: $X_0 \in \rset^d$ and for any
$k \in \nset$
\begin{equation}
  \label{eq:sgd_lazy}
  X_{k+1} = X_k - \gamma_k \nabla U(X_k) + Z_{k+1} \eqsp ,
\end{equation}
where $(Z_k)_{k \in \nset}$ is a sequence of independent Gaussian random
variables with zero mean and covariance matrix $\sigma_k^2 \Id$, with
$(\gamma_k)_{k \in \nset} \in (\rset_+)^\nset$ a sequence of stepsizes and
$(\sigma_k)_{k \in \nset} \in (\rset_+)^\nset$. If $\sigma_k^2 = \gamma_k^2$
for any $k \in \nset$, $\sum_{k \in \nset} \gamma_k = +\infty$ and
$\sum_{k \in \nset} \gamma_k^2 < +\infty$ then the algorithm is in the
\emph{weakly disturbed} regime and under additional assumptions on $U$ one can
show that $(X_k)_{k \in \nset}$ converges almost surely to a local minimizer of
$U$, see \cite{pelletier1998weak} for instance. However,
$(U(X_k)_{k \in \nset})_{k \in \nset}$ does not necessarily converge to the
global minimum of $U$. The intuition behind this behavior is that the variance
of the noise decreases too quickly for the sequence to explore efficiently the
landscape of $U$.

In order to perform global optimization, one can consider \emph{simulated
  annealing} algorithms where $\sigma_k^2 = \gamma_k T_k$ for any $k \in \nset$,
with $(T_k)_{k \in \nset} \in (\rset_+)^\nset$ a sequence of temperatures which
slowly decrease. These algorithms were introduced in the context of discrete
optimization in \cite{kirkpatrick1984optimization} and have been thoroughly
investigated in \cite{gidas1985nonstationary} (discrete state-space),
\cite{gelfand1991recursive,gelfand1993metropolis} (discrete-time algorithm),
\cite{geman1986diffusions,chiang1987diffusion,holley1989asymptotics}
(diffusion), \cite{pelletier1998weak} (CLT type results),
\cite{romeijn1994simulated} (constrained optimization),
\cite{yang2000convergence} (non Gaussian noise) and
\cite{andrieu2001convergence} (control of the sequence in total variation) for
instance. In \cite{hajek1988cooling,hajek1985tutorial} a sufficient and
necessary condition is given on the rate of decrease of $(T_k)_{k \in \nset}$ so
that $U(X_k)_{k \in \nset}$ converges towards the minimum of $U$, more precisely
the condition reads $\sum_{k \in \nset} \exp[-b^\star/T_k] = +\infty$, where
$b^\star>0$ is a parameter which depends only on $U$ called the \emph{kinetic
  barrier} (or depth in \cite{hajek1988cooling,hajek1985tutorial}).  One of the
main limiting factor of simulated annealing is the slow rate of convergence of
$(T_k)_{k \in \nset}$ towards $0$, which is often set as $T_k = C/\log(k+1)$ for
any $k \in \nset$ and for some constant $C \geq 0$. Note that recently the
convergence of modifications of the simulated annealing algorithms under faster
cooling rates have been investigated in discrete spaces by either changing the
Markov chain transitions \citep{choi2019improved} or the energy landscape
\citep{choi2020convergence}, see also \cite{catoni1998energy}. Finally, we
emphasize that most of the results regarding the convergence of
\eqref{eq:sgd_lazy} can be extended to the case where $\nabla U$ is replaced by
an unbiased estimator under additional conditions.

With the advent of neural networks, numerous schemes exploiting the annealing
structure have been proposed to minimize the non-convex losses which arise in
deep learning applications, see \cite{ye2017langevin} for instance. Drawing
connections with unadjusted Langevin algorithms, and in particular Stochastic
Gradient Langevin Dynamics (SGLD) \citep{welling2011bayesian},
\cite{raginsky2017non,zhang2017hitting} replace $\nabla U$ by an unbiased
estimator and let $\gamma_k = \gamma > 0$, $\sigma_k^2 = 2 \gamma \vareps$ for
any $k \in \nset$ with $\vareps > 0$ in \eqref{eq:sgd_lazy}.  Under curvature
and regularity assumptions on the potential $U$, the authors derive quantitative
bounds on $(\expeLigne{U(X_k) - \min_{\rset^d}U})_{k \in \nset}$. Since then
several accelerations have been proposed in the literature
\citep{chau2019stochastic,gao2018global,erdogdu2018global,nguyen2019non,zhang2019nonasymptotic,xu2017global}.

In the next section, we improve the results of \cite{raginsky2017non} by
providing upper bounds w.r.t. the first order Wasserstein distance between the
distribution of $(X_k)_{k \in \nset}$ given by SGLD and an explicit limiting
distribution. Our results complete the ones of \cite{raginsky2017non} which
deal with the behavior of
$(\expeLigne{U(X_k) - \min_{\rset^d}U})_{k \in \nset}$. To the best of our
knowledge, this is the first result establishing quantitative bounds on the
distance between the iterates of SGLD and a limiting measure concentrated on the
minimizers of the target potential $U$.

\subsubsection{Quantitative convergence for SGLD}
\label{sec:quant-conv-sgld}

\paragraph{Setting and notation} In this section we start by recalling the
setting considered in \cite{raginsky2017non}. We assume that there exist a
topological space $(\msz, \mcb{\msz})$, a probability measure
$\mu \in \Pens(\msz, \mcb{\msz})$ and $u: \ \rset^d \times \msz \to \rset_+$
such that for any $z \in \msz$, $u(\cdot, z) \in \rmc^1(\rset^d, \rset_+)$, and
define for any $x \in \rset^d$
\begin{equation}
  \label{eq:def_U}
  \textstyle{U(x) = \int_{\msz} u(x,z) \rmd \mu(z) \eqsp . }
\end{equation}
We also denote by $U^\star$ the global minimum of $U$.  We do not have access to
$\mu$ and $U$ directly but instead consider an empirical version of the target.
Let $n \in \nset$. We define $U_n: \ \rset^d \times \msz^n \to \rset_+$ given
for any $z^{1:n} = \{z_i\}_{i=1}^n \in \msz^n$ and $x \in \rset^d$ by
\begin{equation}
  \textstyle{U_{n}(x, z^{1:n}) = (1/n) \sum_{i=1}^n u(x, z_i) \eqsp . }
\end{equation}
Let $(\msy, \mcy)$ be a measurable space,
$\Rker: \ \msz^n \times \mcy \to \ccint{0,1}$ a Markov kernel and
$g: \ \rset^d \times \msy \to \rset^d$ such that for any $x \in \rset^d$ and
$z^{1:n} \in \msz^n$ we have
\begin{equation}
  \label{eq:sto_grad}
  \textstyle{ \nabla_x U_{n}(x, z^{1:n}) = \int_{\msy} g(x, y) \Rker(z^{1:n}, \rmd y) \eqsp .}
\end{equation}
Let $Z$ be a random variable on $\msz^n$ with distribution $\mu^{\otimes
  n}$. Let $\{Y_k\}_{k \in \nset}$ be a family of independent random variables
on $\msy$ such that conditionally to $Z$ we have for any $k \in \nset$ that
$Y_k$ has distribution $\updelta_Z \Rker$. Finally, let $\vareps > 0$. We
consider the Stochastic Gradient Langevin Dynamics (SGLD) sequence
$(X_k)_{k \in \nset}$ given by the following recursion: $X_0 \in \rset^d$ and
for any $k \in \nset$
\begin{equation}
  \label{eq:recursion_SGLD}
  X_{k+1} = X_k - \gamma g(X_k, Y_k) + \sqrt{2 \gamma \vareps} G_{k+1} \eqsp ,
\end{equation}
where $(G_k)_{k \in \nset}$ is a sequence of independent Gaussian random
variables with zero mean and identity covariance matrix and $\gamma > 0$ is a
stepsize. We also assume that $\{Y_k\}_{k \in \nset}$ and
$\{G_k\}_{k \in \nset}$ are independent. For any $k \in \nset$, we denote by
$\Qker_k: \ \msz^n \times \mcb{\rset^d} \to \ccint{0,1}$, the Markov kernel such
that for any random variable $Z$ on $\msz^n$ with distribution
$\mu^{\otimes n}$, we have that $X_k$ has distribution $\updelta_Z \Qker_k$
conditionally to $Z$. {
\begin{table}[h!]
  {
  \centering
\begin{tabular}{ |c|c|c|c| } 
  \hline
  Notation & Source & Target & Description \\
  \hline \hline
  $\mu$ & $\cdot$ & $\msz$ & Data distribution, see \eqref{eq:def_U}. \\
  \hline
  $\Rker$ & $\msz^n$ & $\msy$ & \makecell{Markov kernel defining the  stochastic gradient \\conditionally to the data, see \eqref{eq:sto_grad}.} \\
  \hline
  $\Qker_k$ & $\msz^n$ & $\rset^d$ & \makecell{Markov kernel associated with SGLD at step $k$ \\ conditionally to the data, see \eqref{eq:recursion_SGLD}.} \\
  \hline 
  $\Sker_\vareps$ & $\msz^n$ & $\rset^d$ & \makecell{Gibbs measure with temperature $\vareps >0$ \\ (conditional to the data) defined by \eqref{eq:def_s_eps}.} \\
  \hline 
  $\Sker_0$ & $\msz^n$ & $\rset^d$ & \makecell{ Limiting measure at temperature zero \\ (conditional to the data) defined by \eqref{eq:def_s_zero}.} \\
  \hline
\end{tabular}
\caption{Summary of Markov kernels and probability measures used in this section.}
  \label{table:markov_kernels}}
\end{table}
  Finally, we also define for any $\vareps >0$ the Markov kernel
$\Sker_\vareps: \ \msz^n \times \mcb{\rset^d} \to \ccint{0,1}$ associated with
the Gibbs measure with temperature $\vareps$ such that for any
$\msa \in \mcb{\rset^d}$ and $z^{1:n} \in \msz^n$
\begin{equation}
  \label{eq:def_s_eps}
    \textstyle{\updelta_{z^{1:n}}\Sker_\vareps(\msa) = \left. \int_{\msa} \exp[-U_n(x,z^{1:n})/\vareps] \rmd x \middle/ \int_{\rset^d} \exp[-U_n(x,z^{1:n})/\vareps] \rmd x \right.  \eqsp .}
  \end{equation}
  Similarly at temperature $0$, we define the Markov kernel
  $\Sker_0: \ \msz^n \times \mcb{\rset^d} \to \ccint{0,1}$ such that for any
  $\msa \in \mcb{\rset^d}$ and $z^{1:n} \in \msz^n$
  \begin{align}
    \label{eq:def_s_zero}
    \updelta_{z^{1:n}}\Sker_0(\msa) &\textstyle{= \left. \int_{\msc(z^{1:n}) \cap \msa}  \det(\nabla_x^2U_n(x,z^{1:n}))^{-1/2} \rmd \calH^{0}(x) \right. }\\
    & \qquad \qquad \textstyle{\left./ \int_{\msc(z^{1:n})} \det(\nabla_x^2U_n(x,z^{1:n}))^{-1/2} \rmd \calH^{0}(x) \right .\eqsp ,}
\end{align}
with $\msc(z^{1:n}) = \argmin \ensembleLigne{U_n(x,z^{1:n})}{x \in \rset^d}$. We
summarize the different probability measures and Markov kernels used in this
section in Table \ref{table:markov_kernels}.
}

\paragraph{Assumptions} We first consider the following assumption, which is similar
to the one of \cite{raginsky2017non}.
\begin{assumptionH}[$n$]
  \label{assum:raginsky}
  For any $z \in \msz$, $u(\cdot, z) \in \rmc^\infty(\rset^d, \rset)$ and the following hold:
  \begin{enumerate}[wide, labelwidth=!, labelindent=0pt, label=(\alph*)]
  \item There exist $A, B \geq 0$ such that for any $z \in \msz$, $\abs{u(0,z)} \leq A$ and $\normLigne{\nabla_x u(0,z)} \leq B$.
  \item There exists $\Mtt \geq 0$ such that for any $x_1,x_2 \in \rset^d$, 
    $z \in \msz$,
    $\normLigne{\nabla_x^k u(x_1,z) - \nabla_x^k u(x_2,z)}\leq \Mtt\normLigne{x_1 -
      x_2}$ for any $k \in \{1, 2, 3\}$.
  \item There exist $\mtt > 0$, $\ctt \geq 0$ such that for any $x \in \rset^d$,
    $z \in \msz$,
    $\langle \nabla_x u(x,z), x \rangle \geq \mtt \normLigne{x}^2 - \ctt$.
  \item There exists $\kappa \geq 0$ such that for any $x \in \rset^d$ and $z^{1:n} \in \msz^n$ we have
    \begin{equation}
      \textstyle{\int_{\msz^n} \int_{\msy} \normLigne{\nabla_x U_{n}(x, z^{1:n}) - g(x,y)}^2 \Rker(z^{1:n}, \rmd y) \rmd \mu^{\otimes n}(z^{1:n}) \leq \kappa(1 + \normLigne{x}^2) \eqsp . }
    \end{equation}
  \end{enumerate}
\end{assumptionH}

Similar to \cite{raginsky2017non} we define, {for any $\vareps > 0$} the uniform spectral gap
\begin{equation}  
  \textstyle{\lambda^\star(\vareps) = \inf_{z^{1:n} \in \msz^n, \ h \in \rmc^1(\rset^d) \cap \rmL^2(\updelta_{z^{1:n}} \Sker_\vareps)} \ensembleLigne{\updelta_{z^{1:n}}  \Sker_\vareps[ \normLigne{\nabla h}^2]  /  \updelta_{z^{1:n}} \Sker_\vareps[ \absLigne{h}^2] }{\ h \neq 0 , \ \updelta_{z^{1:n}} \Sker_\vareps[h] = 0 } \eqsp , }
\end{equation}
and note that under \Cref{assum:raginsky}($n$), $\lambda^\star(\vareps) > 0$.
{We emphasize that \Cref{assum:raginsky}($n$) is satisfied in the
  case of a quadratic loss with a predictor given by a smooth and bounded neural
  network with bounded derivatives up to order $4$ and a quadratic
  regularization. More precisely, we can consider
  $u(x,z) = \normLigne{f(x, z_1) - z_2}^2+ \alpha \normLigne{x}^2$ with
  $f \in \rmc^\infty(\rset^d \times \msz_1, \msz_2)$ and
  $\msz = \msz_1 \times \msz_2$ with $f$ and its derivatives bounded up to order
  $4$ and $\alpha >0$. In this setting $f$ can be seen as a predictor of $z_2$ given $z_1$ and
  $x$ the parameters of $f$.} Finally, we also consider the following assumption
which ensures that the limiting measures we consider are well-defined.
\begin{assumptionH}[$n$]
  \label{assum:U_n_param}
  $u \in \rmc(\rset^d \times \msz, \rset)$ and the following hold:
  \begin{enumerate}[wide, labelwidth=!, labelindent=0pt, label=(\alph*)]
  \item $\msz$ is compact. 
  %\item % There exist $N \in \nset$, $\{x_0^k\}_{k=1}^N \in \rmc(\msz^n, \rset^d)$
    % such that for any $z^{1:n} \in \msz^n$,
    % $\msc(z^{1:n}) \subset \{x_0^k(z)\}_{k=1}^N$ and
    % {For any $z^{1:n} \in \msz^n$, $\msc(z^{1:n}) \neq \emptyset$, where
    % $\msc(z^{1:n}) = \argmin \ensembleLigne{U_n(x,z^{1:n})}{x \in \rset^d}$.}
  \item For any $\upbeta > 0$,
    $\int_{\msz^n} \sigma_\upbeta^\star(z^{1:n}) \rmd \mu^{\otimes n}(z^{1:n})
    < +\infty$, where 
    \begin{equation}
      \label{eq:def_sigma_star}
      \textstyle{\sigma_\upbeta^\star(z^{1:n}) = \sum_{x \in \msc(z^{1:n})} \det(\nabla_x^2 U(x,z^{1:n}))^{-\upbeta}  \eqsp . }
    \end{equation}
  \end{enumerate}
\end{assumptionH}
{Note that \Cref{assum:U_n_param} is satisfied if the number of
  minimizers is bounded w.r.t. $z \in \msz$ and if for any $z^{1:n} \in \msz^n$,
  $x^\star \in \msc(z^{1:n})$, $\nabla_x^2 U(x^\star,z^{1:n}) \succeq \eta \Id$
  with $\eta >0$. Note that this condition is a slight strengthening of the
  condition that for any $z^{1:n} \in \msz^n$, $x^\star \in \msc(z^{1:n})$,
  $\nabla_x^2 U(x^\star,z^{1:n}) \succ 0$. In particular, we impose that the
  landscape of $U(\cdot, z)$ is not too flat around the minimizers.} We also
introduce the \energygap \ $c^\star: \ \msz^n \to \rset$ such that for any
$z^{1:n} \in \msz^n$, $c^\star(z^{1:n}) = +\infty$ if $U_n(\cdot, z^{1:n})$ does
not admit a local minimizer which is not a global minimizer and
\begin{equation}
  c^\star(z^{1:n}) = \inf \ensembleLigne{U_n(x,z^{1:n})}{\text{$x$ local min. of $U_n(\cdot,z)$ but not global min.}} - U_n^\star(z^{1:n}) \eqsp ,
\end{equation}
with $U_n^\star(z^{1:n}) = \inf \ensembleLigne{U_n(x,z^{1:n})}{x \in \rset^d}$.  We
refer to \Cref{sec:import-local-glob} for a discussion on the \energygap \ and
its importance in non-convex optimization.

\paragraph{Main results} We are now ready to state our main results. First,
under \Cref{assum:raginsky}($n$) and \Cref{assum:U_n_param}($n$), we derive
quantitative bounds for the sequence $(X_k)_{k \in \nset}$.

\begin{proposition}
  \label{prop:quantitative_sgld}
  Let $n \in \nset$. Assume \tup{\rref{assum:raginsky}($n$)} and
  \tup{\rref{assum:U_n_param}($n$)}. Then there exist $C \geq 0$,
  $\bvareps, \bgamma, \upbeta > 0$ and $k_0 \in \nset$ such that for any
  $\vareps \in \ocintLigne{0, \bvareps}$, $\gamma \in \ocintLigne{0, \bgamma}$
  and $k \in \nset$ with $k \geq k_0$ we have
  \begin{align}
    \label{eq:convergence_wass}
    \wassersteinD[1](\mu^{\otimes n} \Qker_k, \mu^{\otimes n} \Sker_0) &\leq C(1/\vareps + d)^2(\kappa^{1/4} \log(1/\gamma) + \gamma)/(\lambda^\star(\vareps) \vareps)  \\
                                                                       & \qquad \qquad \textstyle{+ C(1 + D_n) (\vareps^{1/2} + \vareps^{-d/2} \int_{\msz^{1:n}} \exp[-c^\star(z^{1:n})/\vareps] \rmd \mu^{\otimes n}(z^{1:n})) \eqsp , }            
  \end{align}
  with
  $D_n = \int_{\msz^n} \sigma_\upbeta^\star(z^{1:n}) \rmd \mu^{\otimes
    n}(z^{1:n}) < +\infty$.
\end{proposition}

\begin{proof}
  We provide a sketch of the proof. The whole proof is postponed to
  \Cref{sec:application_to_sgld}. First, we assess the convergence of
  $(\mu^{\otimes n} \Qker_k)_{k \in \nset}$ by splitting the error in two parts. A first
  part is bounded using the geometric ergodicity of SGLD as in
  \cite{raginsky2017non} and controls the distance between $\mu^{\otimes n} \Qker_k$ and
  $\mu^{\otimes n} \Sker_\vareps$.
  Then, using a parametric version of \Cref{thm:big_theo}, see
  \Cref{prop:conclusion_param_bounds}, we bound the distance between
  $\mu^{\otimes n} \Sker_\vareps$ and $\mu^{\otimes n} \Sker_0$.  % For the second part of the
  % proof, we adapt stability arguments from \cite{raginsky2017non}, see also
  % \cite{bousquet2002stab}.  \valentin{si on peut montrer la tightness de
  %   $\mu\S_0$ wrt $n$ alors on peut dire que la mesure limite existe et est
  %   concentre sur les min. En fait mieux que ca, on peut montrer que l'on
  %   retombe sur la mesure limite de Hwang lorsque n est grand}
\end{proof}

A few remarks are in order:
\begin{enumerate}[wide, labelwidth=!, labelindent=0pt, label=(\alph*)]
\item The condition that
  $\int_{\msz^n} \upsigma_1^\star(z^{1:n}) \rmd \mu^{\otimes n}(z^{1:n}) <
  +\infty$ is necessary to ensure that $\mu^{\otimes n} \Sker_0$ is
  well-defined. In \Cref{prop:quantitative_sgld} we assume the 
  condition that
  $\int_{\msz^n} \upsigma_\upbeta^\star(z^{1:n}) \rmd \mu^{\otimes n}(z^{1:n})
  < +\infty$ for any $\upbeta > 0$. In fact the condition could be relaxed to
  $\int_{\msz^n} \upsigma_\upbeta^\star(z^{1:n}) \rmd \mu^{\otimes n}(z^{1:n})
  < +\infty$ for any $\upbeta \in \ocint{0, \upgamma}$ with $\upgamma > 0$ an
  explicit constant. Obtaining such a result requires to derive explicit bounds
  in the quantitative Morse lemma \cite[Theorem 3.2]{le2014numerical}, see also
  \Cref{prop:quantitative_morse_lemma}. We leave this analysis to future
  works. Note also that this condition can be satisfied upon regularizing the
  function $u$ and assuming that the number of global minimizers is
  bounded. Indeed, if we replace $u(x,z)$ by $u(x,z) + (\alpha/2) \normLigne{x}^2$
  (with $\alpha > 0$ some regularization parameter) then we get that for any
  $\upbeta >0$,
  $\int_{\msz^n} \upsigma_\upbeta^\star(z^{1:n}) \rmd \mu^{\otimes n}(z^{1:n})
  \leq N_0\alpha^{-\upbeta d}$ where $N_0$ is an upper-bound on the number of global
  minimizers.
\item The upper-bound in \eqref{eq:convergence_wass} also depends on the
  \energygap . This constant quantifies how close the local
  minima which are not global minima are to the global minima and is
  crucial to establish quantitative parametric Laplace-type results. We
  illustrate this situation in \Cref{sec:import-local-glob} with a simple 
  example.
\item \Cref{prop:quantitative_sgld} ensures that for a given precision level
  $r > 0$, there exist $\vareps, \gamma > 0$ small enough and $k \in \nset$
  large enough such that
  $\wassersteinD[1](\mu^{\otimes n} \Qker_k, \mu^{\otimes n} \Sker_0) \leq
  r$. However, note that $\mu^{\otimes n} \Sker_0$ is not concentrated on the
  minimizers of $U$. This highlights the fact that the minimization of the
  empirical risk does not guarantee that the population risk is small.
\end{enumerate}

In order to verify that the population risk is small in expectation
  w.r.t. the target measure $\mu^{\otimes n} \Sker_0$ we use stability tools to establish
  the following proposition. 

  \begin{proposition}
    \label{prop:stability}
{    Assume that \tup{\rref{assum:raginsky}($n$)} and
    \tup{\rref{assum:U_n_param}($n$)} hold uniformly w.r.t. $n \in \nset$. Assume that 
    \begin{equation}
      \lim_{\vareps \to 0} \sup \ensembleLigne{\textstyle{\vareps^{-d/2} \int_{\msz^{1:n}}\exp[-c^\star(z^{1:n})/\vareps] \rmd \mu^{\otimes n}(z^{1:n})}}{n \in \nset} = 0 \eqsp ,
    \end{equation}
    Then
\begin{equation}
 \lim_{n \to +\infty} \{ \mu^{\otimes n} \Sker_0 [U] - U^\star\} = 0 \eqsp .
\end{equation}
    In
    addition, assume that there exist $C_0, \alpha, \bvareps  > 0$ such that for
    any $\vareps \in \ocintLigne{0, \bvareps}$
    \begin{equation}
      \label{eq:thermo_cond}
      \textstyle{\vareps^{-d/2} \int_{\msz^{1:n}}\exp[-c^\star(z^{1:n})/\vareps] \rmd \mu^{\otimes n}(z^{1:n})} \leq C_0 \vareps^{\alpha} \eqsp ,
    \end{equation}
    Then for any
    $\eta \in \ooint{0,1}$, there exist $C_\eta \geq 0$ and $n_0 \in \nset$ such
    that for any $n \geq n_0$ we have
\begin{equation}
  \mu^{\otimes n} \Sker_0 [U] - U^\star \leq C_\eta / \log(n)^{s \eta} \eqsp ,
\end{equation}
with $s = \min(1/4, \alpha/2)$ and where we recall that $U^\star$ is the global minimum of $U$.}
\end{proposition}

\begin{proof}
The proof is postponed to \Cref{sec:stab-limit-meas}.
\end{proof}

Note that \Cref{prop:stability} implies that the sequence
$(\mu^{\otimes n} \Sker_0 [U])_{n \in \nset}$ is tight since under
\tup{\rref{assum:raginsky}($n$)} we have that there exist $R \geq 0$ and
$\upalpha > 0$ such that for any $x \in \rset^d$ with $\normLigne{x} \geq R$
$U(x) \geq \mtt \normLigne{x}^\upalpha$. Furthermore, \Cref{prop:stability}
implies that each limiting point of the sequence
$(\mu^{\otimes n} \Sker_0 [U])_{n \in \nset}$ is concentrated on the minimizers
of $U$ when $n \to +\infty$, \ie \ when the number of training points
$\{z_i\}_{i=1}^n$ grows to infinity. However, we do not necessarily have that
$(\mu^{\otimes n} \Sker_0 [U])_{n \in \nset}$ converges towards $\pi_0$ given
for any $\msa \in \mcb{\rset^d}$ by
\begin{equation}
  \textstyle{\pi_0(\msa) = \int_{\msa \cap \msc} \det(\nabla^2U(x))^{-1/2} \rmd \calH^0(x) / \int_{\msc} \det(\nabla^2U(x))^{-1/2} \rmd \calH^0(x)} \eqsp , 
\end{equation}
where $\msc = \argmin \ensembleLigne{U(x)}{x \in
  \rset^d}$. {\Cref{prop:stability} depends crucially on the
  \energygap . In particular \eqref{eq:thermo_cond} allows us to derive 
  quantitative stability results. The \energygap \ and the associated conditions
  are discussed in \Cref{sec:import-local-glob}.}

In what follows, we describe a counter-example for which
$(\mu^{\otimes n} \Sker_0)_{n \in \nset}$ does not converge weakly towards
$\pi_0$. Let $\msx = \rset$, $\msz = \ccint{-1/2,1/2}$ and
$u: \ \msx \times \msz \to \rset$ such that for any $x \in \rset$ and
$z \in \ccint{-1/2,1/2}$ we have
\begin{equation}
  \label{eq:def_u_counter}
  \textstyle{
  u(x,z) = \left\lbrace
  \begin{matrix*}[l]
    &(x+\uppi)^4/(1 + (x+\uppi)^2) + \cos(3x) + zx \eqsp , & \text{if $x<-\uppi \eqsp ,$} \\
    &\cos(3x) + zx \eqsp , & \text{if $-\uppi\leq x\leq\uppi \eqsp ,$} \\
    &(x-\uppi)^4/(1 + (x-\uppi)^2) + \cos(3x) + zx \eqsp , & \text{if $x>\uppi \eqsp . $} \\
  \end{matrix*}
\right. }
\end{equation}
We also let $\mu$ to be the uniform measure on $\ccintLigne{-1/2,1/2}$. 
We obtain that for any $n \in \nset$, $z^{1:n} \in \msz^n$ and $x \in \msx$ 
\begin{equation}
  \textstyle{U(x) = u(x,0) \eqsp , \qquad U_n(x,z^{1:n}) = u(x,(1/n) \sum_{k=1}^n z^{k}) \eqsp .}
\end{equation}
Note that for any $z \neq 0$, $x \mapsto u(x,z)$ admits a unique global
minimizer, see the proof of \Cref{prop:counter}, whereas if $z=0$,
$x \mapsto u(x,z)$ admits four global minimizers
$\{x_i^\star\}_{i=1}^4 = \{-\uppi, -\uppi/3, \uppi/3, \uppi\}$, see
\Cref{fig:counter} for an illustration. In the next proposition, we show that
$\lim_{n \to +\infty} \mu^{\otimes n} \Sker_0[U]=U^\star$ but that
$(\mu^{\otimes n} \Sker_0)_{n \in \nset}$ does \emph{not} converge weakly 
towards $\pi_0$.

\begin{proposition}
  \label{prop:counter}
  Let $u$ be given by \eqref{eq:def_u_counter} and $\mu$ be the uniform measure
  on $\ccint{-1/2,1/2}$. Then, we have that for any $n \in \nset$
  \begin{equation}
    \textstyle{\lim_{n \to +\infty} \mu^{\otimes n} \Sker_0 =  (\updelta_{-\uppi} +
\updelta_{\uppi})/2\eqsp \quad \text{and} \quad \mu^{\otimes n} \Sker_0[U] - U^\star \leq (\uppi/(6\sqrt{3})) n^{-1/2} \eqsp . }      \end{equation}
\end{proposition}

\begin{proof}
  The proof is postponed to \Cref{proof:prop:counter}.
\end{proof}
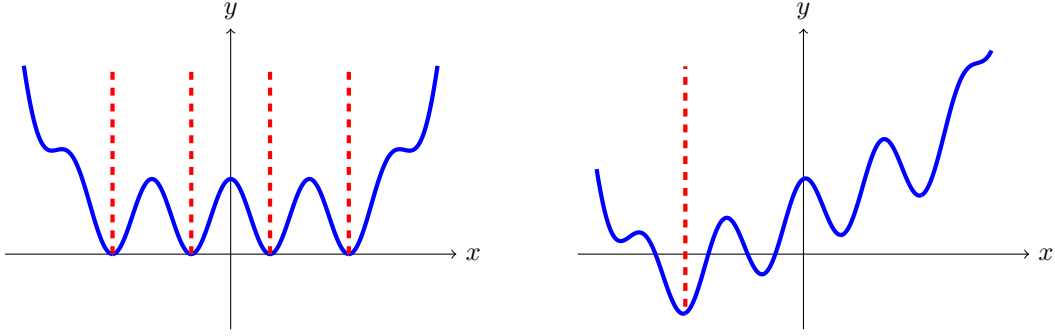
\begin{figure}
  \centering
  \hfill
\begin{tikzpicture}
  \draw[->] (-3, 0) -- (3, 0) node[right] {$x$};
  \draw[->] (0, -1) -- (0, 3) node[above] {$y$};
  
  \draw[scale=0.5, domain=-pi:pi, smooth, variable=\x, blue, ultra thick] plot[samples=100] ({\x}, {cos(3*\x*180/pi)+1});
  \draw[scale=0.5, domain=pi:5.5, smooth, variable=\x, blue, ultra thick] plot[samples=100] ({\x}, {(\x-pi)^4/(1+(\x-pi)^2)+cos(3*\x*180/pi)+1});
  \draw[scale=0.5, domain=-5.5:-pi, smooth, variable=\x, blue, ultra thick] plot[samples=100] ({\x}, {(\x+pi)^4/(1+(\x+pi)^2)+cos(3*\x*180/pi)+1});
  
  \draw[dashed, scale=0.5, domain=0:5, smooth, variable=\y, red, ultra thick] plot[samples=100] (pi, {\y});
  \draw[dashed, scale=0.5, domain=0:5, smooth, variable=\y, red, ultra thick] plot[samples=100] (pi/3, {\y});
  \draw[dashed, scale=0.5, domain=0:5, smooth, variable=\y, red, ultra thick] plot[samples=100] (-pi/3, {\y});
  \draw[dashed, scale=0.5, domain=0:5, smooth, variable=\y, red, ultra thick] plot[samples=100] (-pi, {\y});
\end{tikzpicture}
\hfill
\begin{tikzpicture}
  \draw[->] (-3, 0) -- (3, 0) node[right] {$x$};
  \draw[->] (0, -1) -- (0, 3) node[above] {$y$};
  
  \draw[scale=0.5, domain=-pi:pi, smooth, variable=\x, blue, ultra thick] plot[samples=100] ({\x}, {0.5*\x+cos(3*\x*180/pi)+1});
  \draw[scale=0.5, domain=pi:5, smooth, variable=\x, blue, ultra thick] plot[samples=100] ({\x}, {0.5*\x+(\x-pi)^4/(1+(\x-pi)^2)+cos(3*\x*180/pi)+1});
  \draw[scale=0.5, domain=-5.5:-pi, smooth, variable=\x, blue, ultra thick] plot[samples=100] ({\x}, {0.5*\x+(\x+pi)^4/(1+(\x+pi)^2)+cos(3*\x*180/pi)+1});

  \draw[dashed, scale=0.5, domain=-1.4:5, smooth, variable=\y, red, ultra thick] plot[samples=100] (-pi, {\y});
\end{tikzpicture}
\caption{Left: the function $x \mapsto u(x,0)$ with $u$ given in
  \eqref{eq:def_u_counter}. Right: the function $x \mapsto u(x,0.5)$. The global
  minimizers are given by the dotted red lines.  \label{fig:counter}}
\end{figure}

%%% Local Variables:
%%% mode: latex
%%% TeX-master: "main"
%%% End:

In particular, we have that $(\mu^{\otimes n} \Sker_0)_{n \in \nset}$ converges
towards a limiting probability measure supported on the set of minimizers of
$U$.

\subsubsection{The importance of the \energygap}
\label{sec:import-local-glob}

To conclude this section, we investigate the role of the \energygap \ in order
to establish quantitative parametric Laplace-type results. This quantity should
not be confused with the concept of \emph{kinetic barrier} which has been
investigated in the context of simulated annealing, see
\cite{hajek1988cooling,hajek1985tutorial} for instance. We refer to
\cite{wang2016tuning} for an introduction to the concept of \energygap \ and
\emph{kinetic barrier} in the context of chemistry, see also
\Cref{fig:energy_barrier} for an illustration.

In a general setting, we
consider a function $f: \ \msx \times \msz \to \rset$ where $\msx$ and $\msz$
are topological spaces and $f(\cdot, z)$ admits a global minimizer for any
$z \in \msz$. Let $z \in \msz$, if $f(\cdot,z)$ admits a local minimizer which
is not a global minimizer we recall that the \energygap \ $c^\star(z)$ is given
by
\begin{equation}
 c^\star(z) = \inf \ensembleLigne{f(x,z)}{\text{$x$ is a local minimizer of $f(\cdot,z)$ but not a global minimizer}} - f^\star(z) \eqsp ,
\end{equation}
with $f^\star(z) = \inf \ensembleLigne{f(x,z)}{x \in \msx}$. The \energygap \
quantifies how close the values of the local minima are to the global
ones. Let $\msx = \rset^d$ and for any $\vareps > 0$, $z \in \msz$ and $\msa \in \mcb{\rset^d}$, define
\begin{align}
  &\textstyle{\updelta_z \Sker_\vareps(\msa) = \int_{\msa} \exp[-f(x,z)/\vareps] \rmd x / \int_{\rset^d} \exp[-f(x,z)/\vareps] \rmd x \eqsp ,} \\
    &\textstyle{\updelta_z \Sker_0(\msa) = \int_{\msa \cap \msc(z)} \det(\nabla_x^2f(x,z))^{-1/2} \rmd \calH^0(x) / \int_{\msc(z)} \det(\nabla_x^2f(x,z))^{-1/2} \rmd \calH^0(x)  ,}
\end{align}
with $\msc(z) = \argmin \ensembleLigne{f(x,z)}{x \in \rset^d}$.  In
\Cref{prop:conclusion_param_bounds} we show (under assumptions on $f$) that for
any $z \in \msz$ and for any $\varphi \in \rmc(\rset^d, \rset)$ which satisfies
the conditions of \Cref{prop:conclusion_param_bounds} there exist
$A, \upbeta \geq 0$ and $\bvareps >0$ such that for any
$\vareps \in \ooint{0, \bvareps}$
\begin{equation}
  \label{eq:upper_bound_cstar}
        \abs{\updelta_z \Sker_\vareps[\varphi] - \updelta_z \Sker_0[\varphi]} \leq A(1 + \sigma_{\upbeta}^\star(z)) \{ \vareps^{1/2} + \vareps^{-d/2} \exp[-c^\star(z)/\vareps]\} \eqsp ,
      \end{equation}
      with $A, \upbeta$ that do not depend on $z$. The dependency of the right-hand
      side w.r.t. $\sigma_{\upbeta}^\star(z)$ comes from the fact that
      $\updelta_z \Sker_0$ is well-defined if and only if $\nabla_x^2 f(x,z)$ is
      invertible at the global minimizers of $f(\cdot,z)$.

      We now investigate the dependency w.r.t. the \energygap \ $c^\star$. We are
      going to build a simple example for which the \energygap \ plays a
      crucial role. In particular, we will show that the dependency of the form
      $\exp[-c^\star(z)/\vareps]$ is tight in \eqref{eq:upper_bound_cstar}. Let
      $k \in \nset$, $\msx = \{0,1\}$, $\msz = \rset$ and
      $f: \ \msx \times \msz \to \rset$ such that for any $x \in \msx$ and
      $z \in \msz$, $f(x,z) = xz^{2k+1}$. For any $\vareps > 0$ and
      $z \in \rset$ we define $\updelta_z \Sker_\vareps$ by
      \begin{align}
        \updelta_z \Sker_\vareps &= (\updelta_0 \exp[-f(0,z)/\vareps]+ \updelta_1 \exp[-f(1,z)/\vareps])/(1 + \exp[-f(1,z)/\vareps]) \\
        &= \updelta_0 \sigmoid(z^{2k+1}/\vareps)+ \updelta_1 \sigmoid(-z^{2k+1}/\vareps) \eqsp ,
      \end{align}
      where $\sigmoid: \ \rset \to \rset$ is the sigmoid function given for any
      $t \in \rset$ by $\sigmoid(t) = (1 + \exp[-t])^{-1}$.  When $z > 0$ the
      minimum of $f(\cdot,z)$ is $0$ and is attained at $x=0$. When $z = 0$, we
      have $f=0$ (and the minimum is therefore attained at $x=0$ and
      $x=1$). When $z < 0$ the minimum of $f(\cdot,z)$ is $-z^{2k+1}$ and is
      attained at $x=1$. Therefore we have that
      $\updelta_z \Sker_0 = \updelta_0$ if $z>0$,
      $\updelta_z \Sker_0 = \updelta_1$ if $z<0$ and
      $\updelta_z \Sker_0 = (\updelta_0 + \updelta_1)/2$ if $z=0$.  Using that
      the Wasserstein distance of order $1$ between two Bernoulli distributions
      with parameter $p_1$ and $p_2$ is given by $\absLigne{p_1 - p_2}$ we get
      that for any $z \in \msz$ and $\vareps > 0$
      \begin{equation}
        \wassersteinD[1](\updelta_z \Sker_\vareps, \updelta_z \Sker_0) = \sigmoid(-\abs{z}^{2k+1}/\vareps) \eqsp . 
      \end{equation}
      % Let
      % $\varphi: \ \msx \to \rset$ such that $\varphi(1) = 0$ and
      % $\varphi(0)>0$. First, for any $z, \vareps > 0$ we have that
      % \begin{equation}
      %   \label{eq:approx_laplace_simple}
      %   \updelta_z \Sker_\vareps[\varphi] - \updelta_z \Sker_0[\varphi] = \varphi(0)(1 - \sigma(z^{2k+1}/\vareps)) \leq \vareps / z^{2k+1} \eqsp . 
      % \end{equation}
      Hence, for a \emph{fixed} value of $z \in \msz$, we get that
      $\wassersteinD[1](\updelta_z \Sker_\vareps, \updelta_z \Sker_0)$ is of
      order $\bigO(\exp[-\abs{z}^{2k+1}/\vareps])$. In particular, we get that
      for any $z \in \msz$,
      $\limsup_{\vareps \to 0} \wassersteinD[1](\updelta_z \Sker_\vareps,
      \updelta_z \Sker_0)/\vareps$ is bounded.  In what follows, we show that
      $\liminf_{\vareps \to 0} \wassersteinD[1](\mu \Sker_\vareps,
      \mu \Sker_0)/\vareps = +\infty$, for some probability measure
      $\mu \in \Pens(\rset)$.

      Let $\chi: \ \msx \to \rset$ with $\chi(0) = 1$ and $\chi(1) = 0$. First, note that
      for any $z > 0$ we have
      \begin{equation}
        \updelta_z \Sker_\vareps[\chi] - \updelta_z \Sker_0[\chi]  = \sigmoid(-z^{2k+1}/\vareps) \eqsp . 
      \end{equation}
      Hence, using that $c^\star(z) = \abs{z}^{2k+1}$ for any $z \in \rset$ and that 
      $\sigmoid(-t) \geq \exp[-2t]$ for any $t > 0$ we have  for any $z > 0$,
      \begin{equation}
        \label{eq:lower_bound}
        \updelta_z \Sker_\vareps[\chi] - \updelta_z \Sker_0[\chi] \geq  \exp[-2z^{2k+1}/\vareps] \geq \exp[-2c^\star(z)/\vareps] \eqsp . 
      \end{equation}
      Let $\mu \in \Pens(\rset_+)$ such that $\mu(\{0\})=0$. Using
      \eqref{eq:lower_bound} and that $\chi$ is $1$-Lipschitz, we have
      \begin{equation}
        \textstyle{\wassersteinD[1](\mu \Sker_\vareps, \mu \Sker_0) \geq \int_{0}^{+\infty} \exp[-2c^\star(z)/\vareps] \rmd \mu(z) \eqsp .}
      \end{equation}
      Assume that $\mu$ is the uniform distribution on $\ccint{0,1}$. Then, we
      have for any $\vareps \in \ooint{0,1}$
      \begin{align}
        \textstyle{\wassersteinD[1](\mu \Sker_\vareps, \mu \Sker_0)}  
                                                                      \geq \textstyle{ \int_{0}^{1} \exp[-2z^{2k+1}/\vareps] \rmd x \geq 2^{-1/(2k+1)} (\int_0^1 \exp[-z^{2k+1}]\rmd z) \vareps^{1/(2k+1)} \eqsp .}
      \end{align}
      This shows that the order of
      $\wassersteinD[1](\mu \Sker_\vareps, \mu \Sker_0)$ is at most
      $\bigO(\vareps^{1/(2k+1)})$. This is in stark contrast with the order
      identified for
      $\wassersteinD[1](\updelta_z \Sker_\vareps, \updelta_z \Sker_0)$.

      This latter observation highlights the crucial role of the \energygap \
      when establishing uniform Laplace-type results w.r.t. some parameter
      $z \in \msz$. Note that we recover that
      $\wassersteinD[1](\mu \Sker_\vareps, \mu \Sker_0)$ is of order at least
      $\bigO(\vareps)$ if $\mu$ is supported on $\coint{z_0, +\infty}$ with
      $z_0 > 0$. Similar conclusions hold if we show that $c^\star(z) \geq c_0$
      for any $z \in \rset$ with $c_0 > 0$. Hence, (assuming that $c^\star$ is
      continuous) the discrepancy between the order of
      $\wassersteinD[1](\mu \Sker_\vareps, \mu \Sker_0)$ and the one of
      $\wassersteinD[1](\updelta_z \Sker_\vareps, \updelta_z \Sker_0)$ might
      arise if:
  \begin{enumerate*}[label=(\alph*)]
  \item At least one of the local minima (which is not a global minimum)
    converges towards a global minimum when $z \to z^\star$ for some value of
    $z^\star \in \msz$, \item $z^\star$ belongs to the support of $\mu$.
  \end{enumerate*}
  If these two conditions are fulfilled then a more careful study of
  $\int_\msz \exp[-c^\star(z)/\vareps] \rmd \mu(z)$ is needed in order to obtain
  quantitative bounds.

  % Note that \energygap \ differs from the \emph{kinetic barrier} (also called
  % the \emph{height}, see \cite{hajek1988cooling,hajek1985tutorial}) which is
  % used in the analysis of convergence of the simulated annealing algorithm, see
  % for an illustration.

  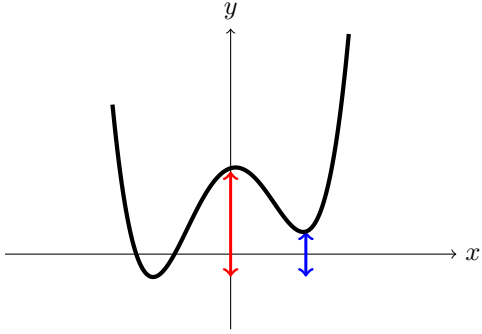
\begin{figure}
  \centering
\begin{tikzpicture}
  \draw[->] (-3, 0) -- (3, 0) node[right] {$x$};
  \draw[->] (0, -1) -- (0, 3) node[above] {$y$};

  \draw[<->, line width=0.4mm, red] (0, -.3) -- (0, 1.1) node[right] {};
  \draw[<->, line width=0.4mm, blue] (1, -.3) -- (1, .3) node[right] {};
  
  \draw[scale=0.5, domain=-pi:pi, smooth, variable=\x, black, ultra thick] plot[samples=100] ({\x}, {(\x-2)^2*(\x+2)^2/7 + .3*\x});
\end{tikzpicture}
\caption{Difference between the \energygap \ (blue) and the \emph{kinetic barrier} (red).}
    \label{fig:energy_barrier}
\end{figure}

%%% Local Variables:
%%% mode: latex
%%% TeX-master: "main"
%%% End:

%%% Local Variables:
%%% mode: latex
%%% TeX-master: "main"
%%% End:

\section{Proofs}
\label{sec:proof}

In this section, we gather the proofs of the previous sections. In
\Cref{sec:proof_thm_macro} we prove \Cref{thm:big_theo_extension}. Then, in
\Cref{sec:proof-macro} we provide the proofs of the results of
\Cref{sec:from-macr-micr}. Finally, the proofs of the results of
\Cref{sec:non_cvx_mon} are given in \Cref{sec:quant_sgld}. 

\subsection{Proof of \Cref{thm:big_theo_extension}}
\label{sec:proof_thm_macro}

In this section, we prove \Cref{thm:big_theo_extension}. We recall that
\Cref{thm:big_theo} is a straightforward consequence of
\Cref{thm:big_theo_extension} upon letting $\Psi = 1$.  We let $k \in \nsets$,
$F: \ \rset^d \to \rset^p$ and $\Psi: \ \rset^d \to \rset_+$.  For any
$\varphi: \ \rset^d \to \rset_+$ and $\vareps > 0$ we define
\begin{align}
  \label{eq:i_vareps_int}
  &\textstyle{\I_{\vareps}(\varphi) = C_{\vareps}^{-1} \int_{\rset^d} \varphi(x) \Psi(x) \exp[-\norm{F(x)}^k/\vareps] \rmd x \eqsp , \quad \J_{\vareps} = \I_{\vareps}(1) \eqsp ,} \\ &\textstyle{C_{\vareps} = \int_{\rset^d} \exp[-\norm{x}^k/\vareps] \rmd x = \vareps^{d/k} \int_{\rset^d} \exp[-\norm{x}^k]  \rmd x = \vareps^{d/k} C_1 \eqsp .}
\end{align}
In addition, we define
\begin{equation}
  \label{eq:i_0_int}
  \textstyle{
    \I_0(\varphi) =  \int_{F^{-1}(0)} \varphi(x) \Psi(x) \jacinv{x} \rmd \calH^{d-\hat{d}}(x) \eqsp , \quad \J_0 = \I_0(1) \eqsp ,
    }
\end{equation}
where $\hat{d} = \min(d, p)$ and we recall that for any $x \in \rset^d$,
$\jac{x} =\det(\rmD F(x) \rmD F(x)^\top)^{1/2}$ if $d \geq p$ and
$\jac{x} =\det(\rmD F(x)^\top \rmD F(x))^{1/2}$ otherwise.  If
$\varphi: \ \rset^d \to \rset$ and $I_{\vareps}(\abs{\varphi}) < +\infty$ for
some $\vareps \geq 0$ we define $\I_{\vareps}(\varphi)$ similarly as in
\eqref{eq:i_vareps_int} and \eqref{eq:i_0_int}. Note that for any
$\vareps \geq 0$ and $\varphi: \ \rset^d \to \rset^p$ such that it is defined we
have $\pi_{\vareps}[\varphi] = \I_\vareps(\varphi) / \J_\vareps$.

The rest of this section is organized as follows.  In \Cref{sec:case-d-geq}, we
prove our main result , \ie \ a quantitative Laplace-type result in the case
$d \geq p$ using the coarea formula. In \Cref{sec:case-d-leq}, we prove similar
results in the case $d \leq p$ using Laplace's method and Morse theory. Finally,
we conclude with the proof of \Cref{thm:big_theo_extension} in
\Cref{sec:proof-crefthm:big}. Additional technical results are postponed to
\Cref{sec:technical-bounds}.

\subsubsection{The case $d \geq p$}
\label{sec:case-d-geq}

In what follows, we assume that $d \geq p$ and for any
$\varphi: \ \rset^d \to \rset_+$, $t \in \rset^p$ we define $\L_t(\varphi)$ by
\begin{equation}
  \label{eq:l_t}
  \textstyle{
    \L_t(\varphi) = \int_{F^{-1}(t)} \varphi(x)  \Psi(x)  \jacinv{x} \rmd \calH^{d-p}(x) \eqsp .
    }
  \end{equation}
  Note that $\L_0(\varphi) = \I_0(\varphi)$.  Let
  $\varphi: \ \rset^d \to \rset$. Then, if $\L_t(\abs{\varphi}) <+\infty$, we
  define $\L_t(\varphi)$ similarly to \eqref{eq:l_t}. The following proposition
  establishes that $t \mapsto \L_t(\varphi)$ is Lipschitz under mild regularity
  conditions. We emphasize that this proposition is no longer true if
  $d \leq p$. Indeed, let us consider the following counterexample. Let
  $F: \ \rset \to \rset^2$ given for any $x \in \rset$ by
  $F(x) = ((1-x^2)/(1+x^2), x(1-x^2)/(1+x^2))$. The set $F(\rset)$ defines a right
  strophoid% , see \Cref{fig:strophoid} for a graph
  . Then, for any $t \neq 0$ we have $\calH^0(F^{-1}(t)) = 1$ or $0$, but
  $\calH^0(F^{-1}(0)) = 2$ and therefore $\L_t(\mathrm{J}F/\Psi) = 1$ or $0$
  near $t=0$ but $\L_0(\mathrm{J}F/\Psi) = 2$ with $\Psi=1$. Hence
  $t \mapsto \L_t(\mathrm{J}F/\Psi)$ is not even continuous.

\begin{proposition}
  \label{prop:lip_l}
  Assume \rref{assum:F}, \textup{\rref{assum:psi}} and that $d \geq p$. Let
  $\msu \subset \rset^d$ be open and such that $F^{-1}(0) \subset \msu$, and let
  $\varphi \in \rmc(\bar{\msu}, \rset)$.  Then
  $\lim_{t \to 0} \L_t(\varphi) = \L_0(\varphi)$. In addition, assume that
  $\Psi, \varphi \in \rmc^1(\bar{\msu}, \rset)$. Then there exist $B_0 \geq 0$
  and $\eta >0$ such that for any $t \in \cball{0}{\eta}$,
  \begin{equation}
    \abs{\L_t(\varphi) - \L_0(\varphi)} \leq B_0 (1 + M_{0, \varphi} + M_{1, \varphi})(1 + M_{0, \Psi} + M_{1, \Psi}) \norm{t} \eqsp ,
  \end{equation}
  with for any $i \in \{0,1\}$ and $f \in \rmc^1(\rset^d, \rset)$,
  $M_{i, f} = \sup \ensembleLigne{\normLigne{\nabla^i f(x)}}{x \in
    F^{-1}(\cball{0}{\eta})}$, $B_0$ and $\eta$ do not depend on $\varphi$ and
  $\Psi$, and $F^{-1}(\cball{0}{\eta}) \subset \msu$.
\end{proposition}

\begin{proof}
  First, we show that there exists an explicit diffeomorphism between
  $F^{-1}(t)$ and $F^{-1}(0)$ for $\normLigne{t}$ small enough.  Then we use the
  coarea formula to express $\L_t$ as an integral over $F^{-1}(0)$ and use the
  dominated convergence theorem to conclude the first part of the proof. For the
  second part of the proof we differentiate the diffeomorphism w.r.t. the
  parameter $t$ and provide explicit bounds for the derivative.
  \begin{enumerate}[wide, labelwidth=!, labelindent=0pt, label=(\alph*)]
  \item The set $F^{-1}(0)$ is compact since
    $\lim_{\norm{x} \to +\infty} \norm{F(x)} =+\infty$. First, there exists
    $\eta_0 > 0$ such that $F^{-1}(\cballinfty{0}{\eta_0}) \subset \msu$. Indeed,
    since $F^{-1}(0)$ is compact, there exists $\vareps_0 > 0$ such that
    $F^{-1}(0)+\cballinfty{0}{\vareps_0} \subset \msu$. We now show that for any
    $\vareps > 0$, there exists $\eta_\vareps > 0$ such that
    $F^{-1}(\cballinfty{0}{\eta_\vareps}) \subset F^{-1}(0) +
    \cballinfty{0}{\vareps}$. If this is false, we let $\vareps > 0$ such that
    for any $\eta > 0$,
    $F^{-1}(\cballinfty{0}{\eta}) \not\subset F^{-1}(0) +
    \cballinfty{0}{\vareps}$. Hence there exists a sequence
    $(x_k)_{k \in \nset} \in (\rset^d)^\nset$ such that
    $\normLigne{F(x_k)}_\infty \leq 1/(k+1)$ and
    $d(x_k, F^{-1}(0)) \geq \vareps$. But, up to taking a subsequence, there exists
    $x^\star \in F^{-1}(\cballinfty{0}{1})$ such that
    $\lim_{k \to +\infty} x_k = x^\star$. Then, we have $F(x^\star) =0$ and
    $d(x^\star, F^{-1}(0))>\vareps$, which is absurd.  Hence for any
    $\vareps > 0$, there exists $\eta_\vareps > 0$ such that
    $F^{-1}(\cballinfty{0}{\eta_\vareps}) \subset F^{-1}(0) +
    \cballinfty{0}{\vareps}$.  We let $\eta_0 =\eta_{\vareps_0}$.

    Second, there exists $\eta_1 > 0$ such that for any
    $x \in F^{-1}(\cballinfty{0}{\eta_1})$, $\jac{x} > 0$. Indeed, if this is
    not the case then there exists $(x_k)_{k \in \nset}$ with
    $\lim_{k \to +\infty} F(x_k) = 0$ and $\jac{x_k} =0$. Since
    $F^{-1}(\cballinfty{0}{1})$ is compact there exists $x^\star$ such that, up to taking a subsequence, $\lim_{k \to +\infty} x_k = x^\star$. Then $F(x^\star)=0$ and
    $\jac{x^\star}=0$, which is absurd.  We define
    $\msk_0 = F^{-1}(\cballinfty{0}{\eta_1})$ and
    $\msk_1 = F^{-1}(\cballinfty{0}{\eta})$ with
    $\eta = \min(\eta_1/2, \eta_0)$. Note that $\msk_1 \subset \msu$ and
    $\msk_1 \subset \mathrm{int}(\msk_0)$.
    
    In what follows for any $x \in \rset^d$ we define
    $G(x) = \rmD F(x) \rmD F(x)^\top$. Note that for any $x \in \rset^d$,
    $\det(G(x))^{1/2} = \jac{x}$ since $d \geq p$. We also have that for any
    $x \in \msk_0$, $G(x)$ is invertible since $\jac{x} > 0$. In addition, we
    have that for any $i, j \in \{1, \dots, p\}$ and $x \in \rset^d$
    \begin{equation}
      G_{i,j}(x) = \langle \nabla F_i(x), \nabla F_j(x) \rangle \eqsp .
    \end{equation}
    We define $\{f_i\}_{i=1}^p$ such that for any $i \in \{1, \dots, p\}$,
    $f_i:\ \rset^d \to \rset^d$ is given for any $x \in \msk_0$ by
    \begin{equation}
      \label{eq:def_f}
      \textstyle{
        f_i(x) = \sum_{k=1}^p h_{i, k}(x) \nabla F_k(x) \eqsp ,
        }
    \end{equation}
    with $\{h_{i,j}(x)\}_{1 \leq i,j \leq p} = G(x)^{-1}$.
    For any $x \in \msk_0$ and $i, j \in \{1, \dots, p\}$ we have
    \begin{equation}
      \textstyle{
        \langle f_i(x), \nabla F_j(x) \rangle = \sum_{k=1}^p h_{i,k}(x) \langle \nabla F_k(x), \nabla F_j(x) \rangle = \updelta_i(j) \eqsp ,
        }
    \end{equation}
    where $\updelta_i$ is the Dirac mass at $i$.  In what follows, we let
    $\{g_i\}_{i=1}^p$ such that for any $i \in \{1, \dots, p\}$,
    $g_i: \ \rset^d \to \rset^d$ and $g_i \in \rmc^1(\rset^d, \rset^d)$ such
    that $g_i(x) = f_i(x)$ for any $x \in \msk_1$, and $g_i(x)=0$ for
    $x \in \mathrm{int}(\msk_0)^\complementary$, such functions exist using
    Whitney extension theorem for instance, see \cite{whitney1934analytic}. In
    what follows, we fix $t = (t_1, \dots, t_p) \in \cballinfty{0}{\eta}$. For
    any $i \in \{1, \dots, p\}$ let $\Phi_i: \ \rset \times \rset^d \to \rset^d$
    given by $\Phi_i(0, x) = x$ for any $x \in \rset^d$ and for any
    $s \in \rset$ and $x \in \rset^d$
    \begin{equation}
      \label{eq:flow}
      \partial_s \Phi_i(s,x) = -g_i(\Phi_i(s,x)) \eqsp . 
    \end{equation}
    For any $i \in \{1, \dots, p\}$, $\Phi_i$ is well-defined using
    \Cref{lemma:flow}.  Therefore, we have for any $i \in \{1, \dots, p\}$ and
    $s \in \rset$, $x \in \rset^d$ such that $\Phi_i(s,x) \in \msk_1$
    \begin{equation}
      \partial_s F(\Phi_i(s,x)) = -(\langle g_i(\Phi_i(s,x)), \nabla F_1(\Phi_i(s,x)) \rangle, \dots, \langle g_i(\Phi_i(s,x)), \nabla F_p(\Phi_i(s,x)) \rangle) = -e_i \eqsp , 
    \end{equation}
    where we recall that $\{e_i\}_{i=1}^p$ is the canonical basis of $\rset^p$.
    We define $\bar{\Phi}_t: \ \rset^d \to \rset^d$ such that for any
    $x \in \rset^d$, $\bar{\Phi}_t(x) = x^{(p)}$ with $x^{(0)} = x$ and for any
    $i \in \{0, \dots, p-1\}$, $x^{(i+1)} = \Phi_{i+1}(t_{i+1}, x^{(i)})$. Note that
    $\bar{\Phi}_t \in \rmc^1(\rset^d, \rset)$ and is a diffeomorphism, see
    \Cref{lemma:flow}. Using \eqref{eq:flow} we have that
    $\bar{\Phi}_t(F^{-1}(t)) = F^{-1}(0)$.  
    In addition, $F^{-1}(t)$ is
    $\calH^{d-p}$ countably rectifiable using
    \Cref{prop:rectifiable_level_set}.  Using this result and the coarea
    formula, see \Cref{thm:area_coarea}, we have
    \begin{equation}
      \textstyle{
      \L_t(\varphi) = \int_{F^{-1}(0)} \varphi(\bar{\Phi}_t^{-1}(x)) \Psi(\bar{\Phi}_t^{-1}(x)) \jacinv{\bar{\Phi}_t^{-1}(x)} \absLigne{\det(\rmD \bar{\Phi}_t^{-1}(x))} \rmd \calH^{d - p}(x) \eqsp . }
    \end{equation}
    Since $F^{-1}(0) \times \ballinfty{0}{\eta}$ is compact and
    $(t,x) \mapsto \bar{\Phi}_t^{-1}(x)$ and
    $(t,x) \mapsto \rmD \bar{\Phi}_t^{-1}(x)$ are continuous with for any $x \in \rset^d$, 
    $\bar{\Phi}_0^{-1}(x) = x$ and $\rmD \bar{\Phi}_0^{-1}(x) = \Id$, we get
    that $\lim_{t \to 0} \L_t(\varphi) = \L_0(\varphi)$ using the dominated
    convergence theorem.
  \item For the second part of the proof we  control the derivative of
    $t \mapsto \chi(t,x)$ where for any $x \in \rset^d$ and $t \in \rset^p$ we have
    \begin{equation}
      \label{eq:Psi_def_nul}
      \chi(t,x) =  \varphi(\bar{\Phi}_t^{-1}(x)) \Psi(\bar{\Phi}_t^{-1}) \jacinv{\bar{\Phi}_t^{-1}(x)} \absLigne{\det(\rmD \bar{\Phi}_t^{-1}(x))} \eqsp . 
    \end{equation}
   Using \Cref{lemma:derivative_control}, there exists
    $P \in \poly{4}{\rset_+}$ such that for $x \in F^{-1}(0)$ and
    $t \in \ballinfty{0}{\eta}$ we have
    \begin{align}
      \norm{\partial_t \chi(t,x)} &\leq (1 + M_{0,\varphi}+ M_{1,\varphi})(1 + M_{0,\Psi}+ M_{1,\Psi}) \\
      & \quad \times P(M_{1,F}, M_{2,F}, M_{3,F}, 1/m_{1,F}) \exp[P(M_{1,F}, M_{2,F}, M_{3,F}, 1/m_{1,F})] \eqsp ,
    \end{align}
    with $P$ that does not depend on $\varphi$ and $\Psi$. 
    Hence we have that for any $t \in \ballinfty{0}{\eta}$
    \begin{align}
      \abs{\L_t(\varphi) - \L_0(\varphi)} &\leq \textstyle{\int_{F^{-1}(0)} \abs{\chi(t,x) - \chi(0, x)} \rmd \calH^{d - p}(x)}
      \\ &\leq (1 + M_{0,\varphi}+ M_{1,\varphi})(1 + M_{0,\Psi}+ M_{1,\Psi}) P(M_{1,F}, M_{2,F}, M_{3,F}, 1/m_{1,F}) \\
                        & \qquad \exp[P(M_{1,F}, M_{2,F}, M_{3,F}, 1/m_{1,F})] \calH^{d - p}(F^{-1}(0)) \norm{t} \eqsp ,
    \end{align}
    which concludes the proof since $\calH^{d - p}(F^{-1}(0))< +\infty$ by \Cref{prop:rectifiable_level_set}.
  \end{enumerate}
\end{proof}

{The smoothness of $F$ can be relaxed to
  $\rmc^3(\rset^d, \rset^p)$ (at least). From our analysis,
  $F \in \rmc^2(\rset^d, \rset^p)$ (or at least $F \in \rmc^1(\rset^d, \rset^p)$
  with Lipschitz derivative) seems to be necessary to obtain quantitative
  results.}

\begin{proposition}
  \label{prop:upper_bound_i_in}
  Assume \rref{assum:F}, \textup{\rref{assum:psi}} and that $d \geq p$. Let
  $\msu \subset \rset^d$ open and bounded such that $F^{-1}(0) \subset \msu$ and
  $\varphi \in \rmc(\bar{\msu}, \rset)$.  Then there exists $\eta > 0$ such that
  $\lim_{\vareps \to 0} \I_\vareps^{\mathrm{in}}(\varphi) = \I_0(\varphi)$,
  where for any $\vareps > 0$,
  $\I_\vareps^{\mathrm{in}}(\varphi) = \I_{\vareps}(\varphi
  \1_{F^{-1}(\ball{0}{\eta})})$. In addition, assume that
  $\varphi, \Psi \in \rmc^1(\bar{\msu}, \rset)$, then there exists $A_2 \geq 0$
  such that for any $\vareps > 0$
  \begin{equation}
    \textstyle{
      \abs{\I_{\vareps}^{\mathrm{in}}(\varphi) - \I_0(\varphi)} \leq A_2 (1 + M_{0, \varphi} + M_{1, \varphi}) (1 + M_{0, \Psi} + M_{1, \Psi}) \vareps^{1/k} \eqsp ,
      }
  \end{equation}
  with for any $i \in \{0,1\}$ and $f \in \rmc^1(\rset^d, \rset)$,
  $M_{i, f} = \sup \ensembleLigne{\normLigne{\nabla^i f(x)}}{x \in
    F^{-1}(\cball{0}{\eta})}$, $A_2$ and $\eta$ do not depend on 
  $\varphi$ and $\Psi$, and $F^{-1}(\cball{0}{\eta}) \subset \msu$.
\end{proposition}

\begin{proof}
  First, note that $\I_0(\varphi) = \L_0(\varphi)$, see \eqref{eq:i_0_int} and
  \eqref{eq:l_t}. In what follows, we let $\eta > 0$ be given by
  \Cref{prop:lip_l} and define
  $\I_0^{\mathrm{in}}(\varphi) = C_{\vareps}^{-1} \int_{\ball{0}{\eta}}
  \exp[-\norm{t}^k/\vareps] \L_0(\varphi) \rmd t$.  We have
  \begin{equation}
    \label{eq:in_out_i0}
    \textstyle{
    \absLigne{\I_0(\varphi) - \I_0^{\mathrm{in}}(\varphi)} = \abs{\I_0(\varphi)} C_{\vareps}^{-1} \int_{\ball{0}{\eta}^\complementary} \exp[-\norm{t}^k/\vareps] \rmd t \leq 2^{d/k} \abs{\I_0(\varphi)} \exp[-\eta^k/(2\vareps)]\eqsp . }
  \end{equation}
  Using the coarea formula, see \Cref{thm:area_coarea}, we have for any $\vareps > 0$,
  \begin{equation}
    \textstyle{\I_\vareps^{\mathrm{in}}(\varphi) = C_\vareps^{-1} \int_{F^{-1}(\ball{0}{\eta})} \Psi(x) \varphi(x) \exp[-\norm{F(x)}^k/\vareps] \rmd x = C_\vareps^{-1} \int_{\ball{0}{\eta}} \exp[-\norm{t}^k/\vareps] \L_t(\varphi) \rmd t \eqsp . }
  \end{equation}
  Therefore, using this result and the change of variable
  $t \mapsto \vareps^{1/k} t$ we have
  \begin{align}
    \absLigne{\I_\vareps^{\mathrm{in}}(\varphi) - \I_0^{\mathrm{in}}(\varphi)} &\textstyle{\leq  C_\vareps^{-1} \int_{\ball{0}{\eta}} \exp[-\norm{t}^k/\vareps] \absLigne{\L_t(\varphi) - \L_0(\varphi)} \rmd t }\\
                                                                                      &\leq \textstyle{ C_1^{-1} \int_{\ball{0}{\eta/\vareps^{1/k}}} \exp[-\norm{t}^k] \absLigne{\L_{t \vareps^{1/k}}(\varphi) - \L_0(\varphi)} \rmd t \eqsp .  } \label{eq:in}
  \end{align}
  Hence, we get that
  $\lim_{\vareps \to 0} \absLigne{\I_\vareps^{\mathrm{in}}(\varphi) -
    \I_0^{\mathrm{in}}(\varphi)} =0$ using the dominated convergence theorem and
  that
  $\lim_{\vareps \to 0} \absLigne{\L_{t \vareps^{1/k}}(\varphi) - \L_0(\varphi)}
  =0$ according to \Cref{prop:lip_l}. This concludes the first part of the proof
  upon combining this result and \eqref{eq:in_out_i0}. In addition, assume that
  $\varphi \in \rmc^1(\bar{\msu}, \rset)$ then using the second part of
  \Cref{prop:lip_l} and \eqref{eq:in} we have
  \begin{equation}
    \abs{\I_\vareps^{\mathrm{in}}(\varphi) - \I_0^{\mathrm{in}}(\varphi)} \leq \textstyle{B_0C_1^{-1} \int_{\rset^p} \norm{t} \exp[-\norm{t}^k] \rmd t  (1 + M_{0,\varphi} + M_{1, \varphi}) (1 + M_{0,\Psi} + M_{1, \Psi}) \vareps^{1/k} \eqsp ,}
  \end{equation}
  which concludes the proof upon combining this result and \eqref{eq:in_out_i0}.
\end{proof}

\subsubsection{The case $d \leq p$}
\label{sec:case-d-leq}

We now turn to the case $d \leq p$. The proof of this result is more classical
and does not rely on geometric measure theory. Instead we build on the Morse
theory approach for Laplace approximation, see \cite{wong2001asymptotic} for
example.  The following proposition is a quantitative extension of \cite[Theorem
3, p.495]{wong2001asymptotic}.

\begin{proposition}
  \label{prop:upper_bound_i_laplace}
  Assume \rref{assum:F}, \textup{\rref{assum:psi}} and that $d \leq p$. Let
  $\msu \subset \rset^d$ open and bounded such that $F^{-1}(0) \subset \msu$ and
  $\varphi \in \rmc(\bar{\msu}, \rset)$.  Then
  $\lim_{\vareps \to 0} \absLigne{\I_\vareps^{\mathrm{in}}(\varphi) -
    \I_0(\varphi)} = 0$, with
  $\I_\vareps^{\mathrm{in}}(\varphi) = \I_{\vareps}(\varphi \1_{\msv})$ and
  $\msv$ open such that $F^{-1}(0) \subset \msv \subset \msu$.  In addition,
  assume that $\varphi, \Psi \in \rmc^1(\bar{\msu}, \rset)$. Then there exists
  $B_1 \geq 0$ such that for any $\vareps > 0$ we have
  \begin{equation}
    \absLigne{\I_\vareps^{\mathrm{in}}(\varphi) - \I_0(\varphi)} \leq B_1 (1 + M_{0,\varphi} + M_{1,\varphi}) (1 + M_{0,\Psi} + M_{1,\Psi}) \vareps^{1/k} \eqsp , 
  \end{equation}
 with $B_1$
  that does not depend on $\varphi$ and $\Psi$, and for any $i \in \{0,1\}$ and
  $f \in \rmc^1(\bar{\msu}, \rset)$,
  $M_{i, f} = \sup \ensembleLigne{\normLigne{\nabla^i f(x)}}{x \in \msu}$.
\end{proposition}

\begin{proof}
  Let $\{x_0^\ell\}_{\ell=1}^N$ and $\{\msw_\ell\}_{\ell=1}^N$ be given by
  \Cref{lemma:existence_U} such that $F^{-1}(0) = \cup_{\ell=1}^N \{x_0^\ell\}$
  and $\rmd F(x)$ is injective for any $\ell \in \{1, \dots, N\}$ and
  $x \in \msw_\ell$.  In addition, for any $\ell, m \in \{1, \dots, N\}$,
  $\msw_\ell \cap \msw_m = \emptyset$. Let $\ell \in \{1, \dots, N\}$ and
  $U: \ \rset^d \to \rset_+$ such that for any $x \in \rset^d$,
  $U(x) = \normLigne{F(x)}^2$.  Since $F \in \rmc^\infty(\rset^d, \rset^p)$ we
  have that $U \in \rmc^\infty(\rset^d, \rset)$.  We divide the rest of the
  proof into two parts.
  \begin{enumerate}[wide, labelwidth=!, labelindent=0pt, label=(\alph*)]
  \item First, we have that
    $\nabla^2U(x_0^\ell) = 2 \rmD F(x_0^\ell)^\top \rmD F(x_0^\ell)$ which is
    invertible since $\jac{x_0^\ell} > 0$. Therefore, we can apply Morse's lemma
    \citep[Theorem 3.1.1]{nirenberg2001topics} and there exists a diffeormorphism
    $\Phi_\ell \in \rmc^1(\bar{\Omega}_\ell, \bar{\msw}_\ell)$ with
    $0 \in \Omega_\ell$, $x_0^\ell \in \msw_\ell$ and
    $\Omega_\ell \subset \rset^d$ open such that for any $x \in \Omega_\ell$,
    $U(\Phi_\ell(x)) = \norm{x}^2$, $\Phi_\ell(0) = x_0^\ell$ and
    $\rmD \Phi_\ell(0) = (\rmD F(x_0^\ell)^\top \rmD F(x_0^\ell))^{-1/2}$. Note
    that $\det(\rmD \Phi_\ell(0)) = \jacinv{x_0^\ell}$.

    Let $r_\ell > 0$ such that $\cball{0}{r_\ell} \subset \Omega_\ell$. We have 
  \begin{equation}
    \label{eq:expe_I0}
    \textstyle{
    1 - \int_{\Omega_\ell/\vareps^{1/k}} \exp[-\norm{x}^k] \rmd x /C_1 \leq \int_{\rset^d} \exp[-\norm{x}^k/2] \rmd x \exp[-r_\ell^k/(2\vareps)]/C_1 \eqsp . }
\end{equation}
In what follows, we no longer consider $\ell \in \{1, \dots, N\}$ to be
fixed.  Let $\msv = \cup_{\ell=1}^N \Phi_\ell(\Omega_\ell)$ and
$\I_{0, \vareps}^{\mathrm{in}}(\varphi) = \sum_{\ell=1}^N
\int_{\Omega_\ell/\vareps^{1/k}} \exp[-\norm{x}^k] \rmd x \varphi(x_0^\ell)
\Psi(x_0^\ell) \jacinv{x_0^\ell}/C_1$. We recall that we have
  \begin{equation}
    \textstyle{
    \I_0(\varphi) = \sum_{\ell=1}^N
    \varphi(x_0^\ell)
    \Psi(x_0^\ell) \jacinv{x_0^\ell} \eqsp . }
  \end{equation}
  Combining this result and \eqref{eq:expe_I0} we get 
  \begin{equation}
    \label{eq:out_i_0}
    \textstyle{
    \abs{\I_{0, \vareps}^{\mathrm{in}}(\varphi) - \I_0(\varphi)} \leq N M \int_{\rset^d} \exp[-\norm{x}^k/2] \rmd x \exp[-r_{\mathrm{min}}^k/(2\vareps)] / C_1 \eqsp ,}
  \end{equation}
  where
  $r_{\mathrm{min}} = \min \ensembleLigne{r_\ell}{\ell \in \{1, \dots, N\}}$ and
  $M = \max \ensembleLigne{\absLigne{\varphi(x_0^\ell)} \Psi(x_0^\ell)
    \jacinv{x_0^\ell}}{\ell \in \{1, \dots, N\}}$.  Using for any
  $\ell \in \{1, \dots, N\}$ the change of variable $x \mapsto \Phi_\ell(x)$ and
  $x \mapsto \vareps^{1/k} x$ we have
  \begin{align}
    \abs{\I_{0, \vareps}^{\mathrm{in}}(\varphi) - \I_\vareps^{\mathrm{in}}(\varphi)}  &\textstyle{\leq \sum_{\ell=1}^N\int_{\Omega_\ell/\vareps^{1/k}} |  \varphi(\Phi_\ell(\vareps^{1/k} x)) \Psi(\Phi_\ell(\vareps^{1/k} x)) \det(\rmD \Phi_\ell(\vareps^{1/k} x))}   \\
    & \qquad \qquad   - \varphi(x_0^\ell) \Psi(x_0^\ell) \jacinv{x_0^\ell}  |  \exp[-\norm{x}^k]\rmd x/C_1 \eqsp ,  \label{eq:diff_laplace}
  \end{align}
  For any $\ell \in \{1, \dots, N\}$, let $\chi_\ell : \ \Omega_\ell \to \rset$
  given for any $x \in \Omega_\ell$ by
  \begin{equation}
    \chi_\ell(x) = \varphi(\Phi_\ell( x)) \Psi(\Phi_\ell(x)) \det(\rmD \Phi_\ell(x)) \eqsp . 
  \end{equation}
  We conclude the first part of the proof using \eqref{eq:out_i_0}, the
  dominated convergence theorem in \eqref{eq:diff_laplace} and that for any
  $\ell \in \{1, \dots, N\}$, $\chi \in \rmc(\Omega_\ell, \rset)$.
  \item For the second part of the proof, since
  $\varphi, \Psi \in \rmc^1(\bar{\msu}, \rset)$ and for any
  $\ell \in \{1, \dots, N\}$,
  $\Phi_\ell \in \rmc^2(\bar{\Omega}_\ell, \bar{\msw}_\ell)$ since
  $F \in \rmc^\infty(\rset^d, \rset^p)$ using \cite[Theorem
  3.1.1]{nirenberg2001topics}, we have that for any $\ell \in \{1, \dots, N\}$,
  $\chi_\ell \in \rmc^1(\bar{\Omega}_\ell, \rset)$ and there exists
  $B_1' \geq 0$ (which do not depend on $\varphi$ and $\Psi$) such that for any
  $\ell \in \{1, \dots, N\}$ and $x \in \Omega_\ell$ we have
  \begin{equation}
    \norm{\rmD \chi_\ell(x)} \leq B_1' (1 + M_{0,\varphi} + M_{1,\varphi}) (1 + M_{0,\Psi} + M_{1,\Psi}) \eqsp .
  \end{equation}
  Using this result and \eqref{eq:diff_laplace} we get that
  \begin{equation}
    \absLigne{\I_{0, \vareps}^{\mathrm{in}}(\varphi) - \I_\vareps^{\mathrm{in}}(\varphi)} \leq \textstyle{N  B_1' (1 + M_{0,\varphi} + M_{1,\varphi}) \vareps^{1/k} \int_{\rset^d} \norm{x} \exp[-\norm{x}^k] \rmd x / C_1 \eqsp .}
  \end{equation}
  Combining this result and \eqref{eq:out_i_0} concludes the proof.
\end{enumerate}
\end{proof}

\subsubsection{Proof of \Cref{thm:big_theo_extension}}
\label{sec:proof-crefthm:big}

We start by proving the results of \Cref{thm:big_theo_extension} in a smooth setting, then we deduce the general case using a smoothing lemma.

\begin{proposition}
  \label{prop:final_d_geq_p}
  Assume \rref{assum:F} and 
  \textup{\rref{assum:psi}}. Let $\msu \subset \rset^d$ open and bounded such that
  $F^{-1}(0) \subset \msu$ and $\varphi \in \rmc(\bar{\msu}, \rset)$ which
  satisfies \eqref{eq:cond_varphi}. Then,
  $\lim_{\vareps \to 0} \abs{\pi_{\vareps}^\Psi[\varphi] - \pi_0^\Psi[\varphi]}
  =0$. In addition, assume that $\varphi, \Psi \in \rmc^1(\bar{\msu},
  \rset)$. Then there exists $A \in \rmc(\rset_+^3, \rset_+)$ such that for any
  $\vareps \in \oointLigne{0, \mtt^k / (4\{C_\varphi + C_\Psi +
    C_{\varphi}C_\Psi + 1\})}$
  \begin{equation}
    \abs{\pi_{\vareps}^\Psi[\varphi] - \pi_0^\Psi[\varphi]} \leq A(C_\varphi, C_\Psi, m_{0,\Psi}) (1 + M_{0,\varphi} + M_{1,\varphi}) (1 + M_{0,\Psi} + M_{1,\Psi}) \vareps^{1/k} \eqsp ,
  \end{equation}
  with for any $i \in \{0,1\}$ and $f \in \rmc^1(\rset^d, \rset)$,
  $M_{i, f} = \sup \ensembleLigne{\normLigne{\nabla^i f(x)}}{x \in \msu}$,
  $m_{0,\Psi} = \inf \ensembleLigne{\Psi(x)}{x \in F^{-1}(0)}$ and $A$
 function that does not depend on $\varphi$ and $\Psi$. Finally, $A$ is non-decreasing w.r.t. its 
  first two variables and non-increasing w.r.t. its last variable.
\end{proposition}

\begin{proof}
Let
  $\bvareps = \mtt^k/(4 + 4 C_{\varphi, \Psi})$, with
  $C_{\varphi, \Psi} = (C_\Psi + 1)(C_\varphi + 1)$ and note that we have
  \begin{equation}
    \label{eq:bound_vareps}
 \bvareps < \min(\mtt^k/(1 + C_{\varphi, \Psi}), \mtt^k/(1 + C_{\varphi, 1})) \eqsp . 
\end{equation}
  For any $\vareps \in \ooint{0, \bvareps}$ we have 
  \begin{align}
    \label{eq:decompopo}
    \abs{\pi_\vareps^\Psi[\varphi] - \pi_0^\Psi[\varphi]} &\leq \abs{\I_\vareps(\varphi)/\J_\vareps - \I_0(\varphi)/\J_0} \\
    &\leq \J_\vareps^{-1} \abs{\I_\vareps(\varphi) - \I_0(\varphi)} + \I_0(\varphi)/(\J_0\J_\vareps) \abs{\I_\vareps(1) - \I_0(1)} \eqsp . 
  \end{align}
  Let $\eta > 0$ be given by \Cref{prop:lip_l} and let
  $\msv = F^{-1}(\ball{0}{\eta})$ if $d \geq p$, and $\msv$ given by
  \Cref{prop:upper_bound_i_laplace} otherwise.  Note that
  $F^{-1}(0) \subset \msv$. For any $\vareps \in \ooint{0, \bvareps}$ we define
  $\I_\vareps^{\mathrm{out}}(\varphi) = \I_\vareps(\varphi
  \1_{\msv^\complementary})$ and
  $\I_\vareps^{\mathrm{in}}(\varphi) = \I_\vareps(\varphi \1_{\msv})$.  
We divide the rest of the proof into two parts. First, we control
$\I_\vareps^{\mathrm{out}}(\varphi)$ using the technical bounds of
\Cref{prop:upper_bound_i_out}. Then, we control
$\absLigne{\I_\vareps^{\mathrm{in}}(\varphi) - \I_0(\varphi)}$ using either
\Cref{prop:upper_bound_i_in} if $d \geq p$ or \Cref{prop:upper_bound_i_laplace}
if $d \leq p$. We conclude upon combining these results.

  \begin{enumerate}[wide, labelwidth=!, labelindent=0pt, label=(\alph*)]
  \item Using \eqref{eq:bound_vareps} and \Cref{prop:upper_bound_i_out} we get that
  for any $\vareps \in \oointLigne{0, \bvareps}$ we have
  \begin{equation}
    \label{eq:ineq_out}
    \I_{\vareps}^{\mathrm{out}}(\varphi) \leq A_1(C_{\varphi, \Psi}) \vareps^{-d/k} \exp[-\beta_1 / \vareps]\eqsp , \qquad \I_{\vareps}^{\mathrm{out}}(1) \leq A_1(C_{1, \Psi}) \vareps^{-d/k} \exp[-\beta_1 / \vareps]\eqsp ,
  \end{equation}
  with $\beta_1 > 0$ and $A_1 \in \rmc(\rset_+, \rset_+)$ that do not depend on
  $\varphi$ and $\Psi$, and are non-decreasing.
\item Using either \Cref{prop:upper_bound_i_in} if $d \geq p$ or
  \Cref{prop:upper_bound_i_laplace} if $d \leq p$ we have that for any
  $\vareps \in \oointLigne{0, \bvareps}$
  \begin{align}
\label{eq:ineq_in}
    &\abs{\I_{\vareps}^{\mathrm{in}}(\varphi) - \I_0(\varphi)} \leq A_2 (1 + M_{0, \varphi} + M_{1, \varphi}) (1 + M_{0, \Psi} + M_{1, \Psi}) \vareps^{1/k} \eqsp , \\
            &\abs{\I_{\vareps}^{\mathrm{in}}(1) - \I_0(1)} \leq 2 A_2 (1 + M_{0, \Psi} + M_{1, \Psi}) \vareps^{1/k} \eqsp ,
  \end{align}
  with for any $i \in \{0,1\}$ and $f \in \rmc^1(\rset^d, \rset)$,
  $M_{i, f} = \sup \ensembleLigne{\normLigne{\nabla^i f(x)}}{x \in \msu}$, and $A_2$
  that does not depend on $\varphi$ and $\Psi$.
\end{enumerate}
Combining \eqref{eq:ineq_out} and
  \eqref{eq:ineq_in} we get that for any $\vareps \in \oointLigne{0, \bvareps}$
  \begin{align}
    &\abs{\I_{\vareps}(\varphi) - \I_0(\varphi)} \leq A_2 (1 + M_{0, \varphi} + M_{1, \varphi}) (1 + M_{0, \Psi} + M_{1, \Psi}) \vareps^{1/k} + A_1(C_{\varphi, \Psi}) \vareps^{-d/k} \exp[-\beta_1 / \vareps] \eqsp , \\
            &\abs{\I_{\vareps}(1) - \I_0(1)} \leq 2 A_2 (1 + M_{0, \Psi} + M_{1, \Psi}) \vareps^{1/k} + A_1(C_{1, \Psi}) \vareps^{-d/k} \exp[-\beta_1 / \vareps] \eqsp . 
  \end{align}
{Since for any $\vareps > 0$, $\vareps^{-d/k} \exp[-\beta/\vareps] \leq ((d+1)/(k\beta ))^{1/k + d/k} \vareps^{1/k}$ (this inequality comes from first multiplying by $\vareps^{-1/k}$ on both sides, and then taking the value of $\vareps$ that achieves the maximum of the left-hand term)}, there exists $\tilde{A} \in \rmc(\rset_+^2, \rset_+)$ such that for any $\vareps \in \oointLigne{0, \bvareps}$ we have
\begin{align}
  \label{eq:inter_bientot}
      &\abs{\I_{\vareps}(\varphi) - \I_0(\varphi)} \leq \tilde{A}(C_\varphi, C_\Psi) (1 + M_{0, \varphi} + M_{1, \varphi}) (1 + M_{0, \Psi} + M_{1, \Psi}) \vareps^{1/k} \eqsp , \\
            &\abs{\I_{\vareps}(1) - \I_0(1)} \leq 2 \tilde{A}(1, C_\Psi) (1 + M_{0, \Psi} + M_{1, \Psi}) \vareps^{1/k} \eqsp , 
\end{align}
with $\tilde{A}$ that does not depend on $\Psi$ and $\varphi$, and $\tilde{A}$
non-decreasing w.r.t. to each of its variables.  Using
\Cref{prop:lower_bound_j_eps}, there exists
$A_0 \in \rmc(\rset_+, \rset_+^\star)$ such that for any
$\vareps \in \ccint{0, \bvareps}$, $\J_{\vareps} \geq A_0(m_{0,\Psi})$ with
$m_{0,\Psi} = \inf \ensembleLigne{\Psi(x)}{x \in F^{-1}(0)}$, and $A_0$ that does not depend on $\Psi$ and is non-increasing. Finally, note that
{$\I_0(\varphi) \leq \sup \ensembleLigne{\jacinv{x}}{x \in F^{-1}(0)} M_{0,
  \varphi} M_{0,\Psi} \calH^{d-p}(F^{-1}(0))$} with $\calH^{d-p}(F^{-1}(0)) < +\infty$ by
\Cref{prop:rectifiable_level_set}. Combining these results, \eqref{eq:decompopo}
and \eqref{eq:inter_bientot} concludes the proof.
\end{proof}

Using this proposition along with a smoothing lemma, see
\Cref{lemma:smoothing_lemma}, we conclude the proof of
\Cref{thm:big_theo_extension}.

\begin{proof}[Proof of \Cref{thm:big_theo_extension}]
  We begin by introducing the families
  $\ensembleLigne{\varphi^\delta}{\delta \in \ooint{0, \bar{\delta}}}$ and
  $\ensembleLigne{\Psi^\delta}{\delta \in \ooint{0, \bar{\delta}}}$ which are smooth
  approximations of $\varphi$ and $\Psi$ respectively.  Since $F^{-1}(0)$ is
  compact and $\msu^\complementary$ is closed there exists $r > 0$ such that
  $F^{-1}(0) + \ball{0}{r} \subset \msu$. Let
  $\msu_0 = F^{-1}(0) + \ball{0}{r/2}$ and note that
  $F^{-1}(0) \subset \msu_0$ and $\msu_0 + \ball{0}{r/2} \subset \msu$. Let
  $\ensembleLigne{\varphi^\delta}{\delta \in \ooint{0, \bar{\delta}}}$ and
  $\ensembleLigne{\Psi^\delta}{\delta \in \ooint{0, \bar{\delta}}}$ with
  $\bar{\delta} > 0$ given by \Cref{lemma:smoothing_lemma}, $\msv \leftarrow \msu$
  and $\msu \leftarrow \msu_0$. Then using the dominated convergence theorem and
  the fact that $F^{-1}(0)$ is compact we have that
  $\lim_{\delta \to 0} \pi_0^{\Psi^\delta}[\varphi^\delta] =
  \pi_0^\Psi[\varphi]$. 
  Similarly, using the dominated convergence theorem we get that there exists
  $\bvareps_0 \in \rmc(\rset_+, \rset_+)$ such that for any
  $\vareps \in \oointLigne{0, \bvareps_0(C_\varphi)}$ we have
  $\lim_{\delta \to 0} \pi^{\Psi^\delta}_\vareps[\varphi^\delta] =
    \pi^\Psi_\vareps[\varphi]$.  We
  conclude upon using \Cref{prop:final_d_geq_p} and
  \Cref{lemma:smoothing_lemma}.
 
  \end{proof}

%%% Local Variables:
%%% mode: latex
%%% TeX-master: "main"
%%% End:

\subsection{Proofs of \Cref{sec:from-macr-micr}}
\label{sec:proof-macro}

In this section, we prove \Cref{prop:behavior_theta}.  Similarly to the proof of
\Cref{thm:big_theo_extension} we divide the proof into two parts depending on
whether $d \geq p$ in \Cref{sec:case-d-geq-1} or $d \leq p$ in
\Cref{sec:case-d-leq-1}. Our main result, which is a generalization of
\Cref{prop:behavior_theta} is presented in \Cref{sec:main-result}. We define
$\I_\vareps$ and $\I_0$ as in \eqref{eq:i_vareps_int} and \eqref{eq:i_0_int}.

\subsubsection{The case $d \geq p$}
\label{sec:case-d-geq-1}

Our first result corresponds to an adaptation of \Cref{prop:upper_bound_i_in} to
the case where $\varphi = \normLigne{F}^k$. Indeed, in this case we have that
$\I_0(\varphi) = 0$ and we can tighten our previous results.

\begin{proposition}
  \label{prop:limite_d_geq_p}
  Assume \rref{assum:F}, \textup{\rref{assum:psi}} and that $d \geq p$.  Then
  there exists $\eta > 0$ such that
  $\lim_{\vareps \to 0} \I_\vareps^{\mathrm{in}}(\normLigne{F}^k)/ \vareps =
  \C_k \I_0(1)$, where for any $\vareps > 0$,
  $\I_\vareps^{\mathrm{in}}(\varphi) = \I_{\vareps}(\varphi
  \1_{F^{-1}(\ball{0}{\eta})})$ for any $\varphi \in \rmc(\bar{\msu}, \rset)$
  and
  \begin{equation}
    \textstyle{
    \C_k = \left. \int_{\rset^p} \normLigne{t}^k \exp[-\normLigne{t}^k] \rmd t \middle/ \int_{\rset^p}  \exp[-\normLigne{t}^k] \rmd t \right.} \eqsp . 
  \end{equation}
 \end{proposition}

\begin{proof}
  Let $\eta > 0$ be given by \Cref{prop:lip_l} with $\varphi = 1$. We recall
  that for any $\vareps > 0$ we have that
  \begin{equation}
    \textstyle{
    \I_\vareps^{\mathrm{in}}(\normLigne{F}^k) = \int_{F^{-1}(\ball{0}{\eta})} \normLigne{F(x)}^k \Psi(x) \exp[-\normLigne{F(x)}^k/\vareps] \rmd x / C_\vareps \eqsp . }
  \end{equation}
  Using the coarea formula, see \Cref{thm:area_coarea}, we have for any $\vareps > 0$,
  \begin{equation}
    \textstyle{
      \I_\vareps^{\mathrm{in}}(\normLigne{F}^k)/\vareps  = \int_{\ball{0}{\eta}} (\norm{t}^k/\vareps) \exp[-\norm{t}^k/\vareps] \L_t(1) \rmd t / C_\vareps \eqsp ,
      }
  \end{equation}
  where $\L_t(1)$ is defined in \eqref{eq:l_t} for any $t \in
  \ball{0}{\eta}$. Using the change of variable $t \mapsto \vareps^{1/k} t$ we
  have for any $\vareps > 0$
  \begin{equation}
    \textstyle{
      \I_\vareps^{\mathrm{in}}(\normLigne{F}^k)/\vareps = \int_{\ball{0}{\eta/\vareps^{1/k}}} \norm{t}^k \exp[-\norm{t}^k] \L_{t \vareps^{1/k}}(1) \rmd t / C_1 \eqsp .
      }
  \end{equation}
  For any $\vareps > 0$ let $g_\vareps : \ \rset^p \to \rset$ such that for any
  $t \in \rset^p$,
  $g_\vareps(t) = \L_{t \vareps^{1/k}}(1)
  \1_{\ball{0}{\eta/\vareps^{1/k}}}(t)$. Note that for any $t \in \rset^d$ and
  $\vareps > 0$, $\abs{g_\vareps(t)} \leq
  \sup_{\cball{0}{\eta}}\abs{\L_t(1)}$. In addition, we have that for any
  $t \in \rset^d$, $\lim_{\vareps \to 0} g_\vareps(t) = \L_0(1)$ using
  \Cref{prop:lip_l}. Therefore, we get that
  \begin{equation}
    \textstyle{
      \lim_{\vareps \to 0} \I_\vareps^{\mathrm{in}}(\normLigne{F}^k)/\vareps = \L_0(1) \int_{\rset^p} \norm{t}^k \exp[-\norm{t}^k] \rmd t / C_1 \eqsp .
      }
  \end{equation}
We conclude the proof upon noting that $\I_0(1) = \L_0(1)$ and that $\C_k = \int_{\rset^p} \norm{t}^k \exp[-\norm{t}^k] \rmd t / C_1$.
\end{proof}

\subsubsection{The case $d \leq p$}
\label{sec:case-d-leq-1}

We now adapt \Cref{prop:upper_bound_i_laplace} to the case where
$\varphi = \normLigne{F}^k$.

\begin{proposition}
  \label{prop:upper_bound_i_laplace_exp}
  Assume \rref{assum:F}, \textup{\rref{assum:psi}} and $d \leq p$. Let
  $\msu \subset \rset^d$ open and bounded such that $F^{-1}(0) \subset \msu$.
  Then
  $\lim_{\vareps \to 0} \I_\vareps^{\mathrm{in}}(\normLigne{F}^k)/ \vareps =
  \C_k \I_0(1)$, where for any $\vareps > 0$,
  $\I_\vareps^{\mathrm{in}}(\varphi) = \I_{\vareps}(\varphi \1_{\msv})$ for any
  $\varphi \in \rmc(\bar{\msu}, \rset)$, with $\msv$ open such that
  $F^{-1}(0) \subset \msv \subset \msu$, and
  \begin{equation}
    \textstyle{
      \C_k = \left. \int_{\rset^p} \normLigne{t}^k \exp[-\normLigne{t}^k] \rmd t \middle/ \int_{\rset^p}  \exp[-\normLigne{t}^k] \rmd t \right. \eqsp .
      }
  \end{equation}
\end{proposition}

\begin{proof}
Let $\{x_0^\ell\}_{\ell=1}^N$, $\{\msw_\ell\}_{\ell=1}^N$,   $\{\Phi_\ell\}_{\ell=1}^N$ and $\{\Omega_\ell\}_{\ell=1}^N$ be given as in  the proof of \Cref{prop:upper_bound_i_laplace}.
Let
  $\msv = \cup_{\ell=1}^N \Phi_\ell(\Omega_\ell)$.  Using for any
  $\ell \in \{1, \dots, N\}$ the change of variable $x \mapsto \Phi_\ell(x)$ and
  $x \mapsto \vareps^{1/k} x$ we have
  \begin{equation}
    \textstyle{
    \I_\vareps^{\mathrm{in}}(\normLigne{F}^k)/\vareps = \sum_{\ell=1}^N \int_{\Omega_\ell/\vareps^{1/k}} \normLigne{x}^k \exp[-\norm{x}^k] \Psi(\Phi_\ell(\vareps^{1/k}x)) \det(\rmD \Phi_\ell(\vareps^{1/k} x)) \rmd x / C_1 \eqsp . }
  \end{equation}
  For any $\ell \in \{1, \dots, N\}$ and $\vareps > 0$, let
  $g_{\ell, \vareps} : \ \rset^p \to \rset$ such that for any $\ell \in \{1, \dots, N\}$, $\vareps > 0$ and
  $x \in \rset^p$ we have
  \begin{equation}
    g_{\ell, \vareps}(x) = \Psi(\Phi_\ell(\vareps^{1/k}x)) \det(\rmD \Phi_\ell(\vareps^{1/k} x)) \1_{\Omega_\ell/\vareps^{1/k}}(x) \eqsp . 
  \end{equation}
  Note that for any $\ell \in \{1, \dots, N\}$, $\vareps > 0$ and
  $x \in \rset^p$ we have
  \begin{equation}
   \absLigne{g_{\ell, \vareps}(x)} \leq \sup \ensembleLigne{\abs{\Psi(\Phi_\ell(x))} \abs{\det(\rmD \Phi_\ell( x))}}{\ell \in \{1, \dots, N\},\ \ x \in \Omega_\ell} \eqsp . 
 \end{equation}
 In addition, we have that for any $\ell \in \{1, \dots, N\}$ and
 $x \in \rset^p$,
 $\lim_{\vareps \to 0} g_{\ell, \vareps}(x) = \Psi(x_0^\ell) \jacinv{x_0^\ell}$.
 We conclude upon using the dominated convergence theorem.  
\end{proof}

\subsubsection{Main result}
\label{sec:main-result}

\begin{proposition}
  \label{prop:lim_d_geq_p}
  Assume \rref{assum:F} and \textup{\rref{assum:psi}}. Let
  $\msu \subset \rset^d$ be open, bounded and such that $F^{-1}(0) \subset \msu$.
  Then
  $\lim_{\vareps \to 0} \pi_\vareps^\Psi(\normLigne{F}^k)/ \vareps = \C_k $,
  where
  \begin{equation}
    \textstyle{
    \C_k = \left. \int_{\rset^p} \normLigne{t}^k \exp[-\normLigne{t}^k] \rmd t \middle/ \int_{\rset^p}  \exp[-\normLigne{t}^k] \rmd t \right. \eqsp . }
  \end{equation}
\end{proposition}

\begin{proof}
  Let $\eta > 0$ be given by \Cref{prop:lip_l} with $\varphi = 1$ and let
  $\msv = F^{-1}(\ball{0}{\eta})$ if $d \geq p$. If $d \leq p$, let $\msv$ be
  given by \Cref{prop:upper_bound_i_laplace_exp}. We have that
  $F^{-1}(0) \subset \msv$. For any $\vareps \in \ooint{0, \bvareps}$ we define
  $\I_\vareps^{\mathrm{out}}(\normLigne{F}^k) = \I_\vareps(\normLigne{F}^k
  \1_{\msv^\complementary})$ and
  $\I_\vareps^{\mathrm{in}}(\normLigne{F}^k) = \I_\vareps(\normLigne{F}^k
  \1_{\msv})$.  Using \Cref{prop:upper_bound_i_out}, we have that
  $\lim_{\vareps \to 0} \I_\vareps^{\mathrm{out}}(\normLigne{F}^k)/\vareps =
  0$. Hence, using \Cref{prop:limite_d_geq_p} if $d \geq p$ and
  \Cref{prop:upper_bound_i_laplace_exp} if $d \leq p$, we get that
  $\lim_{\vareps \to 0} \I_\vareps(\normLigne{F}^k)/\vareps = \C_k
  \I_0(1)$. Similarly, using \Cref{prop:upper_bound_i_out} and
  \Cref{prop:upper_bound_i_in} we have that
  $\lim_{\vareps \to 0} \I_\vareps(1) = \I_0(1)$, which concludes the proof upon
  remarking that for any $\vareps > 0$,
  $\pi_\vareps^\Psi(\normLigne{F}^k) = \I_\vareps(\normLigne{F}^k) /
  \I_\vareps(1)$.
\end{proof}

We are now ready to prove a generalization of \Cref{prop:behavior_theta}.

\begin{proposition}
  \label{prop:behavior_theta_gene}
  Let $\mu \in \Pens(\rset^d)$, $F: \ \rset^d \to \rset^p$ and $k \in
  \nsets$. Assume that the conditions of \Cref{prop:gibbs_measure} with
  $G_\vareps = \normLigne{F}^k - \vareps$ for any $\vareps > 0$ are satisfied
  and for any $\vareps > 0$, let $\macroeps$ the macrocanonical distribution
  with constraint $G_\vareps$ and reference measure $\mu$. Assume that there
  exists $\Psi: \ \rset^d \to \rset_+$ such that $\mu$ admits a density
  w.r.t. the Lebesgue measure given by $\Psi$. In addition, assume that
  \rref{assum:F} and \textup{\rref{assum:psi}} hold. Then, we have that
  $\theta_\vareps \sim_{\vareps \to 0} \C_k / \vareps$, where
  \begin{equation}
    \textstyle{
    \C_k = \left. \int_{\rset^p} \normLigne{t}^k \exp[-\normLigne{t}^k] \rmd t \middle/ \int_{\rset^p}  \exp[-\normLigne{t}^k] \rmd t \right. \eqsp . }
  \end{equation}  
\end{proposition}

\begin{proof}
  Recall that using \Cref{prop:gibbs_measure} we have that for any $\vareps > 0$
  and $\msa \in \mcb{\rset^d}$
  \begin{equation}
    \textstyle{
    \macroeps(\msa) = \left. \int_{\msa} \Psi(x) \exp[-\theta_\vareps \normLigne{F(x)}^k] \rmd x \middle/ \int_{\rset^d} \Psi(x) \exp[-\theta_\vareps \normLigne{F(x)}^k] \rmd x \right. \eqsp .}
  \end{equation}
  Hence, using \Cref{prop:lim_d_geq_p} we have that
  $\lim_{\vareps \to 0} \macroeps(\normLigne{F}^k) \theta_\vareps = \C_k
  $. Since $\macroeps[G_\vareps] = 0$ we have also have that
  $\macroeps[\normLigne{F}^k] = \vareps$, which concludes the proof.
\end{proof}

Note that \Cref{prop:behavior_theta} is obtained upon noting that $\C_2 = p/2$.

%%% Local Variables:
%%% mode: latex
%%% TeX-master: "main"
%%% End:

\subsection{Proofs of \Cref{sec:non_cvx_mon}}
\label{sec:quant_sgld}

  In \Cref{sec:application_to_sgld}, we establish
 \Cref{prop:quantitative_sgld}. 
 In \Cref{sec:stab-limit-meas} we use stability results from
 \cite{raginsky2017non, bousquet2002stab} to obtain \Cref{prop:stability}.
 Finally, we prove \Cref{prop:counter} in \Cref{proof:prop:counter}.  Additional
 technical results are postponed to \Cref{sec:techn-results-crefs}.

\subsubsection{Proof of \Cref{prop:quantitative_sgld}}
\label{sec:application_to_sgld}

In this section, we prove \Cref{prop:quantitative_sgld} which is an application
of a parametric version of the results presented in \Cref{sec:case-d-leq}. We
refer to \Cref{sec:techn-results-crefs} for a detailed presentation of these
results. We will apply them in the context of the non-convex
minimization setting presented in \Cref{sec:non_cvx_mon} which we recall here.

We aim at minimizing $U: \ \rset^d \to \rset$. We assume that there exist a
topological space $(\msz, \mcb{\msz})$, a probability measure
$\mu \in \Pens(\msz, \mcz)$ and $u: \ \rset^d \times \msz \to \rset_+$ such that
for any $x \in \rset^d$
\begin{equation}
  \textstyle{U(x) = \int_{\msz} u(x,z) \rmd \mu(z) \eqsp . }
\end{equation}
For any $n \in \nset$ we define $U_n : \rset^d \times \msz^n \to \rset$ such
that for any $x \in \rset^d$ and $z^{1:n} = \{z_{i}\}_{i=1}^n \in \msz^n$
\begin{equation}
  \textstyle{
    U_n(x,z^{1:n}) = (1/n)\sum_{i=1}^n u(x,z_i) \eqsp .
    }
\end{equation}
For all $\vareps >0$, when it is well-defined we denote by
$\Sker_\vareps : \ \msz^{n} \times \mcb{\rset^d} \to \ccint{0,1}$ the Markov
kernel such that for any $z^{1:n} \in \msz^n$ and $\msa \in \mcb{\rset^d}$ we
have
\begin{equation}
  \textstyle{
    \updelta_{z^{1:n}} \Sker_\vareps(\msa) = \left. \int_{\msa} \exp[-U_n(x,z^{1:n})/\vareps] \rmd x \middle/ \int_{\rset^d} \exp[-U_n(x,z^{1:n})/\vareps] \rmd x \right. \eqsp .
    }
\end{equation}
Similarly, when it is well-defined, we denote by $\Sker_0 : \ \msz^{n} \times \mcb{\rset^d} \to \ccint{0,1}$ the Markov
kernel such that for any $z^{1:n} \in \msz^n$ and $\msa \in \mcb{\rset^d}$ we
have
\begin{equation}
  \textstyle{
    \updelta_{z^{1:n}} \Sker_0(\msa) = \left. \int_{\msa \cap \msc_n(z^{1:n})} \det(\nabla_x^2 U_n(x, z^{1:n}))^{-1}  \rmd \calH^{0}(x) \middle/ \int_{\msc_n(z^{1:n})} \det(\nabla_x^2 U_n(x, z^{1:n}))^{-1}  \rmd \calH^{0}(x) \right. \eqsp ,}
\end{equation}
where $\msc_n(z^{1:n}) = \argmin U_n(\cdot, z^{1:n})$. We recall that for any
$\upbeta >0$, $\sigma_{\upbeta}^\star$ is defined in
\eqref{eq:def_sigma_star}. We begin with the following proposition.

\begin{proposition}
  \label{prop:wass1_res}
  Let $n \in \nset$ and assume \tup{\rref{assum:raginsky}$(n)$} and
  \tup{\rref{assum:U_n_param}$(n)$}. Then there exist $C \geq 0$ and
  $\upbeta, \bvareps > 0$ such that for any $\vareps \in \ooint{0, \bvareps}$
  \begin{align}
    & \textstyle{\int_{\msz^n} \wassersteinD[1](\updelta_{z^{1:n}} \Sker_\vareps, \updelta_{z^{1:n}} \Sker_0) \rmd \mu^{\otimes n}(z^{1:n}) } \\ & \qquad \qquad \qquad \leq \textstyle{C (1 + D_n) (\vareps^{1/2} + \vareps^{-d/2} \int_{\msz^{1:n}} \exp[-c^\star(z^{1:n})/\vareps] \rmd \mu^{\otimes n}(z^{1:n}))\eqsp ,}
  \end{align}
  with
  $D_n = \int_{\msz^n} \sigma_\upbeta^\star(z^{1:n}) \rmd \mu^{\otimes
    n}(z^{1:n}) < +\infty$ and $C, \bvareps, \upbeta$ that do not depend on $n$.
\end{proposition}

\begin{proof}
  The proof of this result is a direct application
  of \Cref{prop:conclusion_param_bounds} which is a parametric version of
  \Cref{thm:big_theo_extension}. In order to apply
  \Cref{prop:conclusion_param_bounds}, we check that \rref{assum:u_param} is
  satisfied for $\msz \leftarrow \msz^n$ and $u \leftarrow U_n$. We first 
  check that there exists $\mtt_0, \upalpha_0 > 0$ and $R_0$ such that for any
  $x \in \rset^d$ with $\normLigne{x} \geq R_0$ and $z^{1:n} \in \msz^n$,
  $U_n(x, z^{1:n}) \geq \mtt_0 \norm{x}^{\upalpha_0}$. We have that for any
  $x \in \rset^d$ and $z \in \msz$
  \begin{align}
    \textstyle{
    u(x,z) = u(0,z) + \int_0^1 \langle \nabla_x u(tx,z), x \rangle \rmd t \geq -A - \ctt + (\mtt/2)\normLigne{x}^2 \eqsp .
    }
  \end{align}
  Hence, for any $x \in \rset^d$ with
  $\normLigne{x} \geq 2((A + \ctt)/\mtt)^{1/2}$ we have that for any
  $z^{1:n} \in \msz^n$, $U_n(x, z^{1:n}) \geq (\mtt/4) \normLigne{x}^2$.
  {Let $z^{1:n} \in \msz$, we show that the number of minimizers
    of $U_n(\cdot, z^{1:n})$ is bounded. Since $x \in \rset^d$ with
    $\normLigne{x} \geq 2((A + \ctt)/\mtt)^{1/2}$ we have that
    $U_n(x, z^{1:n}) \geq (\mtt/4) \normLigne{x}^2$ and
    $\absLigne{U_n(0,z^{1:n})} \leq A$, there exists $\msk$ compact such that
    $\argmin \ensembleLigne{U_n(x, z^{1:n})}{x \in \rset^d} \subset \msk$ (see
    the remark following \rref{assum:u_param}). Assume that the number of
    minimizers is not bounded. In this case, there exists
    $(x_k)_{k \in \nset} \in \argmin \ensembleLigne{U_n(x, z^{1:n})}{x \in
      \rset^d}^\nset$ such that for any $k, \ell \in \nset$, $x_k \neq
    x_\ell$. Up to extraction we can assume that there exists
    $x^\star \in \rset^d$ such that $\lim_{k \to +\infty} x_k = x^\star$. Note
    that $x^\star \in \argmin \ensembleLigne{U_n(x, z^{1:n})}{x \in \rset^d}$ by
    continuity. In particular, $\det(\nabla^2_x U_n(x^\star, z^{1:n})) >0$ and
    there exists $r>0$ such that for any $x \in \cball{0}{r}$,
    $x \in \argmin \ensembleLigne{U_n(x, z^{1:n})}{x \in \rset^d}$ implies that
    $x = x^\star$. Hence, there exists $k_0 \in \nset$ such that for any
    $k \in \nset$ with $k \geq k_0$, $x_k = x^\star$ which is absurd.}  Hence,
  combining this result, \tup{\rref{assum:raginsky}$(n)$} and
  \tup{\rref{assum:U_n_param}$(n)$}, we can apply
  \Cref{prop:conclusion_param_bounds} which states that for any
  $\varphi: \ \rset^d \to \rset$ $M_{1,\varphi}$-Lipschitz function with
  $M_{1,\varphi}, C_\varphi \geq 0$ such that for any $x \in \rset^d$,
  $\abs{\varphi(x)} \leq C_\varphi \exp[C_\varphi\norm{x}^{ \upalpha }]$ then,
  there exist $B_2 \in \rmc(\rset_+, \rset_+)$ and $\upbeta > 0$ such that
  \begin{align}
    \label{eq:prop_a5}
    &\abs{\updelta_{z^{1:n}} \Sker_\vareps[\varphi] - \updelta_{z^{1:n}} \Sker_0[\varphi]} \\
      & \qquad \qquad \qquad \leq B_2(C_\varphi) (1 + M_{0,\varphi} + M_{1, \varphi}) (1 + \sigma_{\upbeta}^\star(z^{1:n})) \{ \vareps^{1/2} + \vareps^{-d/2} \exp[-c^\star(z^{1:n})/\vareps]\} \eqsp ,
      \end{align}
      with
      $M_{0, \varphi} = \sup \ensembleLigne{\absLigne{\varphi(x)}}{x \in \msk}$,
      $\msk$, $B_2$ and $\upbeta$ that do not depend on $z$, and $B_2$ non-decreasing.
      The rest of the proof is similar to the one of \Cref{coro:wass_dist} upon
      replacing \Cref{thm:big_theo} by \eqref{eq:prop_a5}.
\end{proof}

The proof of \Cref{prop:quantitative_sgld} is then a direct application of
\cite[Proposition 3.3]{raginsky2017non}, \cite[Equation (3.3)]{raginsky2017non}
and \Cref{prop:wass1_res}.

\subsubsection{Proof of \Cref{prop:stability}}
\label{sec:stab-limit-meas}

In this section, we prove \Cref{prop:stability}. We start by recalling a
proposition from \cite[Proposition 3.5]{raginsky2017non} about the uniform
stability of the exponential measure with potential $U_n$, see \cite{bousquet2002stab}
for a definition of the uniform stability.

\begin{lemma}
  \label{lemma:stabilite_epsilon}
  Assume that \tup{\rref{assum:raginsky}$(n)$} holds uniformly w.r.t. $n$. Then
  for any $n \in \nset$, $\vareps > 0$, $z_0^{1:n}, z_1^{1:n} \in \msz^n$ which
  only differs along one index, we have
  {
  \begin{equation}
    \textstyle{
      \absLigne{\updelta_{z^{1:n}_0} \Sker_\vareps u(\cdot, z)  - \updelta_{z^{1:n}_1} \Sker_\vareps u(\cdot, z)} \leq 4 (\Mtt^2(\ctt + d \vareps)/\mtt + B^2) \lsi(\vareps)/(n\vareps) \eqsp ,
      }
  \end{equation}
  where $\lsi(\vareps) \geq 0$ is such that for any $z^{1:n} \in \msz^n$,
  $\updelta_{z^{1:n}} \Sker_\vareps$ satisfies the logarithmic Sobolev
  inequality with constant $\lsi(\vareps)$, see \cite[Proposition 3.2]{raginsky2017non}.}  
\end{lemma}

For completeness, we recall that a probability measure $\nu \in \Pens(\rset^d)$
is said to satisfy the logarithmic Sobolev inequality with constant $\lsi$ if
for any $\pi \in \Pens(\rset^d)$ with positive density w.r.t. $\nu$ given by
$g \in \rmc^1(\rset^d, \rset)$ we have
\begin{equation}
  \textstyle{\KL{\pi}{\nu}\leq 2 \lsi \int_{\rset^d} \normLigne{\nabla \log(g(x)) }^2 \rmd \pi(x) \eqsp .}
\end{equation}
{ We can relate the constant appearing in the logarithmic Sobolev
  inequality with the uniform spectral gap given for any $\vareps > 0$ by
\begin{equation}  
  \textstyle{\lambda^\star(\vareps) = \inf_{z^{1:n} \in \msz^n, \ h \in \rmc^1(\rset^d) \cap \rmL^2(\updelta_{z^{1:n}} \Sker_\vareps)} \ensembleLigne{\updelta_{z^{1:n}}  \Sker_\vareps[ \normLigne{\nabla h}^2]  /  \updelta_{z^{1:n}} \Sker_\vareps[ \absLigne{h}^2] }{\ h \neq 0 , \ \updelta_{z^{1:n}} \Sker_\vareps[h] = 0 } \eqsp . }
\end{equation}  
More precisely, we have the following proposition (see \cite[Proposition 3.2,
Appendix B]{raginsky2017non}, see also \cite{bakry2008simple}).}

\begin{proposition}
  \label{prop:lsi_unif}
  {
  Assume that \tup{\rref{assum:raginsky}$(n)$} holds uniformly w.r.t. $n$. Then,
  there exist $A_0, A_1, \bvareps >0$ such that for any $\vareps \in \ocint{0, \bvareps}$
  \begin{equation}
    \lsi(\vareps) \leq A_0 (1 + (\lambda^\star(\vareps) \vareps)^{-1}) \eqsp .
  \end{equation}
In addition, $\lambda^\star(\vareps) \geq (1/A_1)\exp[-A_1/\vareps]$.}
\end{proposition}
Therefore, there exists $A \geq 0$ such that for any $\vareps \in \ocint{0, \bvareps}$
\begin{equation}
  \label{eq:lsi_ineq}
    \lsi(\vareps) \leq (A/\vareps) \exp[A/\vareps] \eqsp .
  \end{equation}
  We are now ready to show that the limiting measures are stable.

\begin{proposition}
  \label{prop:last_step}
  {
  Assume that \tup{\rref{assum:raginsky}$(n)$} and
  \tup{\rref{assum:U_n_param}$(n)$} hold uniformly w.r.t. $n \in \nset$. Assume that
  \begin{equation}
    \label{eq:condition_unif}
\textstyle{\lim_{\vareps \to 0} \sup \ensembleLigne{\vareps^{-d/2} \int_{\msz^n} \exp[-c^\star(z^{1:n})/\vareps] \rmd \mu^{\otimes n}(z^{1:n})}{ n \in \nset} = 0 \eqsp .  }      
\end{equation}
Then for any $\delta >0$, there exists $n_0 \in \nset$ such that for any $z \in \msz$, $n \in \nset$ with
$n \geq n_0$ and $j \in \{1, \dots, n\}$, we have
\begin{equation}
  \label{eq:result_quali}
    \textstyle{\int_{\msz^{n+1}} \absLigne{\updelta_{z^{1:n}_0} \Sker_0 u(\cdot, z)  - \updelta_{z^{1:n}_1} \Sker_0 u(\cdot, z)} \rmd \mu^{\otimes n}(z_0^{1:n})\rmd \mu(z_{1,j})\leq \delta \eqsp ,}
  \end{equation}
  where for any $z_0^{1:n} \in \msz^{n}$ we let $z_1^{1:n} \in \msz^n$ with
  $z_{0,i} = z_{1,i}$ for any $i \in \{1, \dots, N\}$ such that $i \neq j$.  In
  addition, assume that there exist $C_0, \alpha, \bvareps_0 > 0$
  such that for any $n \in \nset$ and $\vareps \in \ocintLigne{0, \bvareps_0}$
  \begin{equation}
    \label{eq:condition_bound}
  \textstyle{\vareps^{-d/2} \int_{\msz^n} \exp[-c^\star(z^{1:n})/\vareps] \rmd \mu^{\otimes n}(z^{1:n}) \leq C_0 \vareps^\alpha \eqsp ,  }  \end{equation}
Then there exists $n_0 \in \nset$ such that for any $\eta \in \ooint{0,1}$,
there exists $C \geq 0$ such that for any $z \in \msz$, $n \in \nset$ with
$n \geq n_0$ and $j \in \{1, \dots, n\}$, we have
\begin{equation}
  \label{eq:result_quanti}
    \textstyle{\int_{\msz^{n+1}} \absLigne{\updelta_{z^{1:n}_0} \Sker_0 u(\cdot, z)  - \updelta_{z^{1:n}_1} \Sker_0 u(\cdot, z)} \rmd \mu^{\otimes n}(z_0^{1:n})\rmd \mu(z_{1,j})\leq C / \log(n)^{s \eta} \eqsp ,}
  \end{equation}
  where $s = \min(\alpha/2, 1/4)$ and for any $z_0^{1:n} \in \msz^{n}$ we let
  $z_1^{1:n} \in \msz^n$ with $z_{0,i} = z_{1,i}$ for any
  $i \in \{1, \dots, N\}$ such that $i \neq j$.}
\end{proposition}

\begin{proof}
  {
  Let $\delta > 0$, $\eta \in \ooint{0,1}$. Using the triangle inequality we
  have for any $\vareps > 0$, $n \in \nset$ and
  $z_0^{1:n}, z_1^{1:n}, z \in \msz^n$ where there exists
  $j \in \{1, \dots, n\}$ such that for any $i \in \{1, \dots, n\}$, $i \neq j$,
  $z_{0,i} = z_{1,i}$
  \begin{align}
    \label{eq:triangle}
    &\absLigne{\updelta_{z^{1:n}_0} \Sker_0 u(\cdot, z)  - \updelta_{z^{1:n}_1} \Sker_0 u(\cdot, z)} \leq \absLigne{\updelta_{z^{1:n}_0} \Sker_0 u(\cdot, z)  - \updelta_{z^{1:n}_0} \Sker_\vareps u(\cdot, z)} \\
    & \qquad + \absLigne{\updelta_{z^{1:n}_0} \Sker_\vareps u(\cdot, z)  - \updelta_{z^{1:n}_1} \Sker_\vareps u(\cdot, z)}  + \absLigne{\updelta_{z^{1:n}_1} \Sker_\vareps u(\cdot, z)  - \updelta_{z^{1:n}_1} \Sker_0 u(\cdot, z)}  \eqsp . 
  \end{align}
  Using \cite[Lemma 3.5]{raginsky2017non} and \Cref{lemma:moment_bound}, there
  exists $C_0 \geq 0$ such that for any $i \in \{0,1\}$, $n \in \nset$,
  $z_0^{1:n}, z_1^{1:n}, z \in \msz^n$ where there exists
  $j \in \{1, \dots, n\}$ such that for any $i \in \{1, \dots, n\}$, $i \neq j$,
  $z_{0,i} = z_{1,i}$ and $\vareps \in \ocint{0, \bvareps_1}$ (where
  $\bvareps_1$ is given by $\bvareps_1 \leftarrow \bvareps$ in
  \Cref{prop:wass1_res}) we have
  \begin{equation}
    \absLigne{\updelta_{z^{1:n}_i} \Sker_0 u(\cdot, z)  - \updelta_{z^{1:n}_i} \Sker_\vareps u(\cdot, z)}  \leq C_0 \wassersteinD[2](\updelta_{z^{1:n}_i} \Sker_0 , \updelta_{z^{1:n}_i} \Sker_\vareps ) \eqsp . 
  \end{equation}
  Using this result, \Cref{lemma:moment_bound} and
  \Cref{lemma:wasserstein_2_bound}, there exists $C_1 \geq 0$ such that for any
  $i \in \{0,1\}$, $n \in \nset$, $z_0^{1:n}, z_1^{1:n}, z \in \msz^n$ where
  there exists $j \in \{1, \dots, n\}$ such that for any
  $i \in \{1, \dots, n\}$, $i \neq j$, $z_{0,i} = z_{1,i}$ and
  $\vareps \in \ocint{0, \bvareps_1}$ (where $\bvareps_1$ is given by
  $\bvareps_1 \leftarrow \bvareps$ in \Cref{prop:wass1_res}) we have
  \begin{equation}
    \label{eq:upper_bound_w1}
    \absLigne{\updelta_{z^{1:n}_i} \Sker_0 u(\cdot, z)  - \updelta_{z^{1:n}_i} \Sker_\vareps u(\cdot, z)} \leq C_1 \wassersteinD[1]^{\eta/2}(\updelta_{z^{1:n}_i} \Sker_0, \updelta_{z^{1:n}_i} \Sker_\vareps) \eqsp . 
  \end{equation}
  We divide the rest of the proof into two parts. First, we start with our
  qualitative result by showing that \eqref{eq:result_quali} holds under
  \eqref{eq:condition_unif}. Then, we turn to our quantitative bounds by showing
  that \eqref{eq:result_quanti} holds under \eqref{eq:condition_bound}
  \begin{enumerate}[wide, labelwidth=!, labelindent=0pt, label=(\alph*)]
  \item Using \Cref{prop:wass1_res} and \eqref{eq:condition_unif} we have that for 
  any $i \in \{0,1\}$
  \begin{equation}
    \textstyle{\lim_{\vareps \to 0} \sup \ensembleLigne{\int_{\msz^n}\wassersteinD[1](\updelta_{z^{1:n}_i} \Sker_0, \updelta_{z^{1:n}_i} \Sker_\vareps)\rmd \mu^{\otimes n}(z_i^{1:n})}{n \in \nset} =0  \eqsp .}
  \end{equation}
  Combining this result, \eqref{eq:upper_bound_w1} and that
  $t \mapsto t^{\eta/2}$ is concave, we get that 
  for any $i \in \{0,1\}$
  \begin{equation}
    \label{eq:intermediate_s_eps_s0_quali}
    \textstyle{ \lim_{\vareps \to 0} \sup \ensembleLigne{\int_{\msz^n} \absLigne{\updelta_{z^{1:n}_i} \Sker_0 u(\cdot, z)  - \updelta_{z^{1:n}_i} \Sker_\vareps u(\cdot, z)} \rmd \mu^{\otimes n}(z_i^{1:n})}{n \in \nset} = 0  \eqsp . }
  \end{equation}
  In addition, using \Cref{lemma:stabilite_epsilon} we have for any $n \in \nset$ and $\vareps > 0$
  \begin{equation}
    \textstyle{\int_{\msz^{n+1}} \absLigne{\updelta_{z^{1:n}_0} \Sker_\vareps u(\cdot, z)  - \updelta_{z^{1:n}_1} \Sker_\vareps u(\cdot, z)} \rmd \mu^{\otimes}(z^{1:n}_0) \rmd \mu(z_{1,j}) \leq 4 (\Mtt^2(\ctt + d \vareps)/\mtt + B^2) \lsi(\vareps)/(n\vareps) \eqsp ,}
  \end{equation}
  Combining this result, \eqref{eq:triangle} and \eqref{eq:intermediate_s_eps_s0_quali} we get that there exists $\vareps > 0$ such that 
  \begin{align}
    &\textstyle{\int_{\msz^{n+1}} \absLigne{\updelta_{z^{1:n}_0} \Sker_0 u(\cdot, z)  - \updelta_{z^{1:n}_1} \Sker_0 u(\cdot, z)} \rmd \mu^{\otimes}(z^{1:n}_0) \rmd \mu(z_{1,j})}\\
    &\qquad \qquad \leq 4 (\Mtt^2(\ctt + d \bvareps)/\mtt + B^2) \lsi(\vareps)/(n\vareps) \\
    & \qquad \qquad \quad + \textstyle{2 \sup \ensembleLigne{\int_{\msz^n} \absLigne{\updelta_{z^{1:n}} \Sker_0 u(\cdot, z)  - \updelta_{z^{1:n}} \Sker_\vareps u(\cdot, z)} \rmd \mu^{\otimes n}(z^{1:n})}{n \in \nset}}  \\
    &\qquad \qquad \leq 4 (\Mtt^2(\ctt + d \bvareps)/\mtt + B^2) \lsi(\vareps)/(n\vareps) + \delta/2  \eqsp .     
  \end{align}
  Hence, there exists $n_0 \in \nset$ such that for any $n \in \nset$ with
  $n \geq n_0$,
  $4 (\Mtt^2(\ctt + d \bvareps)/\mtt + B^2) \lsi(\vareps)/(n\vareps) \leq
  \delta/2$, which concludes the first part of the proof.
\item Let $n \in \nset$ with $n \geq n_0$ and
  $(2/A) \log(n_0)^{-1} < \bvareps = \min(\bvareps_0, \bvareps_1, \bvareps_2)$
  (where $\bvareps_1$ is given by $\bvareps_1 \leftarrow \bvareps$ in
  \Cref{prop:wass1_res} and $A, \bvareps_2$ are given by
  $\bvareps_2 \leftarrow \bvareps$ in \Cref{prop:lsi_unif} and
  \eqref{eq:lsi_ineq}).  In addition, using
  \Cref{prop:wass1_res}, there exists $C_2 \geq 0$ (that does not depend on $n$) such
  that for any $i \in \{0,1\}$ and $\vareps \in \ocint{0, \bvareps}$
  \begin{equation}
    \textstyle{\int_{\msz^n}\wassersteinD[1](\updelta_{z^{1:n}_i} \Sker_0, \updelta_{z^{1:n}_i} \Sker_\vareps)\rmd \mu^{\otimes n}(z_i^{1:n}) \leq C_2 \max(\vareps^{1/2}, \vareps^{\alpha })\eqsp .}
  \end{equation}
  Combining this result, \eqref{eq:upper_bound_w1} and the fact that
  $t \mapsto t^{\eta/2}$ is concave, we get that there exists $C_3 \geq 0$ (that does not
  depend on $n$) such that for any $i \in \{0,1\}$ and
  $\vareps \in \ocint{0, \bvareps}$
  \begin{equation}
    \label{eq:intermediate_s_eps_s0}
    \textstyle{\int_{\msz^n} \absLigne{\updelta_{z^{1:n}_i} \Sker_0 u(\cdot, z)  - \updelta_{z^{1:n}_i} \Sker_\vareps u(\cdot, z)} \rmd \mu^{\otimes n}(z_i^{1:n}) \leq C_3 \max(\vareps^{\eta/4}, \vareps^{\alpha \eta /2})  \eqsp , }
  \end{equation}
  In addition, using \Cref{lemma:stabilite_epsilon} we have for any
  $\vareps \in \ocint{0, \bvareps}$
  \begin{equation}
    \textstyle{\int_{\msz^{n+1}} \absLigne{\updelta_{z^{1:n}_0} \Sker_\vareps u(\cdot, z)  - \updelta_{z^{1:n}_1} \Sker_\vareps u(\cdot, z)} \rmd \mu^{\otimes}(z^{1:n}_0) \rmd \mu(z_{1,j}) \leq 4 (\Mtt^2(\ctt + d \vareps)/\mtt + B^2) \lsi(\vareps)/(n\vareps) \eqsp .}
  \end{equation}
  Combining this result, \eqref{eq:triangle} and \eqref{eq:intermediate_s_eps_s0} we get
  \begin{align}
    &\textstyle{\int_{\msz^{n+1}} \absLigne{\updelta_{z^{1:n}_0} \Sker_0 u(\cdot, z)  - \updelta_{z^{1:n}_1} \Sker_0 u(\cdot, z)} \rmd \mu^{\otimes}(z^{1:n}_0) \rmd \mu(z_{1,j})}\\
    &\qquad \qquad \leq 4 (\Mtt^2(\ctt + d \bvareps)/\mtt + B^2) \lsi(\vareps)/(n\vareps) + 2C_3 \vareps^{s \eta}\\
                                                                                                                                                          &\qquad \qquad  \leq 4 A (\Mtt^2(\ctt + d \bvareps)/\mtt + B^2) \exp[A/\vareps]/(n\vareps^2) + 2C_3 \vareps^{s \eta} \eqsp ,
  \end{align}
  with $s = \min(\alpha/2,1/4)$.
  We conclude the proof upon letting $\vareps = (2/A) \log(n)^{-1}$.
\end{enumerate}}
\end{proof}

The stability of the limiting measures allows us to establish
\Cref{prop:stability} which provides quantitative bounds on
$\mu^{\otimes n} \Sker_0[U] - U^\star$ for large values of $n \in
\nset$. Indeed, once \Cref{prop:last_step} is established the proof of
\Cref{prop:stability} is classical and follows the lines of \cite[Section
3.7]{raginsky2017non}.

\begin{proof}
  Let $n \in \nset$ and $n \geq n_0$ with $n_0$ given by
  \Cref{prop:stability}. {Using the definition of $U$ and
    $\Sker_0$ we have
    $\mu^{\otimes n} \Sker_0 [U] = \int_{\msz^n} \int_{\msz^n} \int_{\rset^d}
    U_n(x, \tilde{z}^{1:n}) \Sker_0(z^{1:n}, \rmd x) \rmd \mu^{\otimes
      n}(z^{1:n}) \rmd \mu^{\otimes n}(\tilde{z}^{1:n})$.}  For any
  $z^{1:n} \in \msz^n$ we define
  $U_n^\star(z^{1:n}) = \inf \ensembleLigne{U_n(x, z^{1:n})}{x \in
    \rset^d}$. Using that
  $U^\star \geq \int_{\msz^n} U_n^\star(z^{1:n}) \rmd \mu^{\otimes n}(z^{1:n})$
  and that {$\updelta_{z^{1:n}}\Sker_0$ is concentrated on
    $\text{argmin} \ensembleLigne{U_n(x, z^{1:n})}{x \in \rset^d$} we have}
  \begin{align}
    &\mu^{\otimes n} \Sker_0 [U] - U^\star \leq \textstyle{\mu^{\otimes n} \Sker_0 [U] - \int_{\msz^n} U_n^\star(z^{1:n}) \rmd \mu^{\otimes n}(z^{1:n})} \\
                                        &\leq \textstyle{\mu^{\otimes n} \Sker_0 [U] - \int_{\msz^n} \int_{\rset^d} U_n(x, z^{1:n}) \Sker_0(z^{1:n}, \rmd x) \rmd \mu^{\otimes n}(z^{1:n}) }\\
    &\leq \textstyle{\int_{\msz^n} \int_{\msz^n} \int_{\rset^d} U_n(x, \tilde{z}^{1:n}) \Sker_0(z^{1:n}, \rmd x) \rmd \mu^{\otimes n}(z^{1:n}) \rmd \mu^{\otimes n}(\tilde{z}^{1:n})} \\
    & \qquad \textstyle{- \int_{\msz^n} \int_{\rset^d} U_n(x, z^{1:n}) \Sker_0(z^{1:n}, \rmd x) \rmd \mu^{\otimes n}(z^{1:n})  }  \\
    &\leq\textstyle{ (1/n) \sum_{i=1}^n \int_{\msz} \int_{\msz^n} \int_{\rset^d} \{ u(x,\tilde{z}_i) - u(x,z_i)\} \Sker_0(z^{1:n}, \rmd x) \rmd \mu^{\otimes n}(z^{1:n}) \rmd \mu(\tilde{z}_i)} \\
    &\leq \textstyle{(1/n) \sum_{i=1}^n \int_{\msz} \int_{\msz^n} \{\int_{\rset^d}  u(x,z_i) \Sker_0(z^{1:n}, \rmd x) - \int_{\rset^d}  u(x,z_i) \Sker_0(\tilde{z}_i^{1:n}, \rmd x) \} \} \rmd \mu^{\otimes n}(z^{1:n}) \rmd \mu(\tilde{z}_i)  \eqsp ,}
  \end{align}
  where for any $i \in \{1, \dots, n\}$, we have that for any
  $j \in \{1, \dots, n\}$, $\tilde{z}_{i,j} = z_i$ and
  $\tilde{z}_{i,i} = \tilde{z}_i$. We conclude using \Cref{prop:stability}.

  \subsubsection{Proof of \Cref{prop:counter}}
  \label{proof:prop:counter}

  We recall that $u$ is given in \eqref{eq:def_u_counter}. We divide the proof
  into two parts.
    \begin{enumerate}[wide, labelwidth=!, labelindent=0pt, label=(\alph*)]
    \item First, we prove that
      $\lim_{n \to +\infty} \mu^{\otimes n} \Sker_0[\varphi] = (\varphi(-\uppi)
      + \varphi(\uppi))/2$. Let $n \in \nset$. Assume that $z < 0$. Then the
      minimum of $x \mapsto u(x,z)$ is attained on
      $\coint{\uppi,+\infty}$. Denote $\bar{u}: \ \rset \times \coint{-1/2,0}$
      such that for any $x \in \rset$ and $z \in \coint{-1/2, 0}$,
      $\bar{u}(x,z) = h(x) + xz + \uppi z + 1 - \cos(3x)$ with
      $h(x) = x^4/(1+x^2)$. Note that for any $x \geq 0$ and
      $z \in \coint{-1/2, 0}$, $\bar{u}(x,z) = u(x+\uppi,z)$. There exists
      $a \in \ccint{0,\uppi/3}$ such that for any $x \geq 0$, $h'(x) - 1 \leq 0$
      if $x \leq a$ and $h'(x)-1>0$ otherwise. Hence, we get that for any
      $x \in \rset$ with $x \geq \uppi/3$ and $z \in \coint{-1/2,0}$
      \begin{equation}
        \textstyle{
          \bar{u}(x,z) - \bar{u}(0,z) = \bar{u}(x,z) - \uppi z\geq h(\uppi/3) - \uppi/6 > 0  \eqsp .
          }
        \end{equation}
        Therefore, for any $z \in \coint{-1/2,0}$, the global minimum of
        $x \mapsto u(x,z)$ is attained on $\ooint{\uppi,4\uppi/3}$. We have that for
        any $x \in \ccint{\uppi/6, \uppi/3}$ and $z \in \coint{-1/2,0}$
        \begin{equation}
          \textstyle{\partial_1 \bar{u}(x,z) \geq h'(x) - 1/2 + 3(2 - (6/\uppi)x) > 0 \eqsp . }
        \end{equation}
        In addition, we have that for any $z \in \coint{-1/2,0}$,
        $\partial_1\bar{u}(0,z) = -z$. Hence, there exists
        $\bar{x}(z) \in \ccint{0,\uppi/6}$ such that
        $\partial_1\bar{u}(\bar{x}(z),z) = 0$. In addition we have that for any
        $z \in \coint{-1/2,0}$, $x \mapsto \partial_1\bar{u}(x,z)$ is increasing on
        $\ccint{0,\uppi/6}$. Therefore, for any $z \in \coint{-1/2,0}$ there
        exists a unique minimizer of $x \mapsto u(x,z)$ on
        $\ccint{\uppi,5\uppi/6}$ given by $x^\star(z) = \uppi + \bar{x}(z)$. The
        same conclusion holds with $x^\star(z) \in \ccint{-5\uppi/6,-\uppi}$ if
        $z \in \ocint{0,1/2}$.  We have that
        $\lim_{z \to 0} \sup \ensembleLigne{\normLigne{u(x,z) - u(x,0)}}{x \in
          \ccint{-5\uppi/6,5\uppi/6}} = 0$. Therefore we have that every limit
        point of $\{x^\star(z)\}_{z < 0}$ when $z \to 0$ is a global minimizer
        of $u(\cdot, 0)$. But recall that
        $\{x^\star(z)\}_{z < 0} \subset \ccint{\uppi, 5\uppi/6}$. Therefore,
        every limit point of $\{x^\star(z)\}_{z < 0}$ is equal to $\uppi$ and we
        have that $\lim_{z\to 0, z > 0}x^\star(z) = \uppi$. For any
        $n \in \nset$, denote by $g_n$ the density of $T_\# \mu^{\otimes n}$
        where $T: \ \msz^n \to \msz$ is given by
        $T(z^{1:n}) = (1/n) \sum_{i=1}^n z_i$. For any $r, \vareps > 0$ there
        exists $n_0 \in \nset$ such that for any $n \in \nset$ with
        $n \geq n_0$,
        $\int_{\ball{0}{r}^\complementary} g_n(z) \rmd z \leq \vareps$. Let
        $\varphi \in \rmc(\rset, \rset)$ bounded and $\vareps > 0$. Let $r> 0$ such that for any
        $z \in \ccintLigne{0,r}$,
        $\absLigne{\varphi(x^\star(z)) - \varphi(-\uppi)} \leq \vareps$ and for
        any $z \in \ccintLigne{-r,0}$,
        $\absLigne{\varphi(x^\star(z)) - \varphi(\uppi)} \leq \vareps$. Using
        this result, we have for any $n \in \nset$ with $n \geq n_0$
        \begin{align}
          & \textstyle{\absLigne{\mu^{\otimes n} \Sker_0[\varphi] - (\varphi(-\uppi) + \varphi(\uppi))/2 }} \\
          & \qquad \qquad \leq \textstyle{\int_0^{+\infty} \absLigne{\varphi(x^\star(z)) - \varphi(-\uppi)} g_n(z) \rmd z + \int_{-\infty}^0 \absLigne{\varphi(x^\star(z)) - \varphi(\uppi)} g_n(z) \rmd z} \\
          & \qquad \qquad \leq (1 + 2 \normLigne{\varphi}_\infty) \vareps  \eqsp . 
        \end{align}
        Therefore, we get that
        $\lim_{n \to +\infty} \mu^{\otimes n} \Sker_0[\varphi] =
          (\varphi(-\uppi) + \varphi(\uppi))/2$, which concludes the first part
          of the proof.
      \item Second, we prove that for any $n \in \nset$,
        $\mu^{\otimes n} \Sker_0[U] - U^\star \leq (\uppi/(6\sqrt{3}))
        n^{-1/2}$. Note that for any $n \in \nset$ and
        $z^{1:n} \in \ccint{-1/2,1/2}^n$, with $\sum_{i=1}^n z_i \neq 0$,
        $\updelta_{z^{1:n}} \Sker_0[U] = U(x^\star(\bar{z}^{1:n}))$ with
        $x^\star(\bar{z}^{1:n}) \in \ccint{-5\uppi/6,5\uppi/6}$. We also have that for
        any $x \in \rset$ and $z_1, z_2 \in \ccint{-1/2,1/2}$,
        $\textstyle{\absLigne{u(x,z_1) - u(x,z_2)} \leq
          \absLigne{x}\absLigne{z_1 - z_2}}$.  In particular, we have that for
        any $z^{1:n} \in \ccint{-1/2,1/2}^n$ and
        $x \in \ccint{-\uppi/3, \uppi/3}$
        \begin{equation}
          \textstyle{\absLigne{U(x) - U_n(x,z^{1:n})} \leq (\uppi/3) \absLigne{(1/n)\sum_{i=1}^n z_i} \eqsp .}
        \end{equation}
        Hence, using this result and that $U^\star = U(\uppi/3) = 0$, we have that for any $z^{1:n} \in \ccint{-1/2,1/2}$
        \begin{equation}
          \textstyle{U_n(x^\star(\bar{z}^{1:n}), z^{1:n}) \leq U_n(\uppi/3, z^{1:n}) \leq U^\star + (\uppi/3) \absLigne{(1/n)\sum_{i=1}^n z_i} \eqsp .}
        \end{equation}
        Combining this result and that $\int_{\rset} z^2 \rmd \mu(z) = 1/12$ we have $\mu^{\otimes n} \Sker_0[U] - U^\star \leq (\uppi/(6\sqrt{3})) n^{-1/2}$, 
        which concludes the proof.
    \end{enumerate}
  \end{proof}

%%% Local Variables:
%%% mode: latex
%%% TeX-master: "main"
%%% End:

%%% Local Variables:
%%% mode: latex
%%% TeX-master: "main"
%%% End:

\section*{Acknowledgement}
\label{sec:acknowledgement}

This work is part of the research program MISTIC, supported by the French Agence
Nationale pour la Recherche (ANR-19-CE40-0005). V. De Bortoli was also supported
by EPSRC grant EP/R034710/1.  We thank Francesca Crucinio for pointing to us the
use of \emph{thermodynamic barriers} and \emph{kinetic barriers} in chemistry.

%%% Local Variables:
%%% mode: latex
%%% TeX-master: "main"
%%% End:

\bibliographystyle{plain}
\bibliography{BiblioMacroMicro}

\appendix

\counterwithin{theorem}{section}
\counterwithin{lemma}{section}
\counterwithin{corollary}{section}
\counterwithin{proposition}{section}
\counterwithin{definition}{section}

\section*{Organization of the appendix}
\label{sec:organ-append}

In this supplementary material we derive technical lemmas and additional
results. In particular, we gather the technical lemmas of
\Cref{sec:main-results} in \Cref{sec:technical-bounds} and the ones of
\Cref{sec:application_to_sgld} in \Cref{sec:techn-results-crefs}. In
\Cref{sec:basics-geom-meas}, we recall basic results from differential geometry
and geometric measure theory.

%%% Local Variables:
%%% mode: latex
%%% TeX-master: "main"
%%% End:

\section{Technical results for \Cref{sec:main-results}}
\label{sec:technical-bounds}

In this section, we derive some technical lemmas used in 
\Cref{sec:proof_thm_macro} in order to prove \Cref{thm:big_theo_extension} and
other results from \Cref{sec:main-results}. We recall that for any
$\varphi: \ \rset^d \to \rset_+$ and $\vareps > 0$, when this is well-defined,
we set
\begin{align}
  &\textstyle{\I_{\vareps}(\varphi) = C_{\vareps}^{-1} \int_{\rset^d} \varphi(x) \Psi(x) \exp[-\norm{F(x)}^k/\vareps] \rmd x \eqsp , \quad \J_{\vareps} = \I_{\vareps}(1) \eqsp ,} \\ &\textstyle{C_{\vareps} = \int_{\rset^d} \exp[-\norm{x}^k/\vareps] \rmd x = \vareps^{d/k} \int_{\rset^d} \exp[-\norm{x}^k]  \rmd x = \vareps^{d/k} C_1 \eqsp .}
\end{align}
In addition, we define
\begin{equation}
  \textstyle{
    \I_0(\varphi) =  \int_{F^{-1}(0)} \varphi(x) \Psi(x) \jacinv{x} \rmd \calH^{d-\hat{d}}(x) \eqsp , \quad \J_0 = \I_0(1) \eqsp ,
    }
  \end{equation}
  with $\hat{d} = \min(d,p)$.

  In \Cref{sec:from-normal-hessian} we establish a link between a Hessian
  computed on the normal bundle of a manifold and the generalized Jacobian.  In
  \Cref{sec:trunc-bounds-lower} we derive technical truncation bounds for the
  proof of \Cref{thm:big_theo_extension}. Explicit controls of some derivative
  are presented in \Cref{sec:quant-contr-deriv} in order to derive
  \Cref{prop:lip_l}.  Finally, we present a smoothing lemma in
  \Cref{sec:regularity-results} which is key to weaken the regularity
  assumptions of \Cref{prop:final_d_geq_p}.

  \subsection{From normal Hessian to generalized Jacobian}
  \label{sec:from-normal-hessian}

  Let $f \in \rmc^2(\rset^d)$ and $\mathrm{M}$ a manifold in $\rset^d$. For any
  $x \in \mathrm{M}$ we define $\nabla_{\perp}^2 f(x)$ to be the projection
  of the Hessian on the orthogonal of the tangent space of $\mathrm{M}$ at $x$,
  see \cite{hwang1980}.
  
  \begin{lemma}
    \label{lemma:hessian_normal}
    Let $U: \ \rset^d \to \rset$ and $F\in \rmc^\infty(\rset^d, \rset^p)$ such
    that for any $x \in \rset^d$, $U(x) = \normLigne{F(x)}^2$. In addition,
    assume that $F^{-1}(0) \neq \emptyset$ and that for any $x \in F^{-1}(0)$,
    $\jac{x}>0$. Then, $\argmin \ensembleLigne{U(x)}{x \in \rset^d}$ is a smooth
    manifold and for any $x \in F^{-1}(0)$ we have that
    $\det(\nabla^2_{\perp} U(x)) = \jac{x}^2$.
  \end{lemma}

  \begin{proof}
    First, we have that
    $\argmin \ensembleLigne{U(x)}{x \in \rset^d} = F^{-1}(0)$. Hence,
    $\argmin \ensembleLigne{U(x)}{x \in \rset^d}$ is a smooth manifold since
    $F \in \rmc^\infty(\rset^d, \rset^p)$. Let $x \in F^{-1}(0)$. We have that
    $\nabla^2 U(x) = \rmD F(x)^\top \rmD F(x)$. Note that
    $\rmD F(x)^\top = (\nabla F_1(x), \dots, \nabla F_p(x))$ is a basis of
    $\mathrm{ker}(\rmD F(x))^\perp$, where we recall that
    $\mathrm{ker}(\rmD F(x))$ is the tangent space to $F^{-1}(0)$ at $x$. Denote
    by $O(x) = (f_1(x), \dots, f_p(x))$ the orthonormal basis of
    $\mathrm{ker}(\rmD F(x))^\perp$ obtained from $\rmD F(x)^\top$ using the
    Gram-Schmidt process. There exists a triangular $p \times p$ matrix $T(x)$
    such that $O(x) = \rmD F(x)^\top T(x)$. We also have
    \begin{equation}
      \Id = O(x)^\top O(x) = T(x)^\top \rmD F(x) \rmD F(x)^\top T(x) \eqsp . 
    \end{equation}
    Hence, we get that $\det(T(x)) = \jac{x}^{-1}$. We also have that
    \begin{align}
      \det(\nabla_{\perp}^2 U(x)) &= \det(O(x)^\top \nabla^2 U(x) O(x)) \\
      &= \det(T(x)^\top \rmD F(x) \rmD F(x)^\top \rmD
    F(x) \rmD F(x)^\top T(x)) = \jac{x}^2 \eqsp ,
  \end{align}
  which concludes the proof.
  \end{proof}

  \subsection{Truncation and lower bounds}
  \label{sec:trunc-bounds-lower}

\begin{lemma}
  \label{prop:lower_bound_j_eps}
  Assume \rref{assum:F} and \textup{\rref{assum:psi}}. Then, for any
  $\bvareps \geq 0$ there exists $A_0 \geq 0$ such that for any
  $\vareps \in \ccint{0, \bvareps}$, $\J_{\vareps} \geq A_0 m_{0,\Psi}$ with
  $m_{0,\Psi} = \inf \ensembleLigne{\Psi(x)}{x \in F^{-1}(0)}$ and $A_0$
  that does not depend on $\Psi$.
\end{lemma}

\begin{proof}
  Since $F(0) = 0$ and $F \in \rmc^1(\rset^d, \rset^p)$, there exists $M \geq 0$
  such that for any $x \in \cball{0}{1}$, $\norm{F(x)} \leq M \norm{x}$. Note
  that $F^{-1}(0)$ is compact since
  $\lim_{\norm{x} \to +\infty} \norm{F(x)} = +\infty$ and
  $F \in \rmc(\rset^d, \rset^p)$. Hence, since for any $x \in F^{-1}(0)$,
  $\Psi(x) > 0$ and $\Psi \in \rmc(\rset^d, \rset_+)$ there exists
  $\eta \in \ooint{0,1}$ such that for any
  $x \in \ball{0}{\eta} \cup F^{-1}(0)$, $\Psi(x) \geq
  m_{0,\Psi}/2$% and for any $x \in F^{-1}(0)$, $\Psi(x) \geq 2m_{0,\Psi}$
  . Using this result we have for any $\vareps > 0$
  \begin{align}
    \J_{\vareps} &= \textstyle{\vareps^{-d/k} C_1^{-1} \int_{\rset^d} \Psi(x) \exp[-\norm{F(x)}^k/\vareps] \rmd x} \\
                 &\geq \textstyle{\vareps^{-d/k} (m_{0,\Psi}/2)  C_1^{-1} \int_{\ball{0}{\eta}} \exp[-M^k \norm{x}^k/\vareps] \rmd x}  \\
    &\textstyle{\geq  C_1^{-1} (m_{0,\Psi}/2) M^{-d} \int_{\cball{0}{M\eta/\bvareps^{1/k}}} \exp[-\norm{x}^k] \rmd x \eqsp . }
  \end{align}
  Using that $F^{-1}(0)$ is compact, there exists $\Mtt \geq 0$ such that for
  any $x \in F^{-1}(0)$, $\jac{x} \leq \Mtt$.  Therefore, we get that
  \begin{equation}
    \textstyle{\J_0 = \int_{F^{-1}(0)} \Psi(x) \jacinv{x} \rmd \calH^{d-\hat{d}}(x) \geq m_{0,\Psi} \Mtt^{-1} \calH^{d-\hat{d}}(F^{-1}(0)) \eqsp . }
  \end{equation}
  Since $\calH^{d-\hat{d}}(F^{-1}(0)) < +\infty$ using
  \Cref{prop:rectifiable_level_set} in the case where $d \geq p$ and the fact
  that $\calH^0(F^{-1}(0))< +\infty$ if $d \leq p$ (see the first part of the
  proof of \Cref{lemma:existence_U}), we have that for any
  $\vareps \in \ccint{0, \bvareps}$
  \begin{align}
    &\J_{\vareps} \geq A_0m_{0,\Psi} \eqsp ,  \qquad A_0 = \min(A_0^1, A_0^2) \eqsp , \\
    &\textstyle{A_0^1 = (1/2)C_1^{-1} M^{-d} \int_{\cball{0}{M\eta/\bvareps^{1/k}}} \exp[-\norm{x}^k]
  \rmd x \eqsp ,} \quad A_0^2 = M^{-1} \calH^{d-\hat{d}}(F^{-1}(0)) \eqsp , 
  \end{align}
  which concludes the proof.
\end{proof}

\begin{lemma}
  \label{prop:upper_bound_i_out}
  Assume \rref{assum:F} and \textup{\rref{assum:psi}}. Let
  $\varphi: \ \rset^d \to \rset$ and $C_\varphi \geq 0$ such that for any
  $x \in \rset^d$
  \begin{equation}
    \label{eq:cond_varphi_prop}
    \textstyle{\absLigne{\varphi(x)} \leq C_\varphi \exp[C_\varphi \normLigne{x}^{\upalpha k}] \eqsp . }
  \end{equation}
  Then, for any $\bvareps \in \oointLigne{0, \mtt^k/(1 + C_{\varphi, \Psi})}$ and
  $\msv \subset \rset^d$ open and bounded such that $F^{-1}(0) \subset \msv$ there exist
  $\beta_1 > 0$ and $A_1 \in \rmc(\rset_+, \rset_+)$ such that for any $\vareps \in \ooint{0, \bvareps}$
  \begin{equation}
    \I_{\vareps}^{\mathrm{out}}(\varphi) \leq A_1(C_{\varphi, \Psi}) \vareps^{-d/k} \exp[-\beta_1 / \vareps]\eqsp ,
  \end{equation}
  with
  $\I_{\vareps}^{\mathrm{out}}(\varphi) = \I_{\vareps}(\varphi
  \1_{\msv^\complementary})$,
  $C_{\varphi, \Psi} = C_\varphi + C_\Psi + C_\varphi C_\Psi$ and $A_1, \beta_1$
  functions that do not depend on $\varphi$ and $\Psi$. Finally, $A_1$ is
  non-decreasing.
\end{lemma}

\begin{proof}
  First using \Cref{assum:psi} and \eqref{eq:cond_varphi_prop} there exists
  $C_{\varphi, \Psi}$ such that for any $x \in \rset^d$
  \begin{equation}
    \label{eq:bound_prod}
    \abs{\varphi(x)} \Psi(x) \leq C_{\varphi, \Psi} \exp[C_{\varphi, \Psi} \norm{x}] \eqsp , \qquad C_{\varphi, \Psi} = C_\varphi + C_\Psi + C_\varphi  C_\Psi \eqsp .
  \end{equation}  
  Since $\msv$ is bounded
  there exists $R' \geq R$ (where $R$ is given in \Cref{assum:F}) such that
  $\msv \subset \cball{0}{R'}$. Note that for any $\vareps > 0$, we have
  \begin{align}
    &\I_{\vareps}^{\mathrm{out}}(\varphi) = \I_{\vareps}^{1}(\varphi) + \I_{\vareps}^{2}(\varphi) \eqsp , \\
    &\I_{\vareps}^{1}(\varphi) = \I_{\vareps}(\varphi \1_{\cball{0}{R'}^\complementary}) \eqsp , \quad  \I_{\vareps}^{2}(\varphi) =  \I_{\vareps}(\varphi
  \1_{\msv^\complementary \cap \cball{0}{R'}}) \eqsp . 
  \end{align}
  Let $\vareps \in \ooint{0, \bvareps}$, we divide the rest of the proof into two parts. First, we
  bound $\I_{\vareps}^{1}(\varphi)$ and then
  $\I_{\vareps}^{2}(\varphi)$.
  \begin{enumerate}[wide, labelwidth=!, labelindent=0pt, label=(\alph*)]
  \item Let $u = (\mtt^k/\vareps - C_{\varphi, \Psi})^{1/\upalpha k}$ (which
    makes sense, since $\vareps < \mtt^k/(C_{\varphi, \Psi} + 1)$). Since
    $R' \geq R$ we have using \eqref{eq:cond_varphi} and that $u \geq 1$
    \begin{align}
      \I_{\vareps}^{1}(\varphi) &= \textstyle{C_1^{-1} \vareps^{-d/k} \int_{\cball{0}{R'}^\complementary} \varphi(x)\Psi(x) \exp[-\norm{F(x)}^k/\vareps]\rmd x}  \\
                                     &\leq \textstyle{C_1^{-1} C_{\varphi, \Psi} \vareps^{-d/k}  \int_{\cball{0}{R'}^\complementary} \exp[- (\mtt^k/\vareps - C_{\varphi, \Psi}) \norm{x}^{\upalpha k}] \rmd x} \\
                                       &\leq \textstyle{C_1^{-1} C_{\varphi, \Psi} \vareps^{-d/k}  \int_{\cball{0}{R' u}^\complementary} \exp[- \norm{x}^{\upalpha k}] \rmd x \eqsp .} \label{int_I_eta_1}
    \end{align}
    Let $C_{\upalpha} = \int_{\rset^d} \exp[- \norm{x}^{\upalpha k}] \rmd
    x$. Using that $u = (\mtt^k / \vareps - C_{\varphi, \Psi})^{1/\upalpha k}$,
    we have
    \begin{align}
      \I_{\vareps}^{1}(\varphi) &\textstyle{\leq  C_1^{-1} C_{\varphi, \Psi} \vareps^{-d/k}  \int_{\cball{0}{R' u}^\complementary} \exp[- \norm{x}^{\upalpha k}] \rmd x }\\
                                     &\textstyle{\leq  C_1^{-1} C_{\varphi, \Psi} \vareps^{-d/k} \int_{\rset^d} \exp[\norm{x}^{\upalpha k}/2] \exp[-\norm{x}^{\upalpha k}] \rmd x  \exp[-(R'u)^{\upalpha k}/2]} \\
      &\leq  2^{d/\upalpha k} C_\upalpha  C_1^{-1} C_{\varphi, \Psi} \exp[(R')^{\upalpha k}C_{\varphi, \Psi}/2] \vareps^{-d/k} \exp[-(R')^{\upalpha k}\mtt^k/(2\vareps) ] \leq A_1^1 \vareps^{-d/k} \exp[-\beta_1^1/\vareps] \eqsp , \label{eq:I_eta_1}
    \end{align}
    with
    \begin{equation}
      A_1^1 =  2^{d/\upalpha k} C_\upalpha  C_1^{-1} C_{\varphi, \Psi} \exp[(R')^{\upalpha k}C_{\varphi, \Psi}/2]  \eqsp, \qquad \beta_1^1 = (R')^{\upalpha k}\mtt^k/2 \eqsp . 
    \end{equation}
  \item Second, note that
    $\msk = \msv^\complementary \cap \cball{0}{R'}$ is
    bounded and closed, \ie \ $\msk$ is compact. Note that for any
    $x \in \msk$, $\norm{F(x)} > 0$, hence there exists $m > 0$ such that for
    any $x \in \msk$, $\norm{F(x)} \geq m$. In addition, we have that for any
    $x \in \msk$,
    \begin{equation}
      \abs{\varphi(x)}\Psi(x) \leq C_{\varphi, \Psi} \exp[C_{\varphi, \Psi} \norm{x}^{\upalpha k}] \leq C_\varphi \exp[C_{\varphi, \Psi} (R')^{\upalpha k}] \eqsp . 
    \end{equation}
    Therefore, we have
    \begin{equation}
      \I_{\vareps}^{2}(\varphi) \leq C_1^{-1} C_{\varphi, \Psi} \vareps^{-d/k}  \exp[C_{\varphi, \Psi}(R')^{\upalpha k}] \exp[-m^k/\vareps] \Leb(\msk) \eqsp ,
    \end{equation}
    where we recall that $\lambda(\msk)$ is the Lebesgue measure of $\msk$.
    Since $\msk \subset \cball{0}{R'}$ we have
    \begin{equation}
      \I_{\vareps}^{2}(\varphi) \leq \uppi^{d/2} (R')^d \Gamma^{-1}(d/2+1) C_1^{-1} C_{\varphi, \Psi}   \exp[(R')^{\upalpha k}]\vareps^{-d/k} \exp[-m^k/\vareps]  \leq A_1^2 \vareps^{-d/k} \exp[-\beta_1^2 / \vareps ]\eqsp , \label{eq:I_eta_2}
    \end{equation}
    where $\Gamma: \ \ooint{0,+\infty} \to \rset_+$ is given for any
    $s \in \ooint{0, +\infty}$ by
    $\Gamma(s) = \int_0^{+\infty} t^{s-1} \exp[-t] \rmd t$ and
    \begin{equation}
      A_1^2 = \uppi^{d/2} (R')^d \Gamma^{-1}(d/2+1) C_1^{-1} C_{\varphi, \Psi}   \exp[C_{\varphi, \Psi}(R')^{\upalpha k}] \eqsp , \qquad \beta_1^2 = m^k \eqsp . 
    \end{equation}
  \end{enumerate}
  We conclude the proof upon combining \eqref{eq:I_eta_1}, \eqref{eq:I_eta_2},
  letting $\beta_1 = \min(\beta_1^1, \beta_1^2)$ and $A_1 = A_1^1 + A_1^2$.
\end{proof}

\begin{lemma}
  \label{lemma:existence_U}
  Assume \rref{assum:F} and that $d \leq p$. Then there exist $N \in \nset$,
  $\{x_0^k\}_{k=1}^N \in (\rset^d)^N$ and $\msw_k \subset \rset^d$ open such
  that for any $k \in \{1, \dots, N\}$, $x_0^k \in \msw_k$,
  $F: \ \bar{\msw}_k \to F(\bar{\msw}_k)$ is a bi-Lipschitz homeomorphism, for
  any $x \in \msw_k$, $\rmd F(x)$ is injective and for any
  $j \in \{1, \dots, N\}$, $\bar{\msw}_k \cap \bar{\msw}_j = \emptyset$. In
  addition, $F^{-1}(0) = \cup_{k=1}^N \{x_0^k\}$.
\end{lemma}

\begin{proof}
  Since, for any $x \in F^{-1}(0)$, $\jac{x} > 0$ and $d \leq p$ there exists
  $r_x > 0$ such that for any $y \in \cball{x}{r_x}$, $F(y) = F(x)$ implies that
  $y=x$. Since $\lim_{\norm{x} \to +\infty} \norm{F(x)} = +\infty$ we have that
  $F^{-1}(0)$ is compact. Assume that $\calH^0(F^{-1}(0)) = +\infty$. Then,
  there exists $(x_k)_{k \in \nset}$ such that for any $k \in \nset$,
  $F(x_k) = 0$ and for any $j \in \{0, \dots, k-1\}$, $x_j \neq x_k$. Up to taking a subsequence, there exists $x^\star \in F^{-1}(0)$ such that
  $\lim_{k \to +\infty} x_k = x^\star$ and for any $k \in \nset$,
  $x_k \neq x^\star$. Hence, there exists $k \in \nset$ such that
  $x_k \in \cball{0}{r_{x^\star}}$ which is absurd. Hence
  $\calH^0(F^{-1}(0))<+\infty$. In what follows, we let $N = \calH^0(F^{-1}(0))$
  and denote $\{x_0^k\}_{k=1}^N \in (\rset^d)^N$ such that
  $F^{-1}(0) = \{x_0^k\}_{k=1}^N$. There exists
  $\{r_k'\}_{k=1}^N \in (\rset_+)^N$ such that for any
  $k, j \in \{1, \dots, N\}$,
  $\cball{x_0^k}{r_k'} \cap \cball{x_0^j}{r_j'} = \emptyset$. For any
  $k \in \{1, \dots, N\}$, we let
  $\msv_k = \ball{x_0^k}{\min(r_k', r_{x_0^k}/2)}$. Let $k \in \{1, \dots,
  N\}$. By construction, $F: \ \bar{\msv}_k \to F(\bar{\msv}_k)$ is bijective
  and continuous. Since $\bar{\msv}_k$ is compact, we have that
  $F: \ \bar{\msv}_k \to F(\bar{\msv}_k)$ is a homeomorphism. Since for any
  $k \in \{1, \dots, N\}$, $\jac{x_0^k} > 0$, there exists $m > 0$ such that for
  any $k \in \{1, \dots, N\}$ and $v \in \rset^d$,
  $v^\top H(x_0^k,x_0^k) v \geq m \normLigne{\v}^2$ with
  $H(x,y) = \rmD F(x)^\top \rmD F(y)$ for any $x, y \in \rset^d$. For any
  $k \in \{1, \dots, N\}$, there exists $\msw_k \subset \msv_k$ such that for
  any $x, y \in \msw_k$ we have
  $\normLigne{H(x,y) - H(x_0^k,x_0^k)}_2 \leq m/2$.  Therefore we have for any
  $x, y \in \msw_k$
  \begin{align}
    &\norm{F(x) - F(y)}^2 = \textstyle{\int_0^1 \int_0^1 \langle \rmD F(x+t(y-x)) (y-x), \rmD F(x+s(y-x)) (y-x) \rangle \rmd t \rmd s} \\
                         & \qquad \textstyle{= \int_0^1 \int_0^1 (y-x)^\top H(x_s,x_t)(y-x)\rmd t  \rmd s} \\
    & \qquad  \textstyle{=(y-x)^\top H(x_0^k,x_0^k)(y-x)+  \int_0^1 \int_0^1 (y-x)^\top (H(x_t,x_s) - H(x_0^k,x_0^k))(y-x)\rmd t \rmd s} \\
    & \qquad \geq (m/2) \norm{y-x}^2 \eqsp , 
  \end{align}
  where $x_t = x + t(y-x)$, which concludes the proof.
\end{proof}

\subsection{Quantitative control of the derivative}
\label{sec:quant-contr-deriv}

  \begin{lemma}
    \label{lemma:derivative_control}
    Under the same assumptions as \Cref{prop:lip_l}, there exist $\eta > 0$ and
    $P \in \poly{4}{\rset_+}$ such that for $x \in F^{-1}(0)$ and
    $t \in \ballinfty{0}{\eta}$ we have
    \begin{align}
      \norm{\rmD_t \chi(t,x)} &\leq (1 + M_{0,\varphi}+ M_{1,\varphi})(1 + M_{0,\Psi}+ M_{1,\Psi}) \\
      & \quad \times P(M_{1,F}, M_{2,F}, M_{3,F}, 1/m_{1,F}) \exp[P(M_{1,F}, M_{2,F}, M_{3,F}, 1/m_{1,F})] \eqsp ,
    \end{align}
    where we recall that $\chi$ is defined in \eqref{eq:Psi_def_nul} and for any
    $\ell \in \nset$, $i \in \{0,\dots,\ell\}$ and
    $f \in \rmc^\ell(\rset^d, \rset^p)$,
    $M_{i, f} = \sup \ensembleLigne{\normLigne{\rmD^i f(x)}}{x \in
      F^{-1}(\ballinfty{0}{\eta})}$.
  \end{lemma}

  \begin{proof}
    In this proof, for any $f: \ \rset^{m_0} \to \rset^{m_1}$ with
    $m_0, m_1 \in \nset$ differentiable, we denote $\rmd f$ its differential.    
    Recall that $\bar{\Phi}_t: \ \rset^d \to \rset^d$ is defined such that for
    any $x \in \rset^d$, $\bar{\Phi}_t(x) = x^{(p)}$ with $x^{(0)} = x$ and for
    any $i \in \{0, \dots, p-1\}$, $x^{(i+1)} = \Phi_{i+1}(t_{i+1}, x^{(i)})$,
    with $\Phi_{i+1}$ given by \eqref{eq:flow}.  The compact sets $\msk_0$ and
    $\msk_1$ are defined in the proof of \Cref{prop:lip_l}. Since for any
    $x \in F^{-1}(0)$, $\absLigne{\det(\rmD \bar{\Phi}_t^{-1}(x))} > 0$, we
    assume without loss of generality that
    $\det(\rmD \bar{\Phi}_t^{-1}(x)) > 0$. For ease of notation we denote
    $\varphi_\Psi = \varphi \times \Psi$. We have
    \begin{equation}
      M_{0, \varphi_\Psi} = M_{0, \varphi} M_{0, \Psi}  \eqsp , \qquad M_{1, \varphi_\Psi} = M_{0, \varphi} M_{1, \Psi} + M_{1, \varphi} M_{0, \Psi} \eqsp . 
    \end{equation}
    In addition, we have for any $x \in \rset^d$ and $t \in \msk_0$
    \begin{equation}
      \label{eq:Psi_def}
      \chi(t,x) =  \varphi_\Psi(\bar{\Phi}_t^{-1}(x)) \jacinv{\bar{\Phi}_t^{-1}(x)} \det(\rmD \bar{\Phi}_t^{-1}(x)) \eqsp . 
    \end{equation}
    We now control the first derivative of $\chi$. We divide the rest of the proof in three steps.
    
    \begin{enumerate}[wide, labelwidth=!, labelindent=0pt, label=(\alph*)]
    \item We start by providing upper bounds for $x \mapsto h_{i,j}(x)$,
      $x \mapsto \rmD h_{i,j}(x)$ and $x \mapsto \rmD^2 h_{i,j}(x)$ for any
      $i,j \in \{1, \dots, p\}$ where we recall that for any $x \in \msk_1$ we
      have
    \begin{equation}
     \{h_{i,j}(x)\}_{1 \leq i,j \leq p} = G(x)^{-1} = \Adj(G(x)) / \det(G(x)) =  \Adj(G(x)) / \jac{x}^2 \eqsp ,
    \end{equation}
    where $\Adj(G(x))$ is the adjugate of
    $G(x) = \{ \langle \nabla F_i(x), \nabla F_j(x) \rangle \}_{1 \leq i,j \leq p}$, where
    for any $x \in \msk_1$ and $i,j \in \{1, \dots, p\}$
    \begin{equation}
      \Adj(G(x)) = \{(-1)^{i+j} \det (G^{i,j}(x))\}_{1 \leq i,j \leq p} \eqsp , 
    \end{equation}
    where for any $i,j \in \{1, \dots, p\}$, $\det(G^{i,j}(x))$ is the $(i,j)$
    minor of $G(x)$. Hence, using the Cauchy-Schwarz inequality there exists
    $D_0 \geq 0$ such that for any $x \in \msk_1$ and $i,j \in \{1, \dots, p\}$
    \begin{equation}
      \label{eq:bound_hij}
      \abs{h_{i,j}(x)} \leq D_0 M_{1,F}^{2p-2}/m_{1,F}^2 \eqsp ,
    \end{equation}
    where we define
    \begin{equation}
      M_{1,F} = \sup \ensembleLigne{\norm{\nabla F_i(x)}}{x \in \msk_1, \ i \in \{1, \dots, p\}} \eqsp , \quad m_{1,F} = \inf \ensembleLigne{\jac{x}}{x \in \msk_1} \eqsp . 
    \end{equation}
    Recall that $m_{1,F} > 0$ since $\msk_1$ is compact and for any $x \in \msk_1$, $\jac{x} > 0$.
    We have that for any $x \in \msk_1$, $u \in \rset^d$ and $i, j \in \{1, \dots, p\}$
    \begin{equation}
      \label{eq:derivative_Gx}
      \rmd h_{i,j}(x)(u) = e_i^\top \rmd G(x)^{-1}(u) e_j = -e_i^\top G(x)^{-1} \rmd G(x)(u) G(x)^{-1} e_j \eqsp . 
    \end{equation}
    In addition, we have for any $x, u \in \rset^d$ and $i,j \in \{1, \dots, p\}$
    \begin{equation}
      \label{eq:derive_g}
      \rmd G_{i,j}(x)(u) = \rmd^2 F_i(x)(\nabla F_j(x), u) + \rmd^2 F_j(x)(\nabla F_i(x), u) \eqsp . 
    \end{equation}
    Combining this result and \eqref{eq:derivative_Gx} we get that there exists $C_1 \geq 0$ such that for any
    $x \in \msk_1$ and $i, j \in \{1, \dots, p\}$
    \begin{equation}
      \label{eq:bound_d_hij}
      \norm{\rmD h_{i,j}(x)} \leq C_1 M_{1,F}^{4p-3} M_{2,F}/m_{1,F}^4 \eqsp , 
    \end{equation}
    where
    $M_{2,F} = \sup \ensembleLigne{\normLigne{\nabla^2 F_i(x)}}{x \in \msk_1, \ i \in
      \{1, \dots, p\}}$.  Hence using \eqref{eq:def_f}, \eqref{eq:bound_hij} and
    \eqref{eq:bound_d_hij}, there exist $C_2, C_3 \geq 0$ such that for any
    $x \in \msk_1$ and $i \in \{1, \dots, p\}$
    \begin{equation}
      \label{eq:bound_g}
      \norm{g_i(x)} \leq C_2 M_{1,F}^{2p-1}/m_{1,F}^2 \eqsp , \quad \norm{\rmD g_i(x)} \leq C_3 M_{2,F}/m_{1,F}^2\{M_{1,F}^{4p-2}/m_{1,F}^2  + M_{1,F}^{2p-2} \} \eqsp .
    \end{equation}
    Similarly, for any $x \in \msk_1$, $u, v \in \rset^d$ and
    $i, j \in \{1, \dots, p\}$ we have
    \begin{align}
      \rmd^2 h_{i,j}(x)(u,v) &= e_i^{\top} \rmd^2G(x)^{-1}(u,v) e_j \\
                             &= e_i^{\top} \left\lbrace G(x)^{-1} \rmd G(x)(u) G(x)^{-1} \rmd G(x)(v) G(x)^{-1} \right. \\
                             & \qquad \left. + G(x)^{-1} \rmd G(x)(v) G(x)^{-1} \rmd G(x)(u) G(x)^{-1} \right. \\
      &\qquad \left. - G(x)^{-1} \rmd^2G(x)(u,v) G(x)^{-1} \right\rbrace e_j \eqsp . 
    \end{align}
    We also have for any $x, u, v \in \rset^d$ and $i, j \in \{1, \dots, p\}$
    \begin{align}
      \rmd^2 G_{i,j}(x) &=  \rmd^3 F_i(x)(\nabla F_j(x), u, v) + \rmd^3 F_j(x)(\nabla F_i(x), u, v) \\
      & \qquad + \rmd^2 F_i(x)(\rmd^2 F_j(x)(v), u)  + \rmd^2 F_j(x)(\rmd^2 F_i(x)(v), u) \eqsp . 
    \end{align}
    Hence there exist $P_3 \in \poly{4}{\rset_+}$ such that for any $x \in \msk_1$ we have
    \begin{equation}
      \label{eq:bound_h_2}
      \normLigne{\rmD^2 h_{i,j}(x)} \leq P_3(M_{1,F}, M_{2,F}, M_{3,F}, 1/m_{1,F}) \eqsp ,
    \end{equation}
    where
    $M_{3,F} = \sup \ensembleLigne{\normLigne{\nabla^3 F_i(x)}}{x \in \msk_1, \
      i \in \{1, \dots, p\}}$. Next, note that for any $i \in \{1, \dots, p\}$,
    we can choose $g_i$ such that $g_i \in \rmc^2(\rset^d, \rset^d)$ (and
    therefore $\Phi_i \in \rmc^{3,2}(\rset \times \rset^d, \rset^d)$ by \Cref{lemma:flow}).
    Combining this result, \eqref{eq:bound_g} and \eqref{eq:bound_h_2}, there
    exist $P_1 \in \poly{2}{\rset_+}$, $P_2 \in \poly{3}{\rset_+}$ and
    $P_3 \in \poly{4}{\rset_+}$ such that for any $x \in \msk_1$ and
    $i \in \{1, \dots, p\}$
    \begin{align}
      \label{eq:bound_total}
      &\norm{g_i(x)} \leq P_1(M_{1,F}, 1/m_{1,F}) \eqsp , \qquad \norm{\rmD g_i(x)} \leq P_2(M_{1,F}, M_{2,F}, 1/m_{1,F}) \eqsp , \\
      &\normLigne{\rmD^2 g_i(x)} \leq P_3(M_{1,F}, M_{2,F}, M_{3,F}, 1/m_{1,F}) \eqsp .
    \end{align}
  \item For any $x \in \msk_1$, $i \in \{1, \dots, p\}$ and $t_i \in \rset$ such
    that $\Phi_i(t_i, x) \in \msk_1$ we have $\abs{t_i} \leq 2 \eta$, since
    $F_i(\Phi_i(t_i,x)) = F_i(x) - t_i$. Hence combining this result, the fact
    that for any $x \in \rset^d$, $\Phi_i(0, x) = x$, \eqref{eq:flow} and
    \eqref{eq:bound_g}, for any $x \in \msk_1$, $i \in \{1, \dots, p\}$ and
    $t_i \in \rset$ such that $\Phi_i(t_i, x) \in \msk_1$ we have
    \begin{equation}
      \textstyle{
        \norm{\Phi_i(t_i,x)} \leq \norm{x} + \int_0^{t_i} \norm{\rmD_t \Phi_i(s, x)} \rmd s \leq \norm{x} + 2 \eta C_2 M_{1,F}^{2p-1}/m_{1,F}^2   \eqsp .
        }
    \end{equation}
    Therefore, there exists $C_4 \geq 0$ such that for any $x \in \msk_1$,
    $i \in \{1, \dots, p\}$ and $t_i \in \rset$ such that
    $\Phi_i(t_i,x) \in \msk_1$ we have
    \begin{equation}
      \label{eq:borne_Phi_i}
      \norm{\Phi_i(t_i,x)} \leq C_4 (1 + M_{1,F}^{2p-1}/m_{1,F}^2) \eqsp , \quad \norm{\rmD_t \Phi_i(t_i,x)} \leq C_2 M_{1,F}^{2p-1}/m_{1,F}^2  \eqsp . 
    \end{equation}
    Similarly, for any $x \in \msk_1$, $i \in \{1, \dots, p\}$ and
    $t_i \in \rset$ such that $\Phi_i(t_i, x) \in \msk_1$ we have
    \begin{align}
      \norm{\rmD_x \Phi_i(t_i,x)} &\textstyle{\leq \norm{\Id} + \int_0^{t_i} \norm{\rmD_{t,x} \Phi_i(s, x)} \norm{\rmD_{x} \Phi_i(s,x)} \rmd s} \\
                                      &\textstyle{\leq \norm{\Id} + \int_0^{t_i} \norm{\rmD g_i(\Phi_i(t_i,x))} \norm{\rmD_{x} \Phi_i(s,x)} \rmd s} \\
                                      &\textstyle{\leq \norm{\Id} + C_3 (M_{2,F}/m_{1,F}^2)\{M_{1,F}^{4p-2}/m_{1,F}^2  + M_{1,F}^{2p-2} \}\int_0^{t_i}  \norm{\rmD_{x} \Phi_i(t_i,x)} \rmd s \eqsp . }
    \end{align}
    Hence, using Gr\"{o}nwall's lemma, for any $x \in \msk_1$,
    $i \in \{1, \dots, p\}$ and $t_i \in \rset$ such that
    $\Phi_i(t_i, x) \in \msk_1$ we have
    \begin{equation}
      \norm{\rmD_x \Phi_i(t_i,x)} \leq  \norm{\Id}  \exp[2 \eta C_3 (M_{2,F}/m_{1,F}^2)\{M_{1,F}^{4p-2}/m_{1,F}^2  + M_{1,F}^{2p-2} \}] \eqsp . 
    \end{equation}
    Therefore, there exists $C_5 \geq 0$ such that for any $x \in \msk_1$,
    $i \in \{1, \dots, p\}$ and $t_i \in \rset$ with $\Phi_i(t_i, x) \in \msk_1$
    we have
    \begin{equation}
      \label{eq:borne_d_Phi_i}
      \norm{\rmD_x \Phi_i(t_i,x)} \leq C_5 \exp[C_5 (M_{2,F}/m_{1,F}^2)\{M_{1,F}^{4p-2}/m_{1,F}^2  + M_{1,F}^{2p-2} \}] \eqsp . 
    \end{equation}
    Using this result, \eqref{eq:flow} and \eqref{eq:bound_total}, there
    exists $C_6 \geq 0$ such that for any $x \in \msk_1$,
    $i \in \{1, \dots, p\}$ and $t_i \in \rset$ with $\Phi_i(t_i, x) \in \msk_1$
    we have
    \begin{equation}
      \label{eq:borne_d_Phi_i_time}
      \norm{\rmD_{t,x} \Phi_i(t_i,x)} \leq C_6 \exp[C_6 (M_{2,F}/m_{1,F}^2)\{M_{1,F}^{4p-2}/m_{1,F}^2  + M_{1,F}^{2p-2} \}] \eqsp . 
    \end{equation}
    Hence combining \eqref{eq:borne_Phi_i}, \eqref{eq:borne_d_Phi_i} and
    \eqref{eq:borne_d_Phi_i_time}, there exist $P_4 \in \poly{2}{\rset_+}$ and 
    $P_5 \in \poly{3}{\rset_+}$  such
    that for any $x \in \msk_1$, $i \in \{1, \dots, p\}$, $t_i \in \rset$ with
    $\Phi_i(t_i, x) \in \msk_1$ we have
    \begin{align}
      \label{eq:polynomial_Phi_int}
      &\normLigne{\rmD_t \Phi_i(t_i,x)} + \normLigne{\Phi_i(t_i,x)} \leq P_4(M_{1,F}, 1/m_{1,F}) \eqsp , \\
      &\normLigne{\rmD_x \Phi_i(t_i,x)} + \normLigne{\rmD_{t,x} \Phi_i(t_i,x)} \leq P_5(M_{1,F}, M_{2,F}, 1/m_{1,F}) \exp[P_5(M_{1,F}, M_{2,F}, 1/m_{1,F})] \eqsp . 
    \end{align}
        In addition, using \eqref{eq:flow} we have for
    any $i \in \{1, \dots, p\}$, $s \in \rset$, $x \in \rset^d$
    \begin{equation}
      \rmD_t \rmD_x^2 \Phi_i(s,x) = \rmD g_i(\Phi_i(s,x)) \rmD_x^2 \Phi_i(s,x) + \rmD^2 g_i(\Phi_i(s,x)) \rmD_x \Phi_i(s,x) \eqsp .
    \end{equation}
    Therefore, using \eqref{eq:bound_total} and \eqref{eq:polynomial_Phi_int},
    we have for any $x \in \msk_1$, $i \in \{1, \dots, p\}$ and $t_i \in \rset$
    such that $\Phi_i(t_i,x) \in \msk_1$
    \begin{align}
      &\textstyle{\normLigne{\rmD_x^2 \Phi_i(t_i,x)} \leq \int_0^{t_i} \{ \normLigne{\rmD g_i(\Phi_i(s,x))}\normLigne{\rmD_x^2 \Phi_i(s,x)} + \normLigne{\rmD^2 g_i(\Phi_i(s,x))}\normLigne{\rmD_x \Phi_i(s,x)} \} \rmd s} \\
                                             &\leq 2 \eta P_3(M_{1,F}, M_{2,F}, M_{3,F}, 1/m_{1,F})P_5(M_{1,F}, M_{2,F}, 1/m_{1,F})\exp[P_5(M_{1,F}, M_{2,F}, 1/m_{1,F})] \\
                                             & \qquad \textstyle{+ P_2(M_{1,F}, M_{2,F}, 1/m_{1,F}) \int_0^{t_i}  \normLigne{\rmD_x^2 \Phi_i(s,x)}  \rmd s \eqsp . }
    \end{align}
    Hence, using Gr\"{o}nwall's lemma, there exists
    $P_6 \in \poly{4}{\rset_+}$ such that for any $x \in \msk_1$,
    $i \in \{1, \dots, p\}$ and $t_i \in \rset$ such that
    $\Phi_i(t_i,x) \in \msk_1$ we have
    \begin{equation}
      \label{eq:borne_d_2_Phi_i}
      \normLigne{\rmD_x^2 \Phi_i(t_i,x)} \leq P_6(M_{1,F}, M_{2,F}, M_{3,F}, 1/m_{1,F}) \exp[P_6(M_{1,F}, M_{2,F}, M_{3,F}, 1/m_{1,F})] \eqsp . 
    \end{equation}
    Hence, combining this result and \eqref{eq:polynomial_Phi_int}, we have for
    any $x \in \msk_1$, $i \in \{1, \dots, p\}$ and $t_i \in \rset$ such that
    $\Phi_i(t_i,x) \in \msk_1$
    \begin{align}
      &\normLigne{\rmD_t \Phi_i(t_i,x)} + \normLigne{\Phi_i(t_i,x)} \leq P_4(M_{1,F}, 1/m_{1,F}) \eqsp , \\
      &\normLigne{\rmD_x \Phi_i(t_i,x)} + \normLigne{\rmD_{t,x} \Phi_i(t_i,x)} \leq P_5(M_{1,F}, M_{2,F}, 1/m_{1,F}) \exp[P_5(M_{1,F}, M_{2,F}, 1/m_{1,F})] \eqsp , \\
      &\normLigne{\rmD_x^2 \Phi_i(t_i,x)} \leq P_6(M_{1,F}, M_{2,F}, M_{3,F}, 1/m_{1,F}) \exp[P_6(M_{1,F}, M_{2,F}, M_{3,F}, 1/m_{1,F})] \eqsp .  \label{eq:polynomial_Phi}
    \end{align}
   \item  In what follows, we fix $t \in \ballinfty{0}{\eta}$ and use
    \eqref{eq:polynomial_Phi} to provide uniform bounds for
    $\rmD_t \bar{\Phi}^{-1}_t$, $\rmD_t \rmD \bar{\Phi}^{-1}_t$ and
    $\rmD \bar{\Phi}^{-1}_t$ on $F^{-1}(0)$. We introduce
    $\{\bar{\Phi}_{i,t}^{-1}\}_{i=1}^p$ such that for any $x \in F^{-1}(0)$ and
    $i \in \{1, \dots, p\}$
    \begin{equation}
      \bar{\Phi}_{t, i}^{-1}(x) = \Phi_{i}(-t_{i}, \bar{\Phi}_{t, i+1}^{-1}(x)) \eqsp , \qquad \bar{\Phi}_{t, p+1}^{-1}(x) = x \eqsp . 
    \end{equation}
    Note that $\bar{\Phi}_t^{-1} = \bar{\Phi}_{t,1}^{-1}$. Let $j \in \{1, \dots, p\}$.
    For any $i \in \{1, \dots, p\}$ we distinguish three cases:
    \begin{enumerate}[label=(\roman*)]
    \item $i>j$, then for any $x \in F^{-1}(0)$, $\rmD_{t_j} \bar{\Phi}_{t, i}^{-1}(x) = 0$.
    \item $i=j$, then for any $x \in F^{-1}(0)$, $\rmD_{t_j} \bar{\Phi}_{t, i}^{-1}(x) = - \rmD_t \Phi_j(-t_j, \bar{\Phi}_{t, j+1}^{-1}(x))$.
    \item $i<j$, then for any $x \in F^{-1}(0)$,
      $\rmD_{t_j} \bar{\Phi}_{t, i}^{-1}(x) = \rmD_x \Phi_i(-t_i,
      \bar{\Phi}_{t, i+1}^{-1}(x)) \rmD_{t_j} \bar{\Phi}_{t, i+1}^{-1}(x)$.
    \end{enumerate}
    Combining these results, \eqref{eq:polynomial_Phi} and the fact that
    $\bar{\Phi}_t^{-1} = \bar{\Phi}_{t,1}^{-1}$ we get that there exists
    $\bar{P}_7 \in \poly{3}{\rset_+}$ such that for any $x \in F^{-1}(0)$, and
    $j \in \{1, \dots, p\}$ we have
    \begin{equation}
      \normLigne{\rmD_{t_j} \bar{\Phi}_t^{-1}(x)} \leq  \bar{P}_7(M_{1,F}, M_{2,F}, 1/m_{1,F}) \exp[\bar{P}_7(M_{1,F}, M_{2,F}, 1/m_{1,F})]  \eqsp .
    \end{equation}
    Hence, there exists $P_7 \in \poly{3}{\rset_+}$ such that for any $x \in F^{-1}(0)$ we have
    \begin{equation}
      \label{eq:bound_d_time_bar_phi}
      \normLigne{\rmD_{t} \bar{\Phi}_t^{-1}(x)} \leq P_7(M_{1,F}, M_{2,F}, 1/m_{1,F}) \exp[P_7(M_{1,F}, M_{2,F}, 1/m_{1,F})] \eqsp . 
    \end{equation}
    Next, note that for any $x \in F^{-1}(0)$, $t \in \ballinfty{0}{\eta}$ and $i \in \{1, \dots, p\}$ we have 
    \begin{equation}
       \rmD \bar{\Phi}_{t, i}^{-1}(x) = \rmD_x \Phi_{i}(-t_{i}, \bar{\Phi}_{t, i+1}^{-1}(x)) \rmD \bar{\Phi}_{t, i+1}^{-1}(x) \eqsp , \qquad \rmD \bar{\Phi}_{t, p+1}^{-1}(x) = \Id \eqsp . 
    \end{equation}
    Combining this result and \eqref{eq:polynomial_Phi}, there exists
    $P_8 \in \poly{3}{\rset_+}$ such that for any $x \in F^{-1}(0)$ and $i \in \{1, \dots, p\}$
    \begin{equation}
      \label{eq:bound_d_space_bar_phi_i}
      \normLigne{\rmD \bar{\Phi}_{t, i}^{-1}(x)} \leq P_8(M_{1,F}, M_{2,F}, 1/m_{1,F}) \exp[P_8(M_{1,F}, M_{2,F}, 1/m_{1,F})] \eqsp . 
    \end{equation}
    In particular, since $\bar{\Phi}_{t} = \bar{\Phi}_{t, 1}$ we have for any
    $x \in F^{-1}(0)$
    \begin{equation}
      \label{eq:bound_d_space_bar_phi}
      \normLigne{\rmD \bar{\Phi}_{t}^{-1}(x)} \leq P_8(M_{1,F}, M_{2,F}, 1/m_{1,F}) \exp[P_8(M_{1,F}, M_{2,F}, 1/m_{1,F})] \eqsp . 
    \end{equation}    
  Finally, we give a uniform upper-bound on $\rmD_t \rmD
  \bar{\Phi}_t^{-1}$. Let $j \in \{1, \dots, p\}$. For any
  $i \in \{1, \dots, p\}$ we distinguish three cases:
  \begin{enumerate}[label=(\roman*)]
  \item $i>j$, then for any $x \in F^{-1}(0)$, $\rmD_{t_j}\rmD \bar{\Phi}_{t, i}^{-1}(x) = 0$.
  \item $i=j$, then for any $x \in F^{-1}(0)$, $\rmD_{t_j}\rmD \bar{\Phi}_{t, i}^{-1}(x) = -\rmD_{t,x}\Phi_j(-t_j, \bar{\Phi}_{t, j+1}^{-1}(x)) \rmD \bar{\Phi}_{t, j+1}^{-1}(x)$.
  \item $i<j$, then for any $x \in F^{-1}(0)$, we have
    \begin{align}
            \rmD_{t_j}\rmD \bar{\Phi}_{t, i}^{-1}(x) &= \rmD_x \Phi_{i}(-t_{i}, \bar{\Phi}_{t, i+1}^{-1}(x)) \rmD_{t_j} \rmD \bar{\Phi}_{t, i+1}^{-1}(x) \\ & \qquad + \rmD_x^2 \Phi_{i}(-t_{i}, \bar{\Phi}_{t, i+1}^{-1}(x)) (\rmD \bar{\Phi}_{t, i+1}^{-1}(x), \rmD_{t_j} \rmD \bar{\Phi}_{t, i+1}^{-1}(x)) \eqsp .
    \end{align}
      \end{enumerate}
      Using \eqref{eq:polynomial_Phi} and \eqref{eq:bound_d_space_bar_phi_i},
      there exists $P_9 \in \poly{4}{\rset_+}$ such that for any
      $x \in F^{-1}(0)$ we have
    \begin{equation}
      \label{eq:bound_d_spatial_time_bar_phi}
      \normLigne{\rmD_t \rmD \bar{\Phi}_t^{-1}(x)} \leq P_9(M_{1,F}, M_{2,F}, M_{3,F}, 1/m_{1,F}) \exp[P_9(M_{1,F}, M_{2,F}, M_{3,F}, 1/m_{1,F})]  \eqsp . 
    \end{equation}
    Therefore, summarizing \eqref{eq:bound_d_time_bar_phi},
    \eqref{eq:bound_d_space_bar_phi} and
    \eqref{eq:bound_d_spatial_time_bar_phi}, there exist
    $P_7, P_8 \in \poly{\rset_+}{3}$ and $P_9 \in \poly{4}{\rset_+}$ such that
    for any $x \in F^{-1}(0)$ and $t \in \ballinfty{0}{\eta}$ we have
    \begin{align}
      \label{eq:summary_phi}
      & \normLigne{\rmD_{t} \bar{\Phi}_t^{-1}(x)} \leq P_7(M_{1,F}, M_{2,F}, 1/m_{1,F}) \exp[P_7(M_{1,F}, M_{2,F}, 1/m_{1,F})] \eqsp , \\
      & \normLigne{\rmD \bar{\Phi}_{t}^{-1}(x)} \leq P_8(M_{1,F}, M_{2,F}, 1/m_{1,F}) \exp[P_8(M_{1,F}, M_{2,F}, 1/m_{1,F})] \eqsp , \\
      & \normLigne{\rmD_t \rmD \bar{\Phi}_t^{-1}(x)} \leq P_9(M_{1,F}, M_{2,F}, M_{3,F}, 1/m_{1,F}) \exp[P_9(M_{1,F}, M_{2,F}, M_{3,F}, 1/m_{1,F})] \eqsp .
    \end{align}
  \item Next, we use \eqref{eq:summary_phi} to conclude the proof by providing
    uniform upper-bounds on the differential $\chi$ on $F^{-1}(0)$ w.r.t. $t$,
    where $\chi$ given in \eqref{eq:Psi_def}. For any $x \in F^{-1}(0)$ and
    $t \in \ballinfty{0}{\eta}$ we have
    \begin{align}
      \label{eq:chi_der}
      \rmD_t \chi(t,x) &= \rmD_t \varphi_\Psi(\bar{\Phi}_t^{-1}(x)) \jacinv{\bar{\Phi}_t^{-1}(t)} \det(G(\bar{\Phi}_t^{-1}(x)))\\
                           &\qquad +\varphi_\Psi(\bar{\Phi}_t^{-1}(x)) \rmD_t \jacinv{\bar{\Phi}_t^{-1}(t)} \det(G(\bar{\Phi}_t^{-1}(x))) \\
                           &\qquad +\varphi_\Psi(\bar{\Phi}_t^{-1}(x)) \jacinv{\bar{\Phi}_t^{-1}(t)} \rmD_t \det(G(\bar{\Phi}_t^{-1}(x))) \eqsp . 
    \end{align}
    First, for any $x \in F^{-1}(0)$ and
    $t \in \ballinfty{0}{\eta}$ we have
    \begin{equation}
      \label{eq:last_bound_1}
      \absLigne{\varphi_{\Psi}(\bar{\Phi}_t^{-1}(x))} \leq M_{0,\varphi_\Psi} \eqsp , \qquad \jacinv{\bar{\Phi}_t^{-1}(x)} \leq 1/m_{1,F} \eqsp .
    \end{equation}
    In addition, using \eqref{eq:summary_phi}, there exists
    $P_{10} \in \poly{3}{\rset_+}$ such that for any $x \in \msk_1$ and
    $t \in \ballinfty{0}{\eta}$
    \begin{equation}
      \label{eq:last_bound_2}
      \absLigne{\det(\rmD \bar{\Phi}_t^{-1}(x))} \leq P_{10}(M_{1,F}, M_{2,F}, 1/m_{1,F}) \exp[P_{10}(M_{1,F}, M_{2,F}, 1/m_{1,F})] \eqsp . 
    \end{equation}
    For any $x \in F^{-1}(0)$ and $t \in \ballinfty{0}{\eta}$ we have
    \begin{equation}
      \rmD_t \varphi_\Psi(\bar{\Phi}_t^{-1}(x)) = \rmD \varphi_\Psi(\bar{\Phi}_t^{-1}(x)) \rmD_t \bar{\Phi}_t^{-1}(x)\eqsp . 
    \end{equation}
    Combining this result and \eqref{eq:summary_phi}, there exists
    $P_{11} \in \poly{3}{\rset_+}$ such that for any $x \in F^{-1}(0)$ and
    $t \in \ballinfty{0}{\eta}$ we have
    \begin{equation}
            \label{eq:last_bound_der_1}
      \normLigne{\rmD_t \varphi(\bar{\Phi}_t^{-1}(x))} \leq M_{1, \varphi_\Psi} P_{11}(, M_{1,F}, M_{2,F}, 1/m_{1,F})\exp[P_{11}(, M_{1,F}, M_{2,F}, 1/m_{1,F})] \eqsp . 
    \end{equation}
    For any $x \in F^{-1}(0)$, $t \in \ballinfty{0}{\eta}$  we have
    \begin{equation}
      \label{eq:der_jac}
      \rmD_t \jacinv{\bar{\Phi}_t^{-1}(x)} = \rmD_t \det(G(\bar{\Phi}_t^{-1}(x)))^{1/2} = \jacinvpower{\bar{\Phi}_t^{-1}(x)}{-1} \rmD_t \det(G(\bar{\Phi}_t^{-1}(x))^{-1}) \eqsp . 
    \end{equation}
    In addition, for any $x \in F^{-1}(0)$, $t \in \ballinfty{0}{\eta}$ and
    $h \in \rset^p$ we have
    \begin{equation}
      \rmD_t \det(G(\bar{\Phi}_t^{-1}(x)))(h) =   \trace(\Adj(G(\bar{\Phi}_t^{-1}(x)))\rmD G(\bar{\Phi}_t^{-1}(x))(\rmD_t\bar{\Phi}_t^{-1}(x)(h))) \eqsp . 
    \end{equation}
    Using this result, \eqref{eq:der_jac}, \eqref{eq:derive_g} and
    \eqref{eq:summary_phi}, there exists
    $P_{12} \in \poly{3}{\rset_+}$ such that for any $x \in F^{-1}(0)$,
    $t \in \ballinfty{0}{\eta}$ we have
    \begin{equation}
            \label{eq:last_bound_der_2}
      \normLigne{\rmD_t \jacinv{\bar{\Phi}_t^{-1}(x)}} \leq P_{12}(M_{1,F}, M_{2,F}, 1/m_{1,F})\exp[P_{12}(M_{1,F}, M_{2,F}, 1/m_{1,F})] \eqsp . 
    \end{equation}   
    Finally, for any $x \in F^{-1}(0)$, $t \in \ballinfty{0}{\eta}$ and $h \in \rset^p$ we have
    \begin{equation}
      \rmD_t \det(\rmD \bar{\Phi}_t^{-1}(x)) = \trace(\Adj(\rmD \bar{\Phi}_t^{-1}(x)) \rmD_t \rmD \bar{\Phi}_t^{-1}(x)(h)) \eqsp . 
    \end{equation}
    Hence, using \eqref{eq:summary_phi}, there exists
    $P_{13} \in \poly{4}{\rset_+}$ such that for any $x \in F^{-1}(0)$ and
    $t \in \ballinfty{0}{\eta}$ we have
    \begin{align}
      \label{eq:last_bound_der_3}
      \normLigne{\rmD_t \det(\rmD \bar{\Phi}_t^{-1}(x))} &\leq P_{13}(M_{1,F}, M_{2,F}, M_{3,F}, 1/m_{1,F})\\
      & \qquad \times \exp[P_{13}(M_{1,F}, M_{2,F}, M_{3,F}, 1/m_{1,F})] \eqsp . 
    \end{align}
    \end{enumerate}
    We conclude the proof upon combining \eqref{eq:chi_der},
    \eqref{eq:last_bound_1}, \eqref{eq:last_bound_2},
    \eqref{eq:last_bound_der_1}, \eqref{eq:last_bound_der_2} and
    \eqref{eq:last_bound_der_3}.
  \end{proof}

\subsection{Regularity results}
\label{sec:regularity-results}

In the following section we prove a smoothing lemma which is key to extend
\Cref{prop:final_d_geq_p} to the case where $\Psi$ and $\varphi$ are no longer
in $\rmc^1(\bar{\msu}, \rset)$ but are Lipschitz continuous.

\begin{lemma}
  \label{lemma:smoothing_lemma}
  Let $\msu$ be open bounded, $r >0$,
  $\msv = \msu +\ball{0}{r}$ and $\varphi \in \rmc(\rset^d, \rset)$ which
  satisfies \eqref{eq:cond_varphi} and is Lipschitz on $\msv$, \ie \ there
  exists $M_{1, \varphi} \geq 0$ such that for any $x, y \in \msv$ we have
  \begin{equation}
    \abs{\varphi(x) - \varphi(y)} \leq M_{1, \varphi} \norm{x-y} \eqsp .
  \end{equation}
  In addition, let
  $M_{0, \varphi} = \sup \ensembleLigne{\absLigne{\varphi(x)}}{x \in \msv}$.
  Then there exist $\bdelta > 0$ and
  $\ensembleLigne{\varphi^\delta}{\delta \in \oointLigne{0, \bdelta}}$ such
  that the following hold:
  \begin{enumerate}[wide, labelwidth=!, labelindent=0pt, label=(\alph*)]
  \item \label{item:primo} For any $x \in \msu$,
    $\lim_{\delta \to 0} \varphi^\delta(x) = \varphi(x)$.
  \item  \label{item:duo} For any $\delta \in \oointLigne{0, \bdelta}$,
    $\varphi^\delta \in \rmc^1(\rset^d, \rset)$ and there exists
    $\Lip_\delta \geq 0$ such that for any $x, y \in \msu$,
  \begin{equation}
    \abs{\varphi^\delta(x) - \varphi^\delta(y)} \leq \Lip_\delta \norm{x-y} \eqsp .
  \end{equation}
  Let
  $\Lip_0 = \sup \ensembleLigne{\Lip_\delta}{\delta \in \oointLigne{0, \bdelta}}
  < +\infty$ and we have
  $\Lip_0 \leq C_1 (1 + M_{0,\varphi} + M_{1,\varphi})(1 + C_\varphi)$, with
  $C_1 \geq 0$ that does not depend on $\varphi$.
\item \label{item:tertio} For any $\delta \in \oointLigne{0, \bdelta}$ there
  exists $M_{\delta} \geq 0$ such that for any $x \in \msu$,
  $\absLigne{\varphi^\delta(x)} \leq M_\delta$ and
  $M_0 = \sup \ensembleLigne{M_\delta}{\delta \in \oointLigne{0, \bdelta}} <
  +\infty$. In addition, $M_0 \leq C_2 (1 + M_{0,\varphi})(1 + C_\varphi)$ with
  $C_2 \geq 0$ that does not depend on $\varphi$.
\item \label{item:quattro} For any $\delta \in \oointLigne{0, \bdelta}$ there
  exists $D_\delta \geq 0$ such that for any $x \in \rset^d$  we have
  \begin{equation}
    \absLigne{\varphi^\delta(x)} \leq D_\delta \exp[D_\delta \norm{x}^{\upalpha k}] \eqsp ,
  \end{equation}
  and $D_\delta \leq C_3 (1 + C_\varphi) \exp[C_3C_\varphi]$ with $C_3$ that does not depend on $\varphi$.
\item \label{item:cinco} Assume \rref{assum:F}, that $F^{-1}(0) \subset \msu$
  and that there exists $m_{0, \varphi} > 0$ such that for any
  $x \in F^{-1}(0)$, $\varphi(x) \geq m_{0, \varphi}$. Then there exists $\bdelta^\star > 0$ such that for any
  $\delta \in \oointLigne{0, \bdelta^\star}$ and $x \in F^{-1}(0)$,
  $\varphi^\delta(x) \geq m_{0, \varphi}/2$. In addition if $\varphi \geq 0$
  then for any $\delta > 0$, $\varphi^\delta \geq 0$.
\end{enumerate}

In addition, $\bdelta = f_1(C_\varphi)$ and
$\bdelta^\star = f_2(M_{0, \varphi}, M_{1, \varphi}, 1/m_{0,\varphi},
C_\varphi)$ with $f_1 \in \rmc(\rset_+, \rset_+)$ and
$f_2 \in \rmc(\rset_+^4, \rset_+)$ and are non-decreasing w.r.t. to each of their
variables. Finally, there exists $C_4 \geq0$ such that
    \begin{equation}
      \limsup_{\delta \to 0} M_{\delta}  \leq C_4 M_{0,\varphi} \eqsp , \quad \limsup_{\delta \to 0} \Lip_{\delta} \leq C_4 M_{1,\varphi} \eqsp , \quad \limsup_{\delta \to 0} D_\delta \leq C_4  C_\varphi \eqsp ,
    \end{equation}
with $C_4$ that does not depend on $\varphi$.
\end{lemma}

\begin{proof}
  Since $\msu$ is bounded there exists $R \geq 0$ such that
  $\msu \subset \cball{0}{R}$.  Let $p \in \nsets$ such that $2p > \upalpha k$
  and for any $\delta >0$ define $k_{\delta}: \ \rset^d \to \rset_+$ such that
  for any $x \in \rset^d$ we have
  \begin{equation}
    \textstyle{k_{\delta}(x) = \exp[-\norm{x}^{2p}/\delta] / \int_{\rset^d} \exp[-\norm{\tilde{x}}^{2p}/\delta] \rmd \tilde{x} \eqsp .} 
  \end{equation}
  Let $C_1 = \int_{\rset^d} \exp[-\norm{\tilde{x}}^{2p}] \rmd \tilde{x}$ and we
  have for any $x \in \rset^d$,
  $k_{\delta}(x) = C_1^{-1} \delta^{-d/2p} \exp[-\norm{x}^{2p}/\delta]$.  Note
  that since $2p > \upalpha k$ and using \eqref{eq:cond_varphi} we have that for
  any $x \in \rset^d$
  $\int_{\rset^d} \absLigne{\varphi(x-y)}k_{\delta}(y) \rmd y < +\infty$. For
  any $\delta >0$ we define $\varphi^\delta$ such that for any $x \in \rset^d$
    \begin{equation}
      \label{eq:def_smoothing}
      \textstyle{
        \varphi^\delta(x) = \int_{\rset^d} \varphi(x-y) k_{\delta}(y) \rmd y \eqsp .
        }
      \end{equation}
      We divide the rest of the proof into five parts.
      \begin{enumerate}[wide, labelwidth=!, labelindent=0pt, label=(\alph*)]
      \item We have that for any $x \in \rset^d$,
        $\lim_{\delta \to 0} \varphi^\delta(x) = \varphi(x)$, since
        $\varphi \in \rmc(\rset^d, \rset)$ and
        $\ensembleLigne{k_{\delta}}{\delta > 0}$ is a mollifier.  This concludes
        the proof of \Cref{lemma:smoothing_lemma}-\ref{item:primo}.
      \item Using that $2p > \upalpha k$ and \eqref{eq:cond_varphi} we have that
        for any $\upbeta \in \nset$,
        $\int_{\rset^d} \norm{y}^{\upbeta} \absLigne{\varphi(x-y)}k_{\delta}(y)
        < +\infty$. Hence, $\varphi^\delta \in \rmc^\infty(\rset^d, \rset)$. Let
        $x \in \msu$.  Using that
        $\int_{\rset^d} y \normLigne{y}^{2p-1} k_\delta(y) \rmd y = 0$, we have
        that for any $\delta > 0$
    \begin{align}
      &\textstyle{\normLigne{\nabla \varphi^\delta(x)} = (2p/\delta) \normLigne{\int_{\rset^d} \varphi(x-y) y \normLigne{y}^{2p-2} k_{\delta}(y) \rmd y }} \\
                                       & \quad\textstyle{= (2p/\delta) \int_{\rset^d} \absLigne{\varphi(x-y) - \varphi(x)} \normLigne{y}^{2p-1} k_{\delta}(y) \rmd y}  \\
      &\quad \textstyle{\leq (2pM_{1, \varphi}/\delta)\int_{\rset^d}\normLigne{y}^{2p} k_{\delta}(y) \rmd y + (2p/\delta)\int_{\cball{0}{r/2}^\complementary} \absLigne{\varphi(x-y) - \varphi(x)} \normLigne{y}^{2p-1} k_{\delta}(y) \rmd y} \\
      &\quad \textstyle{\leq (2pM_{1, \varphi}/C_1) \int_{\rset^d}\normLigne{y}^{2p} \exp[-\normLigne{y}^{2p}] \rmd y }\\
      & \qquad \qquad \textstyle{+ (2p/\delta)\int_{\cball{0}{r/2}^\complementary} \absLigne{\varphi(x-y) - \varphi(x)} \normLigne{y}^{2p-1} k_{\delta}(y) \rmd y  \eqsp .} \label{eq:bound_grad}
    \end{align}
    We now bound the second term.  Let
    $\bar{\upalpha}_k = \ceilLigne{\upalpha k}$. Using \eqref{eq:cond_varphi},
    we have for any and $x \in \msu$ and $y \in \rset^d$ 
    \begin{align}
      \absLigne{\varphi(x-y) - \varphi(x)} &\leq (1 + M_{0, \varphi})C_\varphi \exp[C_\varphi\norm{x-y}^{\upalpha k}] \\
                                           &\leq (1 + M_{0, \varphi})C_\varphi \exp[3^{\bar{\upalpha}_k-1}C_\varphi(1 + \normLigne{x}^{\bar{\upalpha}_k}+\normLigne{y}^{\bar{\upalpha}_k})]\\
      &\leq (1 + M_{0, \varphi})C_\varphi \exp[3^{\bar{\upalpha}_k-1}C_\varphi(1 + R^{\bar{\upalpha}_k})]  \exp[3^{\bar{\upalpha}_k-1}C_\varphi\normLigne{y}^{\bar{\upalpha}_k}] \eqsp . 
    \end{align}
    Therefore we get that for any $x \in \msu$ and $y \in \cball{0}{r/2}^\complementary$
    \begin{align}
      &\absLigne{\varphi(x-y) - \varphi(x)} \norm{y}^{2p-1}\\
      & \qquad \qquad \qquad \leq (1 + M_{0, \varphi})C_\varphi \exp[3^{\bar{\upalpha}_k-1}C_\varphi(1 + R^{\bar{\upalpha}_k})] \exp[(3^{\bar{\upalpha}_k}(r/2)^{\bar{\upalpha}_k- 2p}C_\varphi + 1)\norm{y}^{2p}] \eqsp . 
    \end{align}
    Therefore there exists $C_a \geq 0$ such that for any $x \in \msu$ and
    $y \in \rset^d$ we have
    \begin{equation}
      \absLigne{\varphi(x-y) - \varphi(x)} \norm{y}^{2p-1} \leq C_a (1 + M_{0, \varphi})C_\varphi \exp[C_a (1 + C_\varphi) \norm{y}^{2p}] \eqsp ,
    \end{equation}
    with $C_a \geq 0$ that does not depend on $\varphi$.  Using this result, we have
    for any $x \in \msu$ and $\delta \in \ooint{0, 1/(2C_a(1 + C_\varphi))}$
    \begin{align}
      &\textstyle{(2p/\delta)\int_{\cball{0}{r/2}^\complementary} \absLigne{\varphi(x-y) - \varphi(x)} \norm{y}^{2p-1} k_{\delta}(y) \rmd y} \\
      &\qquad \qquad \qquad \textstyle{\leq 2pC_a(1+M_{0,\varphi})C_\varphi\delta^{-d/2p-1}C_1^{-1}\int_{\cball{0}{r/2}^\complementary}  \exp[(C_a(1 + C_\varphi)-1/\delta)\norm{y}^{2p}]  \rmd y} \\
        &\qquad \qquad \qquad \textstyle{\leq 2^{1+d(1+2p)/2p}pC_a(1+M_{0,\varphi})C_\varphi/(\delta C_1)\int_{\cball{0}{r/(2^{1 + 2p} \delta)^{1/2p}}^\complementary}  \exp[-\norm{y}^{2p}]  \rmd y \eqsp . }
    \end{align}
We have for any $\delta \in \ooint{0, 1/(2C_a(1 + C_\varphi))}$
\begin{equation}
  \textstyle{
    \int_{\cball{0}{r/(2^{1+2p}\delta)^{1/2p}}^\complementary}  \exp[-\norm{y}^{2p}]  \rmd y / C_1 \leq \exp[-r^{2p}/(2^{2 + 2p}\delta)]\int_{\rset^d}  \exp[-\norm{y}^{2p}/2]  \rmd y / C_1 \eqsp .
    }
\end{equation}
Therefore, there exists $C_b \geq0$ (that does not depend on $\varphi$) such that for any
$x \in \msu$ and $\delta \in \oointLigne{0,1/(2C_a(1+C_\varphi))}$ we have
\begin{equation}
  \label{eq:bound_outside_gradient}
  \textstyle{
    (2p/\delta)\int_{\cball{0}{r/2}^\complementary} \absLigne{\varphi(x-y) - \varphi(x)} \norm{y}^{2p-1} k_{\delta}(y) \rmd y \leq C_b(1+M_{0, \varphi})C_\varphi \delta^{-1} \exp[-r^{2p}/(2^{2 + 2p}\delta)] \eqsp .
    }
\end{equation}
Combining this bound and \eqref{eq:bound_grad}, we get that for any
$\delta \in \oointLigne{0, 1/(2C_a(1+C_\varphi))}$ there exists
$\Lip_\delta \geq 0$ such that for any $x \in \msu$,
$\normLigne{\nabla \varphi^\delta(x)} \leq \Lip_\delta$. Let
$\Lip_0 = \sup \ensembleLigne{\Lip_\delta}{\delta \in \oointLigne{0, \bdelta}} <
+\infty$ and using, \eqref{eq:bound_grad}, \eqref{eq:bound_outside_gradient} and
that for any $t \geq 0$, $t \exp[-r^{2p}t/2^{2+2p}] \leq 2^{2+2p}/(\rme r^{2p})$ we have
\begin{equation}
  \textstyle{
    \Lip_0 \leq (2pM_{1, \varphi}/C_1) \int_{\rset^d} \normLigne{y}^{2p} \exp[-\normLigne{y}^{2p}] \rmd y + 2^{2+2p}C_b(1+M_{0,\varphi})C_\varphi/(\rme r^{2p}) \eqsp .
    }
\end{equation}
This concludes the proof of \Cref{lemma:smoothing_lemma}-\ref{item:duo}.
\item For any $x \in \msu$ and $\delta \in \oointLigne{1/(2C_a(1+C_\varphi))}$ we have
\begin{align}
  \label{eq:decompo_order_zero}
  \absLigne{\varphi(x)} &\textstyle{\leq \int_{\cball{0}{r/2}} \absLigne{\varphi(y-x)} k_{\delta}(y) \rmd y + \int_{\cball{0}{r/2}^\complementary} \absLigne{\varphi(y-x)} k_{\delta}(y) \rmd y} \\
  &\textstyle{\leq M_{0, \varphi} + \int_{\cball{0}{r/2}^\complementary} \absLigne{\varphi(y-x)} k_{\delta}(y) \rmd y \eqsp . }
\end{align}
Similarly to \eqref{eq:bound_outside_gradient}, there exists $c \geq 0$
(that does not depend on $\varphi$) such that for $x \in \msu$ and
$\delta \in \oointLigne{1/(2C_a(1+C_\varphi))}$
\begin{equation}
  \textstyle{
  \int_{\cball{0}{r/2}^\complementary} \absLigne{\varphi(y-x)} k_{\delta}(y) \rmd y \leq (C_\varphi/c) \exp[-c/\delta] \eqsp . }
\end{equation}
Combining this result and \eqref{eq:decompo_order_zero} for any
$\delta \in \ooint{0, 1/(2C_a(1+C_\varphi))}$ there exists $M_\delta > 0$ such
that for any $x \in \msu$, $\absLigne{\varphi^\delta(x)} \le M_\delta$. Let
$M_0 = \sup \ensembleLigne{M_\delta}{\delta \in \oointLigne{0, \bdelta}} <
+\infty$. We have that
\begin{equation}
  M_0 \leq M_{0, \varphi} + ( C_\varphi / c) \exp[-c(1+C_\varphi)] \eqsp ,
\end{equation}
which concludes the proof of \Cref{lemma:smoothing_lemma}-\ref{item:tertio}.
\item Using that for any $a, b \geq 0$ and $p \geq 0$,
$(a+b)^{p} \leq 2^{\min(p-1, 0)}(a^p + b^p)$ we have that for any
$x \in \rset^d$ and $\delta > 0$
\begin{equation}
  \textstyle{
  \absLigne{\varphi^\delta(x)} \leq C_\varphi \int_{\rset^d} \exp[C_\varphi \beta_{\upalpha,k}(\norm{x}^{\upalpha k} + \norm{y}^{\upalpha k})] k_\delta(y) \rmd  y\eqsp , }
\end{equation}
where $\beta_{\upalpha, k} = 2^{\min(\upalpha k-1, 0)}$. Hence, using this
result we have for any $\delta \in \ooint{0,1/(2\beta_{\upalpha,k}C_\varphi)}$
\begin{align}
  \absLigne{\varphi^\delta(x)} &\textstyle{\leq C_\varphi /(C_1 \delta^{d/2p})\exp[C_\varphi \beta_{\upalpha,k}\norm{x}^{\upalpha k}] \int_{\rset^d} \exp[(C_\varphi \beta_{\upalpha, k} - 1/\delta) \norm{y}^{2p}] \rmd y} \\
  &\textstyle{\leq C_\varphi /(C_1 \delta^{d/2p})\exp[C_\varphi \beta_{\upalpha,k}\norm{x}^{\upalpha k}]  \int_{\rset^d} \exp[-\norm{y}^{2p}/(2\delta)] \rmd y \eqsp} \\
  &\textstyle{\leq 2^{d/2p} C_\varphi  \exp[C_\varphi \beta_{\upalpha,k}\norm{x}^{\upalpha k}]  \eqsp .}
\end{align}
Therefore, for any $\delta \in \ooint{0,1/(2 \beta_{\upalpha,k} C_\varphi)}$
there exists $D_\delta \geq 0$ such that for any $x \in \rset^d$,
$\varphi(x) \leq D_\delta \exp[D_\delta \norm{x}^{\upalpha k}]$. In addition,
there exists $C_d \geq 0$ (that does not depend on $\varphi$) such that
$D_\delta \leq C_d (1 + C_\varphi) \exp[C_d C_\varphi]$, which concludes the
proof of \Cref{lemma:smoothing_lemma}-\ref{item:quattro}.
\item If $\varphi \geq 0$ then for any $\delta > 0$, $\varphi^\delta \geq 0$
  using \eqref{eq:def_smoothing}. For any $x \in \msu$ we have
  \begin{equation}
    \label{eq:allewz}
    \textstyle{
      \varphi^\delta(x) = \varphi(x) + \int_{\rset^d} (\varphi(y-x) - \varphi(x)) k_{\delta}(y) \rmd y \geq m_{0, \varphi} - \int_{\rset^d} \abs{\varphi(y-x) - \varphi(x)} k_{\delta}(y) \eqsp .
      }
  \end{equation}
For any $x \in \msu$ we have
\begin{equation}
  \textstyle{
    \int_{\rset^d} \abs{\varphi(y-x) - \varphi(x)} k_{\delta}(y) \leq M_{1, \varphi} \delta^{1/k} + \int_{\cball{0}{r/2}^\complementary} \absLigne{\varphi(y-x) - \varphi(x)}k_\delta(y) \rmd y \eqsp .
    }
\end{equation}
Similarly to \eqref{eq:bound_outside_gradient}, there exists $C_d \geq 0$ that does not depend on $\varphi$ and such that for any $x \in \msu$
\begin{equation}
  \textstyle{
    \int_{\cball{0}{r/2}^\complementary} \absLigne{\varphi(y-x) - \varphi(x)}k_\delta(y) \rmd y \leq C_d(1 + M_{0, \varphi}) C_\varphi \exp[-1/(C_d \delta)] \eqsp .
    }
\end{equation}
Let $\bdelta_1 = (m_{0,\varphi}/(4M_{1,\varphi}))^k$ and
$\bdelta_2 = C_d^{-1}(\log(4C_dC_\varphi(1+M_{0,\varphi})) - \min(0,
\log(m_{0,\varphi})))^{-1}$. Then for any
$\delta \in \oointLigne{0, \min(\bdelta_1, \bdelta_2)}$ we have that for any
$x \in \msu$
\begin{equation}
  \textstyle{
    \int_{\rset^d} \abs{\varphi(y-x) - \varphi(x)} k_{\delta}(y) \leq m_{0,\varphi} / 2 \eqsp ,
    }
\end{equation}
which concludes the proof of \Cref{lemma:smoothing_lemma}-\ref{item:cinco} upon
combining this result with \eqref{eq:allewz}.
\end{enumerate}
\end{proof}
%%% Local Variables:
%%% mode: latex
%%% TeX-master: "main"
%%% End:

\section{Technical results for \Cref{sec:application_to_sgld}}
\label{sec:techn-results-crefs}

In this section, we derive a quantitative parametric theory for Laplace-type
expansion in \Cref{sec:quant-morse-lemma} and \Cref{sec:param-lapl-type}. We
start by deriving technical bounds in \Cref{sec:quant-morse-lemma}. Our main
result, \Cref{prop:conclusion_param_bounds}, is presented in 
\Cref{sec:param-lapl-type} along with a quantitative Morse lemma. Finally, in
\Cref{sec:sgld_moments} we derive some moments bounds.

Let $u: \ \rset^d \times \msz \to \rset$ with $\msz$ a metric space. We
consider the following assumption.

\begin{assumptionH}
  \label{assum:u_param}
  $u \in \rmc(\rset^d \times \msz, \rset)$, for any $z \in \msz$,
  $u(\cdot, z) \in \rmc^2(\rset^d, \rset)$ and the following hold:
  \begin{enumerate}[wide, labelwidth=!, labelindent=0pt, label=(\alph*)]
  \item $\msz$ is compact. 
  \item There exists $A \geq 0$ such that for any $z \in \msz$, $\absLigne{u(0,z)} \leq A$.
  \item There exists $\Mtt \geq 0$ such that for any $x_1, x_2 \in \rset^d$ and
    $z \in \msz$,
    $\normLigne{\nabla^k_x u(x_1,z) - \nabla^k_x u(x_2,z)} \leq \Mtt \normLigne{x_1-x_2}$.
  \item There exists $\mtt, \upalpha > 0$ and $R \geq 0$ such that for any $x \in \rset^d$
    with $\normLigne{x} \geq R$, $u(x,z) \geq \mtt \normLigne{x}^\upalpha$.
  \item For any $z \in \msz$ the number of global minimizers is bounded.
  \end{enumerate}
\end{assumptionH}

{There exists $R_A \geq 0$ such that for any $x \in \rset^d$ with
  $\normLigne{x} \geq R_A$ and $z \in \msz$,
  $\absLigne{u(x,z)} \geq A \geq \absLigne{u(0,z)}$. Hence, for any
  $z \in \msz$,
  $\argmin \ensembleLigne{u(x,z)}{x \in \rset^d} \subset \msc =
  \cball{0}{R_A}$. }

We denote $u^\star : \ \msz \to \rset$ such that for any $z \in \msz$,
$u^\star(z) = \min \ensembleLigne{u(x,z)}{x \in \rset^d}$.
{ We have that $u^\star \in \rmc(\msz, \rset)$. Indeed we have that $u$ is
uniformly continuous on $\msz \times \msc$. Hence, for any $\vareps >0$ there
exists $\delta >0$ such that for any $(z_0,x_0), (z_0,x_0) \in \msz \times \msc$
with $d(z_0,z_1) + \normLigne{x_0-x_1} \leq \delta$,
$\absLigne{u(x_0,z_0) - u(x_1,z_1)} \leq \vareps$. Let $z_0 \in \msz$ and
$x^\star_0 \in \argmin \ensembleLigne{u(x,z_0)}{x \in \rset^d} \subset
\msc$. Let $\vareps > 0$, $z \in \msz$ such that $d(z,z_0) \leq \delta$ and
$x^\star \in \argmin \ensembleLigne{u(x,z_0)}{x \in \rset^d} \subset \msc$ then
\begin{equation}
 u^\star(z) \leq u(x^\star,z) \leq u(x^\star_0,z) \leq u(x^\star_0,z) \leq u(x^\star_0,z_0) \leq u^\star(z_0) \eqsp .
\end{equation}
Similarly we have $u^\star(z_0) \leq u^\star(z)+ \vareps$ which concludes the
proof.  }For any $\varphi: \ \rset^d \to \rset_+$ and $\vareps > 0$ we define
\begin{align}
  \label{eq:i_vareps}
  &\textstyle{\I_{\vareps}(\varphi, z) = C_{\vareps}^{-1} \int_{\rset^d} \varphi(x) \exp[-(u(x,z) - u^\star(z))/\vareps] \rmd x \eqsp , \quad \J_{\vareps}(z) = \I_{\vareps}(1, z) \eqsp ,}\\
   &\textstyle{C_{\vareps} = \int_{\rset^d} \exp[-\norm{x}^2/\vareps] \rmd x = (\uppi \vareps)^{d/2} \eqsp .}
\end{align}
In addition, we define
\begin{equation}
  \label{eq:i_0}
  \textstyle{
    \I_0(\varphi, z) =  \int_{\argmin u(\cdot, z)} \varphi(x) \hessinvx{x}{z} \rmd \calH^{0}(x) \eqsp , \quad \J_0(z) = \I_0(1, z) \eqsp ,
    }
\end{equation}
where $\hessinvx{x}{z} = +\infty$ if $\nabla^2_xu(x,z)$ is not
invertible.  If $\varphi: \ \rset^d \to \rset$ and
$I_{\vareps}(\abs{\varphi}) < +\infty$ for some $\vareps \geq 0$ we define
$\I_{\vareps}(\varphi)$ similarly as in \eqref{eq:i_vareps}, \eqref{eq:i_0}. For
any $\vareps \geq 0$ and $\varphi: \ \rset^d \to \rset^p$ such that it is
defined we let
$\updelta_z \Sker_{\vareps}[\varphi] = \I_\vareps(\varphi, z) /
\J_\vareps(z)$. We emphasize that these definitions are the parametric
counterparts to the ones introduced in \Cref{sec:proof_thm_macro}.

In what follows, we define $\sigma: \ \rset^d \times \msz \to \ccint{0,+\infty}$
such that for any $z \in \msz$ and $x \in \rset^d$ $\sigma(x,z)$ is the inverse
of the minimum eigenvalue of $\nabla^2_x u(x,z)$. In addition, for any
$\upbeta > 0$, we define $\sigma^\star_\upbeta: \ \msz \to \ccint{0,+\infty}$ such that
for any $z \in \msz$,
$\sigma^\star_\upbeta(z) = \int_{\argmin u(\cdot, z)} \sigma(x, z)^\upbeta \rmd \calH^0(x)$.

\subsection{Parametric lower and truncation bounds}
\label{sec:quant-morse-lemma}

\begin{lemma}
  \label{prop:lower_bound_j_eps_param}
  Assume \rref{assum:u_param}. Then, for any $\bvareps \geq 0$ there exists
  $A_0 > 0$ such that for any $z \in \msz$ and $\vareps \in \ccint{0, \bvareps}$
  we have $\J_{\vareps}(z) \geq A_0$.
\end{lemma}

\begin{proof}
  Let $z \in \msz$ and $\vareps > 0$. Let $x^\star(z)$ be a global minimizer of
  $x \mapsto u(x,z)$ which exists since
  $\argmin \ensembleLigne{u(x,z)}{x \in \rset^d} \neq \emptyset$. Using the change of variable
  $x \mapsto x + x^\star(z)$ we have
  \begin{equation}
    \label{eq:change_j_eps}
    \textstyle{
    \J_{\vareps}(z) = C_{\vareps}^{-1} \int_{\rset^d} \exp[-(u(x,z) - u^\star(z))/\vareps] \rmd x = C_{\vareps}^{-1} \int_{\rset^d} \exp[-(u(x^\star(z)+x,z) - u^\star(z))/\vareps] \rmd x \eqsp . }
  \end{equation}
  For any $x \in \rset^d$ we have
  \begin{equation}
    \textstyle{
    u(x^\star(z)+x,z) - u^\star(z) = \int_0^1 \langle \nabla_x u(x^\star(z)+tx,z) - \nabla_x u(x^\star(z),z), x \rangle \rmd t \leq \Mtt \normLigne{x}^2/2 \eqsp . }
  \end{equation}
  Combining this result and \eqref{eq:change_j_eps} we get
  $\J_{\vareps}(z) \geq \int_{\rset^d} \exp[-(\Mtt/2)\normLigne{x}^2/\vareps]
  \rmd x / C_\vareps \geq (2 / \Mtt)^{d/2}$.  In addition, we have that for any
  $x \in \rset^d$,
  $\J_0(z) = \int_{\argmin u(\cdot, z)} \hessinvx{x}{z} \rmd \calH^{0}(x)$.  Note
  that if $\sigma_1^\star(z) = +\infty$, then there exists
  $\tilde{x}(z) \in \argmin \ensembleLigne{u(x,z)}{x \in \rset^d}$ such that
  $\sigma(\tilde{x}(z),z) = +\infty$ and therefore $\J_0(z) = +\infty$. We have
  that for any $x \in \rset^d$, $\det(\nabla^2_x u(x,z)) \leq \Mtt^d$ and
  therefore, $\J_0(z) \geq N \Mtt^{-d/2}$. We conclude upon letting
  $A_0 = \min(N, 2^{d/2})/\Mtt^{d/2}$.
\end{proof}

\begin{lemma}
  \label{prop:upper_bound_i_out_param}
  Assume \rref{assum:u_param}. Let $\varphi: \ \rset^d \to \rset$ and
  $C_\varphi \geq 0$ such that for any $x \in \rset^d$
  \begin{equation}
    \label{eq:cond_varphi_prop_param}
    \abs{\varphi(x)} \leq C_\varphi \exp[C_\varphi\norm{x}^{ \upalpha }] \eqsp . 
  \end{equation}
  Let $\bvareps \in \oointLigne{0, \mtt/(1 + C_{\varphi})}$ and $z \in \msz$.
  Assume that there exists $\msv(z) \subset \rset^d$ open and bounded such that
  \begin{equation}
   \argmin \ensembleLigne{u(x,z)}{x \in \rset^d} \subset \msv(z)  \subset \argmin \ensembleLigne{u(x,z)}{x \in \rset^d} + \cball{0}{1} \eqsp . 
  \end{equation}
  Then, there exist $\beta_1 > 0$ and $A_1 \in \rmc(\rset_+, \rset_+)$ such that
  for any $\vareps \in \ooint{0, \bvareps}$
  \begin{equation}
    \I_{\vareps}^{\mathrm{out}}(\varphi, z) \leq A_1(C_{\varphi}) \vareps^{-d/2} \{\exp[- m(z) / \vareps]+ \exp[-\beta_1/\vareps]\}\eqsp ,
  \end{equation}
  with $m(z) = \inf \ensembleLigne{u(x,z)}{x \in \rset^d\backslash \msv(z)} - u^\star(z)$, 
  $\I_{\vareps}^{\mathrm{out}}(\varphi, z) = \I_{\vareps}(\varphi
  \1_{\msv(z)^\complementary}, z)$ and $A_1, \beta_1$ that do not depend on $\varphi$ and
  $z$, with $A_1$ non-decreasing.
\end{lemma}

\begin{proof}
  {First, we have using the remark following \rref{assum:u_param} 
  \begin{equation}
    \msv(z) \subset \msc + \cball{0}{1} \eqsp . 
  \end{equation}
 Since $\msc + \cball{0}{1}$ is compact}, there exists $R' \geq 0$ (that does
  not depend on $z$) such that $R' \geq R$ (where $R$ is given in
  \Cref{assum:u_param}) and $\msv(z) \subset \cball{0}{R'}$. Note that for any
  $\vareps > 0$, we have
  \begin{align}
    &\I_{\vareps}^{\mathrm{out}}(\varphi, z) = \I_{\vareps}^{1}(\varphi, z) + \I_{\vareps}^{2}(\varphi, z) \eqsp , \\
    &\I_{\vareps}^{1}(\varphi, z) = \I_{\vareps}(\varphi \1_{\cball{0}{R'}^\complementary}, z) \eqsp , \quad  \I_{\vareps}^{2}(\varphi, z) =  \I_{\vareps}(\varphi
  \1_{\msv^\complementary \cap \cball{0}{R'}}, z) \eqsp . 
  \end{align}
  Let $\vareps \in \ooint{0, \bvareps}$. The rest of the proof is similar to the
  one of \Cref{prop:upper_bound_i_out} but is given for completeness. We divide
  the proof into two parts. First, we bound $\I_{\vareps}^{1}(\varphi, z)$ and
  then $\I_{\vareps}^{2}(\varphi, z)$.
  \begin{enumerate}[wide, labelwidth=!, labelindent=0pt, label=(\alph*)]
  \item Let $w = (\mtt/\vareps - C_{\varphi})^{1/\upalpha}$ (which makes
    sense, since $\vareps < \mtt/(C_{\varphi} + 1)$). Since $R' \geq R$, using \eqref{eq:cond_varphi_prop_param} and the fact that $u \geq 1$, we have
    \begin{align}
      \I_{\vareps}^{1}(\varphi,z) &\textstyle{= C_1^{-1} \vareps^{-d/2} \int_{\cball{0}{R'}^\complementary} \varphi(x) \exp[-(u(x,z)-u^\star(z))/\vareps]\rmd x  }\\
                                &\textstyle{\leq C_1^{-1} C_{\varphi} \vareps^{-d/2} \exp[\bar{u}^\star] \int_{\cball{0}{R'}^\complementary} \exp[- (\mtt/\vareps - C_{\varphi}) \norm{x}^{\upalpha}] \rmd x} \\
                                &\leq \textstyle{C_1^{-1} C_{\varphi}  \exp[\bar{u}^\star] \vareps^{-d/2}  \int_{\cball{0}{R' w}^\complementary} \exp[- \norm{x}^{ \upalpha}] \rmd x \eqsp ,} \label{int_I_eta_1_param_int}
    \end{align}
    where $\bar{u}^\star = \sup \ensembleLigne{u^\star(z)}{z \in \msz}$.  Let
    $C_{\upalpha} = \int_{\rset^d} \exp[- \norm{x}^{\upalpha}] \rmd x$. Using that
    $w = (\mtt / \vareps - C_\varphi)^{1/ \upalpha}$, we have
    \begin{align}
      \I_{\vareps}^{1}(\varphi,z) &\leq \textstyle{ C_1^{-1} C_{\varphi} \exp[\bar{u}^\star] \vareps^{-d/2}  \int_{\cball{0}{R' w}^\complementary} \exp[- \norm{x}^{ \upalpha}] \rmd x} \\
                                     &\leq \textstyle{C_1^{-1} C_{\varphi} \exp[\bar{u}^\star] \vareps^{-d/2} \int_{\rset^d}  \exp[-\norm{x}^{\upalpha}/2] \rmd x  \exp[-(R'u)^{\upalpha}/2]} \\
                                  &\leq  \textstyle{2^{d/ \upalpha} C_\upalpha  C_1^{-1} C_{\varphi} \exp[\bar{u}^\star] \exp[(R')^{\upalpha}C_\varphi/2] \vareps^{-d/2} \exp[-(R')^{\upalpha}\mtt/(2\vareps) ] }\\
      &\leq \textstyle{A_1^1 \vareps^{-d/2} \exp[-\beta_1^1/\vareps] \eqsp ,} \label{eq:I_eta_1_param}
    \end{align}
    with
    \begin{equation}
      A_1^1 =  2^{d/\upalpha} C_\upalpha  C_1^{-1} C_{\varphi} \exp[\bar{u}^\star] \exp[(R')^{\upalpha}C_\varphi/2]  \eqsp, \qquad \beta_1^1 = (R')^{\upalpha}\mtt/2 \eqsp . 
    \end{equation}
  \item Second, let $\msk(z) = \msv^\complementary(z) \cap \cball{0}{R'}$. Note
    that for any $x \in \msk(z)$, $u(x,z) - u^\star(z) \geq m(z)$.  Hence, we
    have
    \begin{equation}
      \I_{\vareps}^{2}(\varphi, z) \leq C_1^{-1} C_{\varphi} \vareps^{-d/2}  \exp[C_{\varphi}(R')^{\upalpha}] \exp[-m(z)/\vareps] \Leb(\msk(z)) \eqsp ,
    \end{equation}
    where we recall that $\lambda(\msk(z))$ is the Lebesgue measure of $\msk(z)$.
    Since $\msk(z) \subset \cball{0}{R'}$ we have
    \begin{align}
      \I_{\vareps}^{2}(\varphi, z) &\leq \uppi^{d/2} (R')^d \Gamma^{-1}(d/2+1) C_1^{-1} C_\varphi   \exp[(R')^{\upalpha}]\vareps^{-d/2} \exp[-m(z)/\vareps]  \\
      &\leq A_1^2 \vareps^{-d/2} \exp[-m(z) / \vareps ]\eqsp , \label{eq:I_eta_2_param}
    \end{align}
    where $\Gamma: \ \ooint{0,+\infty} \to \rset_+$ is given for any
    $s \in \ooint{0, +\infty}$ by
    $\Gamma(s) = \int_0^{+\infty} t^{s-1} \exp[-t] \rmd t$ and
    \begin{equation}
      A_1^2 = \uppi^{d/2} (R')^d \Gamma^{-1}(d/2+1) C_1^{-1} C_{\varphi}   \exp[C_{\varphi}(R')^{\upalpha}]  \eqsp . 
    \end{equation}
  \end{enumerate}
  We conclude the proof upon combining \eqref{eq:I_eta_1_param},
  \eqref{eq:I_eta_2_param} and letting $\beta_1 = \beta_1^1$ and
  $A_1 = A_1^1 + A_1^2$.
\end{proof}

\subsection{Quantitative Morse lemma and parametric Laplace-type results}
\label{sec:param-lapl-type}

We begin by recalling a quantitative version of the Morse lemma, see
\cite[Theorem 4.2]{le2014numerical}.

\begin{proposition}
  \label{prop:quantitative_morse_lemma}
  Let $\msu \subset \rset^d$ be open with $x_0 \in \msu$ and
  $f \in \rmc^k(\msu, \rset)$ with $k \in \nset$ and $k \geq 3$. Assume that
  $\nabla f(x_0) = 0$ and let $K \geq 0$ such that 
  $K \geq \max_{j \in \{1, \dots, k\}} \sup
  \ensembleLigne{\normLigne{\rmD^jf(x)}}{x \in \msu} < +\infty$. Let $\sigma$ be
  the minimal eigenvalue of $\nabla^2 f(x_0)$ and assume that $\sigma > 0$. Let
  $\sigma_0 = \min(\sigma, 1)$. Then, there exist $c_0 > 0$,
  $\delta = c_0 \sigma_0^4$ and
  $\Phi: \ \ball{0}{\delta} \to \Phi(\ball{0}{\delta})$ such that the following
  hold:
  \begin{enumerate}[wide, labelwidth=!, labelindent=0pt, label=(\alph*)]
  \item $\Phi\in \rmc^{k-1}(\ball{0}{\delta}, \Phi(\ball{0}{\delta}))$ is a diffeomorphism.
  \item For any $x \in \ball{0}{\delta}$, $f(\Phi(x)) = f(x_0) + \normLigne{x}^2$.
  \item There exist $c_1, \upbeta > 0$ such that $\max_{j \in \{1, \dots, k-1\}} \sup \ensembleLigne{\normLigne{\rmD^j \Phi(x)}}{x \in \ball{0}{\delta}} \leq c_1 \sigma_0^{-\upbeta}$.
  \end{enumerate}
  In addition, $c_0, c_1$ and $\upbeta$ depend only on $k$, $d$ and $K$.
\end{proposition}

Note that in \cite[Theorem 4.2]{le2014numerical}, the constant
$c_1 \sigma_0^{-\upbeta}$ is replaced by $M(K, \sigma_0, k)$ where
$M: \ \rset_+^3 \to \rset_+$. However, a close examination of the proof reveals
that the dependency of $M(K, \sigma_0, k)$ with respect to $\sigma_0^{-1}$ is of
order $\sigma_0^{-\upbeta}$ for some $\upbeta >0$ which can be made
explicit. Using \Cref{prop:quantitative_morse_lemma} we derive the following
parametric version of \Cref{prop:upper_bound_i_laplace}.

This proposition relies on the notion of \energygap \  associated with $u$
which we define as follows. Assume that
$\argmin \ensembleLigne{u(x,z)}{x \in \rset^d} \neq \emptyset$ for any
$z \in \msz$. For any $z \in \msz$ we introduce $\msa(z)$ such that
$\msa(z) = \emptyset$ if there are no other minimizers than the global
minimizers and $\msa(z) = \ensembleLigne{u(x,z)}{\text{$x$ is a local minimizer of $u(\cdot,z)$
    but not a global minimizer}}$ otherwise. Then we define
$c^\star: \ \msz \to \rset$ such that for any $z \in \msz$ we have
\begin{equation}
  \label{eq:locglob}
  c^\star(z) = \inf \msa(z) - \inf \ensembleLigne{u(x,z)}{x \in \rset^d} \eqsp ,
\end{equation}
with the convention that $\inf \emptyset = +\infty$. In words, the \energygap \
constant quantifies how close the value of the local minimizers are from the
global ones. We refer to \Cref{sec:import-local-glob} for a discussion on the
importance of \energygap \ when establishing parametric Laplace-type expansions.

\begin{proposition}
  \label{prop:upper_bound_i_laplace_param}
  Assume \rref{assum:u_param}. Let $\varphi \in \rmc(\rset^d, \rset)$ and
  $z \in \msz$, and assume that $\sigma_1^\star(z) < +\infty$.  There exists $\msv(z)$ open such that
  \begin{equation}
    \argmin \ensembleLigne{u(x,z)}{x \in \rset^d} \subset \msv(z) \subset \argmin \ensembleLigne{u(x,z)}{x \in \rset^d} + \cball{0}{1} \eqsp .
  \end{equation} We have that 
  $\lim_{\vareps \to 0} \absLigne{\I_\vareps^{\mathrm{in}}(\varphi, z) -
    \I_0(\varphi, z)} = 0$, with
  $\I_\vareps^{\mathrm{in}}(\varphi, z) = \I_{\vareps}(\varphi \1_{\msv(z)},
  z)$.  
  Assume that $\varphi \in \rmc^1(\rset^d, \rset)$. Then there exist
  $B_1 \geq 0$ and $\upbeta > 0$ such that for any $\vareps > 0$ we have
  \begin{equation}
    \label{eq:bound_I_vareps}
    \absLigne{\I_\vareps^{\mathrm{in}}(\varphi, z) - \I_0(\varphi, z)} \leq B_1 \sigma_\upbeta^\star(z) (1 + M_{0,\varphi}(z) + M_{1,\varphi}(z)) \vareps^{1/2} \eqsp , 
  \end{equation}
 with $B_1$ that does not depend on $\varphi$ and $z$, and for any $i \in \{0,1\}$, 
  $M_{i, \varphi}(z) = \sup \ensembleLigne{\normLigne{\nabla^i \varphi(x)}}{x \in \msv(z)}$.
  In addition, there exist $c_1, \upgamma > 0$ (that do not depend on $z$) such that
  $m(z) \geq \min(c_1/\sigma_\upgamma^\star(z), c^\star(z))$, with
  $m(z) = \inf \ensembleLigne{u(x,z)}{x \in \rset^d\backslash \msv(z)} -
  u^\star(z)$.
\end{proposition}

\begin{proof}
  Let
  $\{x_\star^k(z)\}_{k=1}^{M} = \argmin \ensembleLigne{u(x,z)}{x \in
    \rset^d}$. 
  Using \rref{assum:u_param} $\{x_\star^k(z)\}_{k=1}^{M} \subset \msc$ with
  $\msc$ compact. {We define $K_{\mathrm{global}} \geq 0$ such that
  \begin{equation}
    K_{\mathrm{global}} = \sup \ensembleLigne{\normLigne{\nabla_x^j u(x,z)}}{x \in \msk, \ z \in \msz, \ j \in \{1, 2, 3\}} \eqsp ,
  \end{equation}
  with $\msk = \msc + \cball{0}{1}$, where $\msc$ is defined in the remark
  following \rref{assum:u_param}}.  Let $i \in \{1, \dots, M\}$. Using that
$\sigma_1^\star(z) < +\infty$ and \Cref{prop:quantitative_morse_lemma} with
$K \leftarrow K_{\mathrm{global}}$ and $k = 3$, there exist
$\Phi_i^z \in \rmc^2(\ball{0}{\delta_i(z)}, \Phi_i^z(\ball{0}{\delta_i(z)}))$
which is a diffeomorphism, with $\Phi_i^z(0) = x_\star^i(z)$,
$\delta_i(z) = c_0 \min(\sigma^4(x_\star^i(z), z), 1)$ and for any
$x \in \ball{0}{\delta_i(z)}$ we have
$u(\Phi_i^z(x), z) = u^\star + \normLigne{x}^2$.

  We let
  $\delta_0(z) = \min \ensembleLigne{\delta_i(z)}{i \in \{1, \dots, M\}}$.
  Using \Cref{prop:quantitative_morse_lemma}, there exist $c_0'\geq0$ and
  $\upalpha > 0$ that do not depend on $z$ such that for any
  $\ell \in \{1, \dots, M\}$ and $x \in \ball{0}{\delta_0(z)}$,
  $\norm{\rmd \Phi_\ell^z(x)} \leq c_0' \sigma_\upalpha^\star(z)$. Let
  $\delta(z) = \min(\delta_0(z), 1/(c_0' \sigma_\upalpha^\star(z)))$. We have
  that for any $\ell \in \{1, \dots, N\}$,
  $\Phi_\ell(\ball{0}{\delta(z)}) \subset \ball{x_\star^\ell(z)}{c_0'
    \sigma_\upalpha^\star(z) \delta(z)} \subset \argmin \ensembleLigne{u(x,z)}{x
    \in \rset^d} + \cball{0}{1}$. We let
  $\msv(z) = \cup_{i=1}^M \Phi_\ell(\ball{0}{\delta(z)})$.
  
  We now show that for any $i, j \in \{1, \dots, M\}$,
  $\Phi_i^z(\ball{0}{\delta(z)}) \cap \Phi_j^z(\ball{0}{\delta(z)}) =
  \emptyset$. Let $i, j \in \{1, \dots, M\}$ and for ease of notation let
  $\msw_\ell = \Phi_\ell^z(\ball{0}{\delta(z)})$ for any
  $\ell \in \{1, \dots, M\}$. Assume that $\msw_i \cap \msw_j \neq
  \emptyset$. Then, since $\msw_i$ and $\msw_j$ are connected,
  $\msw = \msw_i \cup \msw_j$ is connected as well. In addition, note that
  $\Phi_i^z(0) \notin \msw_j$ and $\Phi_j^z(0) \notin \msw_i$. There exists
  $\gamma \in \rmc(\ccint{0,1}, \msw)$ such that $\gamma(0) = \Phi_i^z(0)$ and
  $\gamma(1) = \Phi_j^z(0)$. Denote
  $t^\star = \inf \ensembleLigne{t \in \ccint{0,1}}{\gamma(t) \in \msw_j}$. We
  have that $\gamma(t^\star) \in \bar{\msw}_j \backslash \msw_j$. Hence
  $u(\gamma(t^\star), z) = \delta(z)^2 + u^\star(z)$. But
  $\gamma(t^\star) \in \msw_i$ and therefore
  $u(\gamma(t^\star), z) < \delta(z)^2 + u^\star(z)$. This is absurd hence for
  any $i,j \in \{1, \dots, M\}$, $\msw_i \cap \msw_j = \emptyset$.

  Let $\vareps > 0$ and $\I_{0, \vareps}^{\mathrm{in}}$ be given by
  \begin{equation}
    \label{eq:def_I_0_eps_inter}
    \textstyle{
   \I_{0, \vareps}^{\mathrm{in}}(\varphi, z) = \sum_{\ell=1}^M
   \int_{\ball{0}{\delta(z)/\vareps^{1/2}}} \exp[-\norm{x}^2] \rmd x \varphi(x_\star^\ell(z)) \hessinvx{x_\star^\ell(z)}{z}/C_1 \eqsp .
   }
\end{equation}
Recall that $\I_{0}(\varphi, z) = \sum_{\ell=1}^M \varphi(x_\star^\ell(z)) \hessinvx{x_\star^\ell(z)}{z}$.
Therefore,  we have
\begin{align}
  \label{eq:inter_truc}
  \textstyle{
    \abs{\I_{0}(\varphi, z) - \I_{0, \vareps}^{\mathrm{in}}(\varphi, z)}} &\leq \textstyle{\sigma_{1/2}^\star M_{0,\varphi} \exp[-\delta(z)^2/(2\vareps)] \int_{\rset^d} \exp[-\normLigne{x}^2/2] \rmd x / C_1}\\
  &\leq 2^{d/2}  \sigma_{1/2}^\star M_{0,\varphi} \exp[-\delta(z)^2/(2\vareps)] \eqsp .
\end{align}
Finally, using \eqref{eq:def_I_0_eps_inter}, that for any
$\ell \in \{1, \dots, M\}$ and $x \in \ball{0}{\delta(z)}$,
$u(\Phi_\ell^z(x), z) = u^\star(z) + \normLigne{x}^2$, that for any
$\ell \in \{1, \dots, M\}$,
$\det(\rmd \Phi_\ell^z(0)) = \hessinvx{x_\star^\ell(z)}{z}$ and
$\Phi_\ell^z(0)= x_\star^\ell(z)$, we have 
\begin{align}
  \label{eq:the_big_diff}
  \textstyle{\absLigne{\I_{0, \vareps}^{\mathrm{in}}(\varphi, z) - \I_{\vareps}^{\mathrm{in}}(\varphi, z)} }&\textstyle{= \sum_{\ell=1}^M
  \int_{\ball{0}{\delta(z)/\vareps^{1/2}}} \left| \vphantom{ \varphi(\Phi_\ell^z(\vareps^{1/2} x)) \det(\rmd \Phi_\ell^z(\vareps^{1/2} x))}\varphi(\Phi_\ell^z(0)) \det(\rmd \Phi_\ell^z(0)) \right.} \\
  &\qquad \qquad \textstyle{\left. - \varphi(\Phi_\ell^z(\vareps^{1/2} x)) \det(\rmd \Phi_\ell^z(\vareps^{1/2} x)) \right|  \exp[-\norm{x}^2] \rmd x /C_1 \eqsp ,}
\end{align}
which concludes the first part of the proof upon combining this result and the
dominated convergence theorem. Next assume that
$\varphi \in \rmc^1(\rset^d, \rset)$ and for any $\ell \in \{1, \dots, M\}$, let
$\chi_\ell^z: \ \ball{0}{\delta(z)} \to \rset$ given for any
$x \in \ball{0}{\delta(z)}$ by
\begin{equation}
  \chi_\ell^z(x) = \varphi(\Phi_\ell^z(x)) \det(\rmd \Phi_\ell^z(x))
\end{equation}
We have that for any $\ell \in \{1, \dots, M\}$,
$\chi_\ell^z \in \rmc^1(\ball{0}{\delta})$. Therefore, we have that for any
$\ell \in \{1, \dots, M\}$ and $x \in \ball{0}{\delta(z)}$ and $h \in \rset^d$
\begin{equation}
  \rmd \chi_\ell(x)(h) = \rmd \varphi(\Phi_\ell^z(x)) \rmd \Phi_\ell^z(x)(h) \det(\rmd \Phi_\ell^z(x)) + \varphi(\Phi_\ell^z(x)) \trace(\Adj(\rmd \Phi_\ell^z(x))\rmd^2 \Phi_\ell^z(x)(h)) \eqsp . 
\end{equation}
Therefore, using \Cref{prop:quantitative_morse_lemma} we have that there exist
$C \geq 0$ and $\upbeta > 0$ such that for any $\ell \in \{1, \dots, M\}$ and
$x \in \ball{0}{\delta(z)}$,
$\normLigne{\rmd \chi_\ell^z(x)} \leq C (1 + M_{0, \varphi} + M_{1, \varphi})
\sigma^\star_\upbeta(z)$. Using this result in \eqref{eq:the_big_diff} we get that
\begin{equation}
  \textstyle{
    \abs{\I_{0, \vareps}^{\mathrm{in}}(\varphi, z) - \I_{\vareps}^{\mathrm{in}}(\varphi, z)} \leq C \sigma^\star_\upbeta(z) \vareps^{1/2} (1 + M_{0, \varphi} + M_{1, \varphi}) \int_{\rset^d} \norm{x} \exp[-\norm{x}^2] \rmd x / C_1 \eqsp ,
    }
\end{equation}
which concludes the proof of \eqref{eq:bound_I_vareps} upon combining this
result, \eqref{eq:inter_truc}, and the fact that $\exp[-t] \leq 1/t$ for any
$t > 0$.

Next, we show that there exist $c_1, \upbeta > 0$ such that
$m(z) \geq c_1 \min(1/\sigma_\beta^\star(z), c^\star(z))$ with
$c_1, \upbeta > 0$ that do not depend on $z$. Since
$\lim_{\normLigne{x} \to +\infty} u(x,z) = +\infty$, there exists $\tilde{x}(z)$
which minimizes $x \mapsto u(x,z)$ on $\mse = \rset^d\backslash \msv(z)$. We
distinguish two cases. If $\tilde{x}(z) \in \interior(\mse)$ then $\tilde{x}(z)$
is a local minimizer of $x \mapsto u(x,z)$. Hence $m(z) \geq c^\star(z)$. If
$\tilde{x}(z) \in \mse \backslash \interior(\mse)$ then
$\tilde{x}(z) \in \bar{\msv}(z)\backslash \msv(z)$ and we have that
$m(z) \geq \delta(z)^2$, which concludes the proof.
\end{proof}

Finally using \Cref{prop:lower_bound_j_eps_param},
\Cref{prop:upper_bound_i_out_param} and \Cref{prop:upper_bound_i_laplace_param}
we establish our main result.

\begin{proposition}
  \label{prop:conclusion_param_bounds}
  Assume \rref{assum:u_param}. Let $\varphi: \ \rset^d \to \rset$ be a
  $M_{1,\varphi}$-Lipschitz function, $M_{1,\varphi}, C_\varphi \geq 0$ and
  $z \in \msz$. Assume that $\sigma_1^\star(z) < +\infty$ and that for any
  $x \in \rset^d$,
  $\abs{\varphi(x)} \leq C_\varphi \exp[C_\varphi\norm{x}^{ \upalpha }]$.  Then,
  there exist $B_2 \in \rmc(\rset_+, \rset_+)$ and $\upbeta > 0$ such that
      \begin{equation}
        \abs{\updelta_z \Sker_\vareps[\varphi] - \updelta_z \Sker_0[\varphi]} \leq B_2(C_\varphi) (1 + M_{0,\varphi} + M_{1, \varphi}) (1 + \sigma_{\upbeta}^\star(z)) \{ \vareps^{1/2} + \vareps^{-d/2} \exp[-c^\star(z)/\vareps]\} \eqsp ,
      \end{equation}
      with 
      $M_{0, \varphi} = \sup \ensembleLigne{\absLigne{\varphi(x)}}{x \in \msk}$,
      $\msk$, $B_2$ and $\upbeta$ that do not depend on $z$, and $B_2$ non-decreasing.
\end{proposition}

\begin{proof}
  In this proof we assume that $\varphi \in \rmc^1(\rset^d, \rset)$. The
  extension to Lipschitz function is similar to the proof of
  \Cref{thm:big_theo_extension}, see \Cref{sec:proof-crefthm:big}, \ie \ we use
  the smoothing \Cref{lemma:smoothing_lemma}. Let
  $\bvareps \in \ooint{0, \mtt/(1+C_\varphi)}$. {First, let
    $\msv(z)$ be given by \Cref{prop:upper_bound_i_laplace_param} and set $\msk$
    such that $\msk= \msc + \cball{0}{1}$, with $\msc$ given in the remark
    following \rref{assum:u_param}.}
Note that $\msv(z) \subset \msk$. Applying
  \Cref{prop:lower_bound_j_eps_param} there exists $A_0 \geq 0$ such that for
  any $\vareps \in \ccint{0, \bvareps}$,
  \begin{equation}
    \label{eq:lower_param}
    \J_\vareps(z) \geq A_0 \eqsp .
  \end{equation}
  In addition, using \Cref{prop:upper_bound_i_out_param} we have that there
  exist $\beta_1 > 0$ and $A_1 \in \rmc(\rset_+, \rset_+)$ such that for any
  $\vareps \in \ooint{0, \bvareps}$ we have 
  \begin{equation}
    \label{eq:inter_out_param}
    \I_{\vareps}^{\mathrm{out}}(\varphi, z) \leq A_1(C_{\varphi}) \vareps^{-d/2} \{\exp[- m(z) / \vareps]+ \exp[-\beta_1/\vareps]\}\eqsp .
  \end{equation}
  Using \Cref{prop:upper_bound_i_laplace_param} there exist $c_1, \upgamma > 0$
  such that $m(z) \geq \min(c_1/\sigma_\upgamma^\star(z), c^\star(z))$ (with
  $c_1, \upgamma > 0$ that do not depend on $z$). Hence, combining this result,
  \eqref{eq:inter_out_param} and the fact that there exists $c_3$ such that for
  any $t > 0$, $\exp[-1/t] \leq c_3 t^{(d+1)/2}$, we get
  \begin{equation}
    \label{eq:inter_out_param_duo}
    \I_{\vareps}^{\mathrm{out}}(\varphi, z) \leq A_1(C_{\varphi}) \vareps^{-d/2} \{\exp[- c^\star(z) / \vareps] + (\sigma_\upbeta^\star(z))^{(d+1)/2} \vareps^{(d+1)/2}/c_3 + \exp[-\beta_1/\vareps]\}\eqsp .
  \end{equation}
  Therefore, there exist $\upbeta' > 0$ and $A_2 \in \rmc(\rset_+, \rset_+)$
  (non-increasing) such that for any $\vareps \in \ooint{0, \bvareps}$ we have
  \begin{equation}
    \label{eq:inter_out_param_attention}
    \I_{\vareps}^{\mathrm{out}}(\varphi, z) \leq A_2(C_{\varphi})  \{(1 + \sigma_{\upbeta'}^\star(z)) \vareps^{1/2} + \vareps^{-d/2} \exp[-c^\star(z)/\vareps]\}\eqsp .
  \end{equation}
  In addition, using \Cref{prop:upper_bound_i_laplace_param}, there exist
  $B_1 \geq 0$ and $\upbeta'' > 0$ such that for any $\vareps > 0$ we have
  \begin{equation}
    \label{eq:bound_I_vareps_attention}
    \absLigne{\I_\vareps^{\mathrm{in}}(\varphi, z) - \I_0(\varphi, z)} \leq B_1 \sigma_{\upbeta''}^\star(z) (1 + M_{0,\varphi}(z) + M_{1,\varphi}(z)) \vareps^{1/2} \eqsp , 
  \end{equation}
  Note that $M_{i,\varphi}(z) \leq M_{i, \varphi}$ for $i \in \{0,1\}$.  Hence,
  combining \eqref{eq:inter_out_param_attention} and
  \eqref{eq:bound_I_vareps_attention}, there exist
  $A_3 \in \rmc(\rset_+, \rset_+)$ (non-increasing) and $\upbeta_0 > 0$ such that for any
  $\vareps \in \ooint{0, \bvareps}$ we have
  \begin{equation}
    \absLigne{\I_\vareps(\varphi, z) - \I_0(\varphi, z)} \leq A_3(C_\varphi) (1 + M_{0,\varphi} + M_{1,\varphi}) \{ (1 + \sigma_{\upbeta_0}^\star(z))\vareps^{1/2} + \vareps^{-d/2} \exp[-c^\star(z)/\vareps]\} \eqsp , 
  \end{equation}
  Similar results hold if $\varphi$ is replaced by $1$ and we get that there
  exist $A_4 \geq 0$, $\upbeta_1 > 0$ such that for any
  $\vareps \in \ooint{0, \bvareps}$ we have
  \begin{align}
    &\absLigne{\I_\vareps(\varphi, z) - \I_0(\varphi, z)} \leq A_3(C_\varphi) (1 + M_{0,\varphi} + M_{1,\varphi}) \{ (1 + \sigma_{\upbeta_1}^\star(z))\vareps^{1/2} + \vareps^{-d/2} \exp[-c^\star(z)/\vareps]\} \eqsp ,  \\
        &\absLigne{\J_\vareps(z) - \J_0(z)} \leq A_4  \{ (1 + \sigma_{\upbeta_1}^\star(z))\vareps^{1/2} + \vareps^{-d/2} \exp[-c^\star(z)/\vareps]\} \eqsp .     \label{eq:summary_param}
  \end{align}
   In addition, we have that 
   \begin{equation}
     \label{eq:decompo_param}
    \abs{\updelta_z \Sker_\vareps(\varphi) - \updelta_z \Sker_0(\varphi)}\leq (\I_\vareps(\varphi,z) - \I_0(\varphi,z))/\J_\vareps(z) + \I_0(\varphi, z) (\J_0(z) - \J_\vareps(z))/(\J_0(z) \J_\vareps(z)) \eqsp .
  \end{equation}
  Finally, we have that
  $\I_0(\varphi, z) \leq M_{0, \varphi} \sigma_1^\star(z)$. Combining this
  result \eqref{eq:lower_param}, \eqref{eq:summary_param} and
  \eqref{eq:decompo_param} concludes the proof.

\end{proof}

  \subsection{Control of the moments}
\label{sec:sgld_moments}

In order to derive the uniform stability of the
limiting measure, we first need to control the moments of
$\updelta_{z^{1:n}} \Sker_\vareps$ uniformly in w.r.t. $\vareps$, $z$ and $n$. Let
$z^{1:n} \in \msz^n$ and $\vareps > 0$. We consider the Langevin diffusion
$(\bfX_t^{\vareps}(z^{1:n}))_{t \geq 0}$ given by the following Stochastic Differential
Equation (SDE): $\bfX_0^\vareps(z^{1:n}) \in \rset^d$ and  
\begin{equation}
  \rmd \bfX_t^{\vareps}(z^{1:n}) = -\nabla_x U_n(\bfX_t^{\vareps}(z^{1:n}), z^{1:n}) \rmd t + \sqrt{2 \vareps} \rmd \bfB_t \eqsp ,
\end{equation}
where $(\bfB_t)_{t \geq 0}$ is a $d$-dimensional Brownian motion with filtration
$(\mathcal{F}_t)_{t \geq 0}$. We recall that for any $n \in \nset$,
$z^{1:n} \in \msz^n$ and $x \in \rset^d$ we have
$U_n(x, z^{1:n}) = (1/n) \sum_{i=1}^n u(x, z_i)$. Therefore, under
\Cref{assum:raginsky}($n$) we have that $(\bfX_t^{\vareps}(z^{1:n}))_{t \geq 0}$ is
well-defined and admits $\updelta_{z^{1:n}} \Sker_\vareps$ as an invariant
measure, see \cite{roberts1996exponential} for instance.

\begin{lemma}
  \label{lemma:moment_bound}
  Let $n \in \nset$ and assume \tup{\rref{assum:raginsky}($n$)} and
  \tup{\rref{assum:U_n_param}($n$)}. Let $z^{1:n} \in \msz^n$ and assume that
  $\sigma_1^\star(z^{1:n}) < +\infty$. Then there exist $\bvareps > 0$ such that
  for any $k \in \nset$ there exists $C_k \geq 0$ such that for any
  $\vareps \in \coint{0, \bvareps}$
  \begin{equation}
    \textstyle{
      \int_{\rset^d} \norm{x}^{2k} \Sker_\vareps(z^{1:n}, \rmd x) \leq C_k \eqsp , 
      }
    \end{equation}
    with $C_k$ and  $\bvareps$ that do not depend on $n \in \nset$ and $z^{1:n} \in \msz^n$.
\end{lemma}

\begin{proof}
  Let $\vareps > 0$.  First, since $x \mapsto \nabla_x U_n(x, z^{1:n})$ is
  Lipschitz continuous we have that $(\bfX_t(z^{1:n})^\vareps)_{t \geq 0}$ is
  well-defined and is a continuous semi-martingale using \cite[Theorem 2.3,
  Theorem 2.4 , Chapter 4]{ikeda2014stochastic} such that for any $t \geq 0$
  \begin{equation}
    \textstyle{\bfX_t^\vareps(z^{1:n}) = \bfX_0^\vareps(z^{1:n}) - \int_0^t \nabla_x U_n(\bfX_s^\vareps(z^{1:n}), z^{1:n}) \rmd s + \sqrt{2\vareps} \bfB_t \eqsp . }
  \end{equation}
  Let $\mtt_0 > 0$. Using It\^{o}'s formula, see \cite[Theorem
  5.1]{ikeda2014stochastic} we have that for any $t \geq 0$ and $k \in \nset$
  with $k \geq 2$
  \begin{align}
    &\normLigne{\bfX_t^\vareps(z^{1:n})}^{2k} \exp[\mtt_0 t] \\
    &  = \textstyle{\bfX_0^\vareps(z^{1:n}) + 2k \int_0^t \langle \bfX_s^\vareps(z^{1:n}), \nabla_x U_n(\bfX_s^\vareps(z^{1:n}), z^{1:n}) \rangle \normLigne{\bfX_s^\vareps(z^{1:n})}^{2(k-1)} \exp[\mtt_0 s] \rmd s} \\
                                                        &  \quad  \textstyle{+ \mtt_0 \int_0^t \normLigne{\bfX_s^\vareps(z^{1:n})}^{2k} \exp[\mtt_0 s] \rmd s + 4\vareps k(2k-1) \int_0^t \normLigne{\bfX_s^\vareps(z^{1:n})}^{2(k-1)} \exp[\mtt_0 s] \rmd s + \bfM_t^\vareps(z^{1:n}) \eqsp ,}
  \end{align}
  with $(\bfM_t^\vareps(z^{1:n}))_{t \geq 0}$ a $\mathcal{F}_t$-martingale such
  that $\bfM_0^\vareps(z^{1:n}) = 0$. First, using Fubini-Tonelli theorem, we have that for any $t \geq 0$ and $k \in \nset$
  \begin{align}
    &\textstyle{\expeLigne{\normLigne{\bfX_t^\vareps(z^{1:n})}^{2k}} \exp[\mtt_0 t] \leq \bfX_0^\vareps(z^{1:n}) + \abs{2k\mtt -\mtt_0} \int_0^t \expeLigne{\normLigne{\bfX_s^\vareps(z^{1:n})}^{2k}} \exp[\mtt_0 s]  \rmd s } \\
    & \qquad \textstyle{+ (4\vareps k(2k-1) +2k\ctt) (1 + \int_0^t \expeLigne{\normLigne{\bfX_s^\vareps(z^{1:n})}^{2k}} \exp[\mtt_0 s] \rmd s) \eqsp .  }%+ 4\vareps k(2k-1)/\mtt_0 \exp[\mtt_0 t]  \eqsp . }
  \end{align}
  Using Gr\"{o}nwall's lemma we have that for any $t \geq 0$ and $k \in \nset$,
  $\expeLigne{\normLigne{\bfX_t^\vareps(z^{1:n})}^{2k}} < +\infty$.  Hence,
  using this result, the Fubini theorem and \rref{assum:raginsky}($n$), we have
  that for any $t \geq 0$ and $k \in \nset$
  \begin{align}
    &\textstyle{\expeLigne{\normLigne{\bfX_t^\vareps(z^{1:n})}^{2k}} \exp[\mtt_0 t] \leq \bfX_0^\vareps(z^{1:n}) - (2k\mtt -\mtt_0) \int_0^t \expeLigne{\normLigne{\bfX_s^\vareps(z^{1:n})}^{2k}} \exp[\mtt_0 s]  \rmd s } \\
    & \qquad \textstyle{+ (4\vareps k(2k-1) +2k\ctt) \int_0^t \expeLigne{\normLigne{\bfX_s^\vareps(z^{1:n})}^{2(k-1)}} \exp[\mtt_0 s] \rmd s \eqsp .}%+ 4\vareps k(2k-1)/\mtt_0 \exp[\mtt_0 t]  \eqsp . }
  \end{align}
  Combining this result and the fact that for any $t, a, b > 0$ and
  $k \in \nset$, $a t^{2(k-1)} \leq b t^{2k} + a^k/b^{k-1}$, we have for any $t \geq 0$ and $k \in \nset$
  \begin{align}
    \label{eq:allez_quasi}
    \expeLigne{\normLigne{\bfX_t^\vareps(z^{1:n})}^{2k}} \exp[\mtt_0 t] &\textstyle{\leq \bfX_0^\vareps(z^{1:n}) - (k\mtt -\mtt_0) \int_0^t \expeLigne{\normLigne{\bfX_s^\vareps(z^{1:n})}^{2k}} \exp[\mtt_0 s]  \rmd s } \\
    % & \quad + \{(4\vareps k(2k-1) +2k\ctt)^k/(k\mtt)^{k-1} +  4\vareps k(2k-1) \}/\mtt_0\exp[\mtt_0 t] \eqsp .
        & \quad + (4\vareps k(2k-1) +2k\ctt)^k/((k\mtt)^{k-1}\mtt_0)\exp[\mtt_0 t] \eqsp . 
  \end{align}
  Therefore, for any $k \in \nset$, there exists $C_k \geq 0$ such that for any
  $t \geq 0$, $\expeLigne{\normLigne{\bfX_t^\vareps(z^{1:n})}^{2k}} \leq C_k$
  upon letting $\mtt_0 = k\mtt$ in \eqref{eq:allez_quasi} with $C_k$ that does not depend on $\vareps$ and $z^{1:n}$. Therefore, using that the sequence of
  distributions associated with $(\bfX_t^\vareps(z^{1:n}))_{t \geq 0}$ weakly
  converges towards $\updelta_{z^{1:n}} \Sker_\vareps$, see
  \cite{roberts1996exponential} for instance, and the monotone convergence theorem we get that
  $\int_{\rset^d} \norm{x}^{2k} \Sker_\vareps(z^{1:n}, \rmd x) \leq C_k$. We
  conclude upon using \Cref{prop:conclusion_param_bounds} in the case where
  $\vareps = 0$.
\end{proof}

Finally, we will make use of the following lemma.

\begin{lemma}
  \label{lemma:wasserstein_2_bound}
  Let $\mu, \nu \in \Pens(\rset^d)$  such that for any $k \in \nset$
  there exists $C_k \geq 0$ such that
  $\int_{\rset^d} \normLigne{x}^{2k} \rmd \mu(x) + \int_{\rset^d} \normLigne{x}^{2k}
  \rmd \nu(x) \leq C_k$. Then, for any $k \in \nset$ we have that
  \begin{equation}
   \wassersteinD[2](\mu, \nu)\leq \sqrt{2 C_k}
  \wassersteinD[1](\mu, \nu)^{(k-1)/(2k-1)} \eqsp . 
  \end{equation}
\end{lemma}

\begin{proof}
  Let $k \in \nset$ and $(X,Y)$ be the optimal coupling between $\mu$ and $\nu$
  w.r.t. the Wasserstein distance of order one. Let $Z = X- Y$. Using
  H\"{o}lder's inequality we have 
  \begin{equation}
    \textstyle{\expeLigne{\normLigne{Z}^2} \leq \expeLigne{\normLigne{Z}^{2k}}^{1/(2k-1)} \expeLigne{\normLigne{Z}}^{(2k-2)/(2k-1)} \leq \expeLigne{\normLigne{Z}^{2k}}^{1/(2k-1)} \wassersteinD[1](\mu, \nu)^{(2k-2)/(2k-1)} \eqsp . }
  \end{equation}
  We conclude upon combining this result and the fact that
  $\expeLigne{\normLigne{Z}^{2k}}^{1/(2k-1)} \leq 2 C_k$.
\end{proof}

%%% Local Variables:
%%% mode: latex
%%% TeX-master: "main"
%%% End:

\section{Basics on flows and geometric measure theory}
\label{sec:basics-geom-meas}

In this section, we recall basic facts from geometric measure theory. We refer to
\cite{federer1969geometric,morgan2016geometric,ambrosio2000functions} for a
complete exposition of geometric measure theory concepts. We begin with a
proposition establishing the existence of flows. Then, we prove useful facts on
rectifiable sets. Finally, we state area and coarea formulas which are central to
our analysis.

\begin{lemma}
  \label{lemma:flow}
  Let $k \in \nsets$  and
  $X \in \rmc^k(\rset^d, \rset^d)$. Assume that $X$ is compactly supported. Then, there
  exists a unique mapping $\Phi \in \rmc^{k+1, k}(\rset \times \rset^d, \rset^d)$ such that for
  any $x \in \rset^d$, $\Phi(0, x) = x$ and $\partial_s \Phi(s,\rset^d) = X(\Phi(s,\rset^d))$.  In
  addition, for any $s \in \rset$, $x \mapsto \Phi(s, x)$ is a diffeomorphism.
\end{lemma}

\begin{proof}
  The first part of the proposition is an application of \cite[Theorem 2.5,
  Theorem 2.6]{lang2002introduction}. The second part is an application of
  \cite[Lemma 2.4]{MR0163331}.
\end{proof}

Note that the previous lemma can be extended to smooth manifolds.

\begin{definition}{\cite[Definition 2.57]{ambrosio2000functions}}
  Let $\mse \subset \rset^d$ and $k \in \nset$ with $k \leq d$. $\mse$ is
  countably $\calH^k$-rectifiable if there exists $(\psi_i)_{i \in \nset}$ such
  that for any $i \in \nset$, $\psi_i: \ \rset^k \to \rset^d$ is Lipschitz
  continuous and
  \begin{equation}
\calH^k\parenthese{\mse \backslash \cup_{i \in \nset} \psi_i(\rset^k)} = 0 \eqsp . 
  \end{equation}
\end{definition}
In what follows, we provide an easy criterion to verify if a given level set is
a countably $\calH^k$-rectifiable set. We start by recalling the following
proposition.

\begin{proposition}{\cite[Theorem 10.5]{lang2007notes}}
  \label{thm:rectifiable_almost}
  Let $k \in \nset$ with $k \leq d$ and $F: \ \rset^d \to \rset^k$ be Lipschitz continuous. Then
  for $\Leb$-almost every $y \in \rset^k$, $F^{-1}(\{y\})$ is countably
  $\calH^{d-k}$-rectifiable.
\end{proposition}

The following lemma ensures the stability of countably
$\calH^k$-rectifiable sets through diffeomorphisms.

\begin{lemma}
  \label{prop:diffeo_recti}
  Let $k \in \nset$ with $k \leq d$, $\mse \subset \rset^d$ be a compact
  countably $\calH^k$-rectifiable set and $\Phi \in \rmc^1(\rset^d, \rset^d)$ be
  a diffeomorphism. Then $\Phi(\mse)$ is countably $\calH^k$-rectifiable.
\end{lemma}

\begin{proof}
  Since $\mse \subset \rset^d$ is countably $\calH^k$-rectifiable there exists
  $(\psi_i)_{i \in \nset}$ such that for any $i \in \nset$,
  $\psi_i: \ \rset^k \to \rset^d$ is Lipschitz continuous and
  \begin{equation}
\calH^k\parenthese{\mse \backslash \cup_{i \in \nset} \psi_i(\rset^k)} = 0 \eqsp . 
  \end{equation}
  Let $y \in \Phi(\mse) \backslash \cup_{i\in \nset} (\Phi \circ
  \psi_i)(\rset^k)$. Then there exists $x \in \mse$ such that $y = \Phi(x)$ and for
  any $i \in \nset$ and $y_i \in \psi_i(\rset^k)$, $\Phi(x) \neq \Phi(y_i)$, \ie \
  $x \neq y_i$. Therefore
  $\Phi(\mse) \backslash \cup_{i\in \nset} (\Phi \circ \psi_i)(\rset^k) \subset
  \Phi(\mse \backslash \cup_{i\in \nset} \psi_i(\rset^k))$. $\Phi$ is
  Lipschitz-continuous on $\mse$ with constant $\Lip_\mse \geq 0$ and therefore
  using \cite[Proposition 2.49 (iv)]{ambrosio2000functions} we get
  \begin{equation}
    0 \leq \calH^k\parenthese{\Phi(\mse) \backslash \cup_{i\in \nset} (\Phi \circ \psi_i)(\rset^k))} \leq \Lip_\mse^k \calH^k\parenthese{\mse \backslash \cup_{i\in \nset} \psi_i(\rset^k)} \leq 0 \eqsp ,
  \end{equation}
  which concludes the proof.
\end{proof}

Finally, we show that if a function is regular enough then its level-sets are
countably $\calH^k$-rectifiable.

\begin{lemma}
  \label{prop:rectifiable_level_set}
  Let $F \in \rmc^1(\rset^d, \rset^p)$ with $ d \geq p$ and
  $\msa \subset \rset^p$ compact such that $F^{-1}(\msa)$ is compact and
  $0 \in \interior(\msa)$. In addition, assume that for any
  $x \in F^{-1}(\{0\})$, $\absLigne{\jac{x}} > 0$. Then, there exists $\eta > 0$
  such that for any $t \in \cballinfty{0}{\eta}$, $F^{-1}(t)$ is countably
  $\calH^{d-p}$-rectifiable and $\calH^{d-p}(F^{-1}(t)) < +\infty$.
\end{lemma}

\begin{proof}
  In this proof, we show that the level sets $F^{-1}(t)$ and $F^{-1}(s)$ are
  diffeomorphic for $t, s \in \cball{0}{\eta}$. Then, we conclude upon combining
  \Cref{thm:rectifiable_almost} and \Cref{prop:diffeo_recti}. First, note that
  $F^{-1}(0)$ is compact since it is closed in $F^{-1}(\msa)$ which is
  compact. There exists $\eta_0 > 0$ such that for any
  $x \in F^{-1}(\cballinfty{0}{\eta_0})$, $\abs{\jac{x}} > 0$. Indeed, if this
  is not the case then there exists $(x_k)_{k \in \nset}$ with
  $\lim_{k \to +\infty} F(x_k) = 0$ and $\abs{\jac{x_k}} =0$. Since
  $0 \in \interior(\msa)$ there exists $\eta' > 0$ such that
  $F^{-1}(\cball{0}{\eta'})$ is compact and therefore there exists $x^\star$
  such that, up to taking a subsequence, $\lim_{k \to +\infty} x_k = x^\star$. Then
  $F(x^\star)=0$ and $\jac{x^\star}=0$, which is absurd.  We define
  $\msk_0 = F^{-1}(\cballinfty{0}{\eta_0})$ and
  $\msk_1 = F^{-1}(\cballinfty{0}{\eta})$ with $\eta = \min(\eta_0/2,
  \eta')$. Note that $\msk_1 \subset \interior(\msk_0)$.
    
  Let $G(x) = \rmD F(x) \rmD F(x)^\top$. Note that for any $x \in \msk_0$,
  $G(x)$ is invertible.  We define $\{f_i\}_{i=1}^p$ such that for
  any $i \in \{1, \dots, p\}$, $f_i:\ \rset^d \to \rset^d$ with for any
  $x \in \msk_0$
  \begin{equation}
    \textstyle{
      f_i(x) = \sum_{k=1}^p h_{i, k}(x) \nabla F_k(x) \eqsp ,
      }
    \end{equation}
    with $\{h_{i,j}(x)\}_{1 \leq i,j \leq p} = G^{-1}(x)$. In addition, we
    assume that $f_i(x) = 0$ for any $x \notin \msk_1$.  For any $x \in \msk_1$
    and $i, j \in \{1, \dots, p\}$ we have
    \begin{equation}
      \textstyle{
        \langle f_i(x), \nabla F_j(x) \rangle = \sum_{k=1}^p h_{i,k}(x) \langle \nabla F_k(x), \nabla F_j(x) \rangle = \updelta_i(j) \eqsp .
        }
    \end{equation}
    In what follows, we let $\{g_i\}_{i=1}^p$ such that for any
    $i \in \{1, \dots, p\}$, $g_i: \ \rset^d \to \rset^d$ and
    $g_i \in \rmc^1(\rset^d, \rset^d)$ such that for any $x \in \msk_1$,
    $g_i(x) = f_i(x)$ and for any $x \notin \interior(\msk_1)$, $g_i(x)=0$, such
    functions exist using Whitney extension theorem for instance, see
    \cite{whitney1934analytic}. In what follows, we fix
    $t = (t_0, \dots, t_p) \in \cballinfty{0}{\eta}$. For any
    $i \in \{1, \dots, p\}$ let $\Phi_i: \ \rset \times \rset^d \to \rset^d$
    given by $\Phi_i(0, \cdot) = \Id$ and for any $s \in \rset$ and
    $x \in \rset^d$
    \begin{equation}
      \partial_s \Phi_i(s,x) = -g_i(\Phi_i(s,x)) \eqsp . 
    \end{equation}
    For any $i \in \{1, \dots, p\}$, $\Phi_i$ is well-defined using
    \Cref{lemma:flow}.  Therefore, we have for any $i \in \{1, \dots, p\}$ and
    $s \in \rset$, $x \in \rset^d$ such that $\Phi(s,x) \in \msk_1$
    \begin{equation}
      \partial_s F(\Phi_i(s,x)) = -(\langle g_i(\Phi_i(s,x)), \nabla F_1(\Phi_i(s,x)) \rangle, \dots, \langle g_i(\Phi_i(s,x)), \nabla F_p(\Phi_i(s,x)) \rangle) = -e_i \eqsp , 
    \end{equation}
    where we recall that $\{e_i\}_{i=1}^p$ is the canonical basis of
    $\rset^p$. We define $\bar{\Phi}_t: \ \rset^d \to \rset^d$ such that for any
    $x \in \rset^d$, $\bar{\Phi}_t(x) = x^{(p)}$ with $x^{(0)} = x$ and for any
    $i \in \{0, \dots, p-1\}$, $x^{(i+1)} = \Phi_{i+1}(t_{i+1}, x^{(i)})$. Note that
    $\bar{\Phi}_t \in \rmc^1(\rset^d, \rset)$ is a diffeomorphism using
    \Cref{lemma:flow}. Using \Cref{thm:rectifiable_almost}, there exists
    $t_0 \in \cballinfty{0, \eta}$ such that $F^{-1}(t_0)$ is countably
    $\calH^{d-p}$ rectifiable. Let $s \in \cballinfty{0}{\eta}$. Using
    \Cref{prop:diffeo_recti} and that $\bar{\Phi}_{-s} \circ \bar{\Phi}_{t_0}$
    is a diffeomorphism between $F^{-1}(t_0)$ and $F^{-1}(s)$ we get that
    $F^{-1}(s)$ is countably $\calH^{d-p}$ rectifiable, which concludes the
    first part of the proof.

    For the second part of the proof, we let $R \geq 0$ such that
    $F^{-1}(\msa) \subset \cball{0}{R}$. Since $\cball{0}{R}$ is $\calH^d$-rectifiable
    we obtain using the coarea formula \Cref{thm:area_coarea}
    \begin{equation}
      \textstyle{
        \int_{\cball{0}{R}} \absLigne{\jac{x}} \rmd x   = \int_{\rset^p} \calH^{d-p}(\cball{0}{R} \cap F^{-1}(t)) \rmd t \geq \int_{\cballinfty{0}{\eta}} \calH^{d-p}(F^{-1}(t)) \rmd t \eqsp .
        }
  \end{equation}
  Therefore, there exists $t_0 \in \cballinfty{0}{\eta}$ such that
  $\calH^{d-p}(F^{-1}(t_0))<+\infty$. Let
  $\Psi = \bar{\Phi}_{-t} \circ \bar{\Phi}_{t_0}$ and $\Lip \geq 0$ such that
  for any $x \in F^{-1}(t_0)$ we have $\norm{\rmd \Psi(x)} \leq \Lip$. Then,
  using \cite[Proposition 2.49 (iv)]{ambrosio2000functions} we have
  \begin{equation}
    \calH^{d-p}(F^{-1}(t)) =  \calH^{d-p}(\Psi(F^{-1}(t_0)))\leq \Lip^{d-p} \calH^{d-p}(F^{-1}(t_0)) < +\infty \eqsp ,
  \end{equation}
  which concludes the proof.
\end{proof}

We conclude this section with the area and coarea formulae. 

\begin{theorem}
  \label{thm:area_coarea}
  Let $F: \rmc^1(\rset^d, \rset^p)$ be Lipschitz continuous, $k \in \nset$ with
  $k \leq d$ and $\mse \subset \rset^d$ be a countably $\calH^k$-rectifiable
  set. Let $\varphi : \ \rset^d \to \rset$ measurable such that
  \begin{equation}
    \textstyle{
      \int_{\mse} \abs{\varphi(x)} \jac{x} \rmd \calH^k(x) < +\infty \eqsp ,
      }
  \end{equation}
or assume that $\varphi: \ \rset^d \to \coint{0, +\infty}$. Then, the following hold:
\begin{itemize}
\item (Area formula) If $d \leq p$ then
  \begin{equation}
    \textstyle{
    \int_{\mse} \varphi(x) \jac{x} \rmd \calH^{k}(x) = \int_{\rset^p} \parenthese{\int_{\mse \cap F^{-1}(y)} \varphi(x) \rmd \calH^0(x)} \rmd \calH^k(y) \eqsp . }
  \end{equation}
\item (Coarea formula) If $k \geq p$ then
  \begin{equation}
    \textstyle{
      \int_{\mse} \varphi(x) \jac{x} \rmd \calH^{k}(x) = \int_{\rset^p} \parenthese{\int_{\mse \cap F^{-1}(y)} \varphi(x) \rmd \calH^{k-p}(x)} \rmd \calH^p(y) \eqsp .
      }
  \end{equation}
\end{itemize}
\end{theorem}

\begin{proof}
  These results follow from \cite[Theorem 2.91, Theorem
  2.93]{ambrosio2000functions} combined with \cite[Exercise
  2.12]{ambrosio2000functions}.
\end{proof}

In particular, note that if $F\in \rmc^1(\rset^d, \rset^d)$ is a Lipschitz
diffeomorphism we have using \Cref{thm:area_coarea} that for any
$\varphi : \ \rset^d \to \rset$ measurable such that
\begin{equation}
  \textstyle{
    \int_{\mse} \abs{\varphi(x)} \abs{J \Phi(x)} \rmd \calH^k(x) < +\infty \eqsp ,
    }
  \end{equation}
  the following change of variable formula with respect to $\calH^k$ holds
  \begin{equation}
    \label{eq:change_hausdorff}
    \textstyle{
  \int_{\mse} \varphi(x) \jac{x} \rmd \calH^{k}(x) = \int_{\rset^d} \parenthese{\int_{\mse \cap F^{-1}(y)} \varphi(x) \rmd \calH^0(x)} \rmd \calH^k(y) =  \int_{F(\mse)} \varphi(F^{-1}(y)) \rmd \calH^{k}(y) \eqsp .}
\end{equation}
Note that in order for \eqref{eq:change_hausdorff} to hold, $\mse$ needs to be
countably $\calH^k$-rectifiable.

%%% Local Variables:
%%% mode: latex
%%% TeX-master: "main"
%%% End:

\end{document}